\documentclass[11pt,a4paper,reqno]{amsart}
\usepackage[a4paper,left=2.0cm,right=2.0cm,top=2.3cm,bottom=2.3cm]{geometry}
\usepackage{mathrsfs,amssymb,amsmath,amsthm,graphicx,mathtools,xfrac}
\usepackage{upgreek}
\usepackage[amsstyle]{cases}
\usepackage{enumitem}
\usepackage{array}
\usepackage{aliascnt}
\usepackage[linkcolor=blue,citecolor=cyan,urlcolor=blue,colorlinks,hypertexnames=false]{hyperref}
\usepackage[noabbrev]{cleveref}
\usepackage[utf8]{inputenc}
\usepackage[T1]{fontenc}
\usepackage{xcolor}
\usepackage[integrals]{wasysym}
\usepackage{subcaption}

\newtheorem{theorem}{Theorem}[section]
\newaliascnt{lemma}{theorem}
\newtheorem{lemma}[lemma]{Lemma}
\aliascntresetthe{lemma}
\newaliascnt{definition}{theorem}
\newtheorem{definition}[definition]{Definition}
\aliascntresetthe{definition}
\newaliascnt{proposition}{theorem}
\newtheorem{proposition}[proposition]{Proposition}
\aliascntresetthe{proposition}
\newaliascnt{remark}{theorem}
\newtheorem{remark}[remark]{Remark}
\aliascntresetthe{remark}
\newaliascnt{corollary}{theorem}

\aliascntresetthe{corollary}
\newaliascnt{assumption}{theorem}

\aliascntresetthe{assumption}

\crefname{theorem}{Theorem}{Theorems}
\Crefname{theorem}{Theorem}{Theorems}

\crefname{proposition}{Proposition}{Propositions}
\Crefname{proposition}{Proposition}{Propositions}

\crefname{lemma}{Lemma}{Lemmas}
\Crefname{lemma}{Lemma}{Lemmas}

\crefname{corollary}{Corollary}{Corollaries}
\Crefname{corollary}{Corollary}{Corollaries}

\crefname{assumption}{Assumption}{Assumptions}
\Crefname{assumption}{Assumption}{Assumptions}

\crefname{definition}{Definition}{Definitions}
\Crefname{definition}{Definition}{Definitions}

\crefname{remark}{Remark}{Remarks}
\Crefname{remark}{Remark}{Remarks}

\newcommand{\Om}{\Omega}
\newcommand{\Ga}{\Gamma}
\newcommand{\R}{\mathbb R}
\newcommand{\diver}{\operatorname{div}}
\newcommand{\curl}{\operatorname{curl}}

\newcommand{\X}{\mathcal X}
\newcommand{\D}{\mathcal D}

\numberwithin{equation}{section}
\allowdisplaybreaks
\begin{document}
\title[Hadamard theory for free-boundary ideal MHD]{Hadamard well-posedness and nonuniform dependence for the incompressible free-boundary ideal MHD equations}
\author{Chengchun Hao \and Siqi Yang}
\address{Chengchun Hao\newline
State Key Laboratory of Mathematical Sciences, Academy of Mathematics and Systems Science, Chinese Academy of Sciences, Beijing 100190, China \newline
\and
School of Mathematical Sciences, University of Chinese Academy of Sciences, Beijing 100049, China
}
\email{hcc@amss.ac.cn}
\address{Siqi Yang\newline School of Mathematics, Statistics and Mechanics, Beijing University of Technology, Beijing 100124, China}
\email{siqiyang@bjut.edu.cn}

\begin{abstract}
We prove local Hadamard well-posedness for the three-dimensional
incompressible free-boundary ideal magnetohydrodynamic equations
without surface tension under the Taylor sign condition.
For every $5/2<s<4$, the result holds on bounded fluid domains
whose closures are diffeomorphic to the closed ball, with the
velocity, magnetic field, and boundary in the standard
$H^s$ spaces. No additional regularity in the magnetic direction
is required. The proof rests on a fourth order geometric energy
inequality whose coefficients are controlled by these $H^s$ norms,
the fixed geometric bounds, and the Taylor lower bound.
A harmonic extension of the boundary normal provides the
coefficient control needed at this regularity.
We also prove that, near every admissible initial state and at
each positive time in a common local existence interval, the
Eulerian solution map is not uniformly continuous on any
nonempty open subset of a sufficiently small data neighborhood. The construction combines
localized magnetic oscillations prescribed at the observation
time with time reversal and Galilean translations.
\end{abstract}

\maketitle
\enlargethispage{2pt}
\tableofcontents
\section{Introduction}

We consider the incompressible free-boundary ideal magnetohydrodynamic (MHD)
equations in three space dimensions:
\begin{subnumcases}{\label{eq:MHD}}
D_t u+\nabla p=B\cdot\nabla B,
&\text{in }\Om_t,\label{eq:mom}\\
D_t B=B\cdot\nabla u,
&\text{in }\Om_t,\\
\diver u=0,\quad \diver B=0,
&\text{in }\Om_t,\\
V_{\Ga_t}=u\cdot N,\quad B\cdot N=0,\quad p=0,
&\text{on }\Ga_t.\label{eq:bc}
\end{subnumcases}
Here $\Om_t\subset\R^3$ is the bounded fluid region, $\Ga_t=\partial\Om_t$ is the moving boundary, $u$ is the velocity, and $B$ is the magnetic field.  The total pressure is
\begin{equation*}
p=q+\frac12|B|^2,
\end{equation*}
where $q$ is the fluid pressure.  The material derivative is
\begin{equation*}
D_t=\partial_t+u\cdot\nabla.
\end{equation*}
On $\Ga_t$, $N$ denotes the outward unit normal and $V_{\Ga_t}$ the normal velocity of the free boundary. We assume the Taylor sign condition
\begin{equation*}
-\partial_Np_0\ge c_0>0
\qquad\text{on }\Ga_0,
\end{equation*}
where $\partial_N=N\cdot\nabla$, $p_0=p|_{t=0}$, and $c_0$ is a
positive constant.

\subsection{Previous work}

\subsubsection{The incompressible Euler equations}
For $B=0$, problem \eqref{eq:MHD} reduces to the free-boundary Euler equations.
Under the Taylor sign condition, Christodoulou and Lindblad
\cite{ChristodoulouLindblad2000} established geometric a priori
estimates, and Lindblad \cite{Lindblad2005} and Coutand and Shkoller
\cite{CoutandShkoller2007} proved local well-posedness. The geometric
approach was further developed by Shatah and Zeng
\cite{ShatahZeng2008}. For irrotational flows, the paradifferential
theory of gravity water waves was developed in
\cite{AlazardBurqZuily2014}; see also \cite{Lannes2013}.
For rotational flows, Ifrim, Pineau, Tataru, and Taylor
\cite{IPTT-Euler} established sharp local Hadamard well-posedness
in the standard Sobolev spaces.

A related line of work concerns the local uniform continuity and
differentiability of the Eulerian solution map. On the torus and in the whole space, Himonas and Misio\l ek
\cite{HimonasMisiolek2010} proved nonuniform dependence on initial
data. In spatial dimension $n$, for $s>n/2+1$, \.Inci \cite{Inci2015} showed that the
solution map on $H^s(\R^n)$ is nowhere locally uniformly continuous.
On bounded two-dimensional domains, \.Inci and Li \cite{InciLi2019} proved that,
at each positive time, the solution map is nowhere locally uniformly
continuous and nowhere Fr\'echet differentiable. Bourgain and Li
\cite{BourgainLi2019} obtained nonuniform dependence by localized
Galilean transformations. For a moving boundary, the corresponding
question also requires comparing fields on different fluid domains.

\subsubsection{The incompressible free-boundary ideal MHD equations}
For the one-phase problem with surface tension, Luo and Zhang
\cite{LuoZhang2021SurfaceTension} established a priori estimates,
and Gu, Luo, and Zhang \cite{GuLuoZhang2024SurfaceTension} proved
local well-posedness. The first and second authors \cite{HaoYang2025}
obtained a priori estimates and a blow-up criterion for a closed free
surface, while Luo and the second author \cite{LuoYang2026} derived
low-regularity estimates and studied finite-time breakdown and
self-intersection singularities. The zero-surface-tension limit under
the Taylor sign condition was studied in
\cite{GuLuoZhang2022ZeroSurfaceTension}.

Surface tension also enters related problems involving plasma--vacuum
interfaces or current--vortex sheets. For the plasma--vacuum problem,
Liu and Xin \cite{LiuXin2025} proved nonlinear local well-posedness
without a graph assumption, under either surface tension or magnetic
noncollinearity, and established zero-surface-tension limits under
the Taylor sign or noncollinearity condition. In two dimensions,
Liu and Luo \cite{LiuLuo2026} obtained local well-posedness under
surface tension or nondegeneracy of the tangential magnetic fields,
together with zero-surface-tension limits. For the distinct two-plasma
current--vortex sheet problem, Liu and Xin \cite{LiuXin2025AdvMath}
proved local well-posedness in standard Sobolev spaces under surface
tension or the Syrovatskij stability condition, without a graph
assumption, and established the zero-surface-tension limit under
the Syrovatskij stability condition.

In the absence of surface tension, the stability conditions depend
on the magnetic configuration and the pressure at the interface.
For plasma--vacuum interfaces with a nontrivial vacuum magnetic field,
the first author \cite{Hao2017} obtained geometric a priori estimates
under a Taylor-type sign condition on the pressure jump. Morando,
Trakhinin, and Trebeschi \cite{MorandoTrakhininTrebeschi2014} studied
the linearized problem under magnetic noncollinearity, and Sun, Wang,
and Zhang \cite{SunWangZhang2019} established nonlinear well-posedness
under magnetic noncollinearity. Zhao \cite{Zhao2024} proved local
well-posedness in Eulerian coordinates under the Taylor sign
condition when the vacuum magnetic
field vanishes.

For the one-phase free-boundary problem without surface tension,
the first author and Luo \cite{HaoLuo2014} established geometric
a priori estimates under the Taylor sign condition, with separate
constant boundary conditions on the fluid pressure and the
magnetic-field strength. They subsequently constructed
two-dimensional examples of ill-posedness for the Lagrangian
flow map when the Taylor sign is violated \cite{HaoLuo2020},
and proved well-posedness of the linearized problem under the
Taylor sign condition \cite{HaoLuo2021}. Lee \cite{Lee2017}
constructed inviscid solutions by letting equal viscosity and
magnetic diffusivity tend to zero, with
the magnetic field vanishing in the vacuum and on the plasma boundary.
Gu and Wang \cite{GuWang2019} proved nonlinear local well-posedness
in an $H^4$ energy class on an initial periodic slab. At lower
regularity, Luo and Zhang \cite{LuoZhang2020} derived Lagrangian
a priori estimates for velocity and magnetic field in
$H^{5/2+\delta}$, $0<\delta<1/2$, with the flow map in
$H^{3+\delta}$, under a smallness assumption on the volume and a
constant magnetic field magnitude at the free boundary.

For the one-phase problem \eqref{eq:MHD}, Ifrim, Pineau, Tataru,
and Taylor \cite{IPTT-MHD} established local Hadamard
well-posedness in three dimensions for every $s>5/2$ on bounded
domains of general topology. Their phase space requires
$D_Bu,D_BB\in H^{s-1/2}$, where $D_B=B\cdot\nabla$, in addition
to $H^s$ regularity of the velocity, magnetic field, and free
boundary. Standard $H^s$ data, however, generally provide only
$H^{s-1}$ control of these magnetic derivatives.

We establish local Hadamard well-posedness
for \eqref{eq:MHD} in the standard Eulerian moving-domain $H^s$ topology
for every $5/2<s<4$. On bounded fluid domains whose closures are
diffeomorphic to the closed ball, the result gives existence, uniqueness,
and continuous dependence on a common lifespan under the Taylor sign
condition.
We further prove that, at every positive time in a common local existence
interval, the Eulerian solution map is not uniformly continuous on any
nonempty open subset of a local data neighborhood, with distances measured
in a fixed Eulerian metric.

\subsection{Main results}

We first present the main results informally in the Eulerian moving-domain
$H^s$ topology, with the precise definition of this topology and the full
statements given in \cref{sec:phase-spaces}. Admissibility is specified in
\cref{thm:main}.

\begin{theorem} 
Let $5/2<s<4$, and consider \eqref{eq:MHD} on a bounded
fluid domain whose closure is diffeomorphic to the closed ball.
Then the following hold:
\begin{enumerate}[label=(\roman*),leftmargin=2.6em]
\item For admissible $H^s$ initial data, \eqref{eq:MHD} is locally
Hadamard well-posed in the standard Eulerian moving-domain
$H^s$ topology.
\item For every positive time in a common local existence interval,
the solution map is not uniformly continuous on any nonempty
open subset of a sufficiently small data neighborhood, with
respect to the fixed Eulerian $H^s$ metric.
\end{enumerate}
\end{theorem}

\begin{remark}
Part~(i) gives continuous dependence at each initial datum, whereas
part~(ii) rules out a common modulus of continuity on any nonempty
open subset of the local data neighborhood. 
\end{remark}
\subsubsection{Proof strategy}\mbox{}\par
\vspace{2mm}
For part~(i), we use the geometric energy introduced by the first author
in \cite[Section~5]{Hao2017}. For smooth solutions, we prove a fourth
order energy inequality whose coefficients depend only on the $H^s$
norms of the velocity, magnetic field, and boundary, the fixed geometric
bounds, and the positive lower bound for $-\partial_Np$.
Here the boundary regularity is measured through local graph functions.

For the coupled first order equations, the leading magnetic transport
terms cancel between the velocity and magnetic energies, and between
their curl energies, by the divergence-free conditions and magnetic tangency. The
remaining coefficient estimates must account for the roughness of the
boundary: for $5/2<s\le3$, bounded $H^s$ norms allow arbitrarily large
curvature. Let $F$ be the componentwise harmonic extension of $N$
from $\Ga_t$ into $\Om_t$, and let $\Pi=\mathrm{Id}-N\otimes N$ be the tangential
projection. The identities 
\begin{equation*}
D_tN=-\Pi(\nabla u)^TN,
\qquad
D_BN=-\Pi(\nabla B)^TN,
\end{equation*} 
together with elliptic estimates for the commuted harmonic equations, give
$D_tF,D_BF\in H^{s-1}(\Om_t)$ throughout $5/2<s<4$.
Normalization and a cutoff give a nonnegative bulk quadratic form
$\gamma$, equal to $\Pi$ on the boundary and to $\mathrm{Id}$ outside a boundary
neighborhood, with 
\begin{equation*}
\nabla\gamma\in L^6(\Om_t),
\qquad
D_t\gamma,D_B\gamma\in L^\infty(\Om_t),
\end{equation*} 
controlled by the $H^s$ low norm and the fixed geometric bounds.
The $L^6$ bound is sufficient for the bulk pressure term because
the elliptic estimate for $p$ in $H^{9/2}(\Om_t)$ controls
$\nabla^4p$ in $L^3(\Om_t)$ linearly by the fourth order state norm.

The highest order pressure terms cancel between the bulk and
boundary energies. The remaining boundary term contains
$\Pi^{\otimes4}\nabla^4D_tp$, where $\Pi^{\otimes4}$ projects each
of the four derivative indices onto the tangent plane.
To estimate this trace, we use the zero boundary values of $p$ and $D_tp$.
Four differentiations in surface coordinates
express the projected pressure derivatives in terms of derivatives
of the boundary graph and lower order pressure derivatives. The
Taylor sign then recovers the fourth derivatives of the graph from
the boundary energy. Norms already covered by the low regularity
estimates are bounded directly; higher intermediate norms are obtained
by interpolation with the fourth order bounds. The powers of the high
norm that must be absorbed are strictly below one throughout
$5/2<s<4$. This closes the fourth order energy inequality.

The construction of rough solutions starts from the compatible smooth
approximations in \cref{lem:compatible-regularization}, the smooth existence
theorem \cite[Theorem~1.9]{IPTT-MHD}, and the weak stability estimate
\cite[Theorem~4.1]{IPTT-MHD}, which controls differences at the $L^2$ level. A further comparison is needed because consecutive
solutions occupy different domains. We construct one family of
inward regularizers from Piola transforms,
with a kernel independent of the domain and the largest approximation
index. These regularizers preserve divergence and are defined
slightly beyond each moving boundary. Moment cancellations give
approximation and consecutive-scale estimates.
A frequency envelope records the dyadic $H^s$ size of the initial
data in a square-summable sequence. We retain the decay between distinct
dyadic frequencies
and combine the weak differences and fourth order bounds in
$\ell^2$. The strict inequality $s<4$ supplies the high-frequency
decay needed for this summation.

After choosing the low norm bound, the bootstrap in \cref{prop:E4}
reduces the common lifespan to the control of one smooth coarse
approximation and a small frequency tail. The tail estimates in
\cref{lem:telescoping} give strong convergence and continuous dependence,
following the frequency-envelope method of
\cite{AlazardBurqIfrimTataruZuily}.
The uniqueness criterion \cite[Theorem~1.6]{IPTT-MHD} applies to the
same solution class, as shown in \cref{prop:persistence}\textup{(ii)}.

For part~(ii), after shrinking the initial neighborhood and the common
time interval, time reversal and the local Hadamard theory give a
continuous local backward map near each terminal state. Fix a positive
time in this interval and an initial state in any prescribed nonempty
open subset of the data neighborhood. At its terminal state, we prescribe
oscillatory magnetic perturbations as curls supported in the fluid interior.
They preserve zero divergence and boundary tangency, have a fixed
small $H^s$ size, and converge to zero in $H^{s'}$ for any
$5/2<s'<s$. Pressure continuity at this lower regularity preserves
the Taylor sign for sufficiently high frequencies. Choosing the
fixed packet amplitude sufficiently small places all their backward
images in a smaller neighborhood contained in the prescribed initial
open set. At time zero, a Galilean boost changes only the velocity
by a vanishing constant; at the chosen positive time, it translates
the magnetic field by half a wavelength. Thus the initial distances
tend to zero, while the terminal magnetic fields remain separated
in $H^s$ on a fixed interior set, and hence in the fixed Eulerian metric. 
\subsubsection{{Main contributions}}\mbox{}\par
\vspace{2mm} 
The main contributions of this paper are summarized as follows:
\begin{enumerate}[label={\rm (\roman*)},topsep=5pt,itemsep=4pt]
\item \emph{Hadamard well-posedness in the standard Sobolev spaces.}
Under the Taylor sign condition, we establish local Hadamard
well-posedness with velocity, magnetic field, and free boundary
at the same $H^s$ regularity, for $5/2<s<4$, on bounded domains
whose closures are diffeomorphic to the closed ball.
Compared with \cite{IPTT-MHD}, we remove the additional conditions
$D_Bu,D_BB\in H^{s-1/2}$; standard $H^s$ data generally provide
only $H^{s-1}$ control.
This complements the $H^4$ construction on a periodic slab in
\cite{GuWang2019}
and the low-regularity a priori estimates of \cite{LuoZhang2020}.

\item \emph{A geometric energy estimate controlled by the standard low norm.}
We prove a fourth order estimate with coefficients controlled by
the standard $H^s$ bound, the fixed geometry, and the Taylor lower
bound, including $5/2<s\le3$, where bounded $H^s$ norms allow arbitrarily
large curvature.
A harmonic extension of the normal gives an extended tangential
quadratic form with spatial derivatives in $L^6$ and material
and magnetic derivatives in $L^\infty$.
The zero boundary values of $p$ and $D_tp$, together with fractional
interpolation, control the projected pressure derivatives and
boundary remainders. The constants are uniform in the approximation scale.

\item \emph{Strong approximation and continuous dependence on moving domains.}
We combine the weak comparison of \cite{IPTT-MHD} with compatible
regularizations whose constraint corrections preserve the
adjacent-scale $L^2$ estimates and fourth order bounds.
In particular, $H^{-1/2}$ estimates for the normal trace error
control the correction restoring magnetic tangency.
A separate family of inward regularizers, constructed from Piola
transforms, has domain-independent kernels and preserves divergence
on outward enlargements of the same order as the smoothing scale.
This permits differences controlled on domain intersections to be
compared on a common domain through the same operator at each finite stage.
Frequency-by-frequency estimates control the strong $H^s$ tails
by an $\ell^2$ frequency envelope, following the nonlinear
interpolation principle of \cite{AlazardBurqIfrimTataruZuily}.
A bootstrap over finite families then yields a common lifespan
in a neighborhood of each admissible datum,
strong convergence, and continuous dependence in the standard
moving-domain topology.

\item \emph{Nonuniform dependence near arbitrary admissible states.}
At every fixed positive time in a common local interval, the
solution map is not uniformly continuous on any nonempty open
subset of the local data neighborhood, in the fixed Eulerian metric.
Related results hold for Euler flows on fixed domains
\cite{HimonasMisiolek2010,Inci2015,InciLi2019} and water-wave
trajectory maps \cite{RimahSaid2023}.
We combine the Galilean mechanism of \cite{BourgainLi2019}
with localized magnetic perturbations prescribed at the observation
time and the continuous backward flow provided by the local
Hadamard theory. The construction respects divergence, magnetic
tangency, and the Taylor sign, and yields separation near arbitrary
admissible states, detected in the magnetic fields on a common
interior set of the two fluid domains.
\end{enumerate}

\subsubsection{Notations}
We use the inhomogeneous Sobolev spaces $H^r$ and the H\"older spaces
$C^{k,\alpha}$, with the underlying domain indicated when needed.
We write $\widehat f$ for the unitary Fourier transform, $f*g$ for
convolution, and use
$\|f\|_{H^r(\R^3)}^2=\int_{\R^3}(1+|\xi|^2)^r|\widehat f(\xi)|^2\,d\xi$.
For negative $r$, boundary Sobolev spaces are defined by duality with
respect to surface measure. The symbols $dx$ and $dS$ denote volume and
surface measure, and $\|\cdot\|_p$ denotes the $L^p$ norm on the domain
or surface under discussion. We write $\ell^p$ for the usual sequence
space. Brackets $[\cdot]_{H^r}$ and $[\cdot]_{C^\alpha}$ denote the
corresponding homogeneous seminorms, and $\langle\cdot,\cdot\rangle$
denotes a duality pairing when its arguments belong to dual spaces.
Spatial component indices range from $1$ to $3$, and repeated
component indices are summed. The vectors $e_1,e_2,e_3$ form the standard
basis of $\R^3$. We use $(\nabla v)_{ij}=\partial_jv_i$;
$\nabla^m$ denotes the full tensor of spatial derivatives of order $m$. For a multi-index
$\alpha\in\mathbb N_0^3$, where $\mathbb N_0=\{0,1,2,\ldots\}$, set
$|\alpha|=\alpha_1+\alpha_2+\alpha_3$ and
$D^\alpha=\partial_1^{\alpha_1}\partial_2^{\alpha_2}\partial_3^{\alpha_3}$.
The superscript $T$ denotes transpose, $\operatorname{tr}$ denotes the
matrix trace, $\otimes$ denotes tensor product, and $\mathrm{Id}$ denotes the
identity matrix or operator on the indicated space. For operators
$L,M$, $[L,M]=LM-ML$. Constants denoted by $C$ may change from line to
line; any allowed dependence is specified in the relevant estimate.
We write $A\lesssim B$ when $A\le CB$, and $A\simeq B$ when both
$A\lesssim B$ and $B\lesssim A$ hold.

\subsubsection{Paper structure}
The rest of the paper is organized as follows.
Section~\ref{sec:phase-spaces} defines the Sobolev state space, states
the main theorems, and collects the elliptic and pressure estimates. Section~\ref{sec:lowerreg} constructs compatible smooth
approximations and establishes weak stability.
Section~\ref{sec:fourth-order} proves the fourth order energy estimate
and uses it to obtain a common lifespan, strong convergence, uniqueness,
and continuous dependence. Section~\ref{sec:nonuniform} proves
nonuniform dependence of the Eulerian solution map.

\section{State space, main theorems, and coefficient estimates}
\label{sec:phase-spaces}

In this section, we introduce the local Sobolev state space and the
Eulerian metrics for states on moving domains, using a fixed reference
collar to describe the boundaries and smooth ambient solenoidal fields
to compare the bulk fields. We then give precise statements of the
two main theorems: local Hadamard well-posedness in the standard $H^s$
topology and failure of uniform continuity of the solution map on
every nonempty open subset of a local data neighborhood at each
nonzero time in a common existence interval. We also collect the
elliptic and pressure estimates used in the approximation and energy arguments.

In the rest of the paper, we fix $5/2<s<4$.
\subsection{Reference collar and local states}

Choose a smooth reference domain $\Om_*$, set $\Ga_*=\partial\Om_*$,
and let $N_*$ be the outward unit normal to $\Ga_*$. Fix a smooth vector field $X_*$ along $\Ga_*$
with a uniform transverse component,
\begin{equation*}
\inf_{y\in\Ga_*}X_*(y)\cdot N_*(y)>0.
\end{equation*}
For a sufficiently small tubular neighborhood, the map
\begin{equation*}
(y,r)\longmapsto y+rX_*(y)
\end{equation*}
parametrizes a collar of $\Ga_*$. For a scalar graph function $\eta$ on
$\Ga_*$, define the graph map and its image by
\begin{equation*}
G_\eta(y)=y+\eta(y)X_*(y),\qquad \Ga_\eta=G_\eta(\Ga_*),
\end{equation*}
and denote by $\Om_\eta$ the bounded domain enclosed by $\Ga_\eta$.
For a boundary $\Ga$ in this chart, $\eta_\Ga$ denotes its graph function.
We use bounded subcharts with a common reference collar, uniform
transversality and injectivity bounds, and uniform $C^1$ bounds for the
graph maps and their inverses. These are our fixed geometric bounds.
For any Sobolev index $r$ under consideration, we set
$\|\Ga\|_{H^r}:=\|\eta_\Ga\|_{H^r(\Ga_*)}$; H\"older norms of
$\Ga$ are defined in the same way from $\eta_\Ga$.

A state in the chart is a triple $U=(u,B,\Ga)$, where $u$ and $B$ are defined on
the bounded domain $\Om$ enclosed by $\Ga$ and satisfy
\begin{equation*}
\diver u=\diver B=0\quad\text{in }\Om,
\qquad B\cdot N=0\quad\text{on }\Ga.
\end{equation*}
The associated pressure solves the Dirichlet problem
\begin{equation*}
\Delta p=-\operatorname{tr}\bigl((\nabla u)^2\bigr)
+\operatorname{tr}\bigl((\nabla B)^2\bigr)
\quad\text{in }\Om,
\qquad p=0\quad\text{on }\Ga.
\end{equation*}
On each local set, we impose a common
lower bound $-\partial_Np\ge c_*>0$ on the free surface.

Throughout the paper, fluid domains have closures diffeomorphic to the closed ball.
For rough boundaries, diffeomorphism is understood in the $C^1$ sense.
For every $r>5/2$, let $\X^r$ denote the corresponding local Sobolev state
space:
\begin{equation*}
\|U\|_{\X^r}^2
:=\|\eta_\Ga\|_{H^r(\Ga_*)}^2
+\|u\|_{H^r(\Om)}^2
+\|B\|_{H^r(\Om)}^2,
\end{equation*}
which measures the size of a state $U$. Convergence is defined as follows.

\subsection{Convergence and Eulerian metrics}
\label{subsec:state-topology}

\begin{definition}
\label{def:stdtop}
Let $U_n=(u_n,B_n,\Ga_n)$ and $U=(u,B,\Ga)$ be states in one
bounded chart.  We say that $U_n\to U$ in $\X^s$ if the following three
conditions hold.
\begin{enumerate}[label=(\roman*),leftmargin=2.5em]
\item The surfaces remain in the fixed chart and all normal pressure coefficients have
a common positive lower bound.
\item $\eta_{\Ga_n}\to\eta_\Ga$ in $H^s(\Ga_*)$.
\item Given $\delta>0$, one can find smooth divergence-free ambient fields
$\widetilde u,\widetilde B$, defined on a neighborhood of
$\overline\Om$, for which
\begin{equation*}
\|u-\widetilde u\|_{H^s(\Om)}
+\|B-\widetilde B\|_{H^s(\Om)}<\delta
\end{equation*}
and, writing $\Om_n$ for the domain enclosed by $\Ga_n$,
\begin{equation*}
\limsup_{n\to\infty}
\left(\|u_n-\widetilde u\|_{H^s(\Om_n)}
+\|B_n-\widetilde B\|_{H^s(\Om_n)}\right)<\delta.
\end{equation*}
\end{enumerate}
For sufficiently large $n$, the domains $\Om_n$ are contained in the ambient
neighborhood on which the comparison fields are defined.
\end{definition}

When applying a whole space extension estimate to such a local comparison
field, we first multiply it by a smooth compactly supported cutoff inside
its domain of definition, equal to one on a neighborhood of
$\overline\Om$ and all sufficiently late $\overline\Om_n$, and extend
the product by zero. The resulting field belongs to $C_c^\infty(\R^3)$
and agrees with the comparison field on that common neighborhood, where
it remains divergence-free.

We use the ambient comparison topology of
\cite[Section~3.3, Remark~3.5]{IPTT-MHD} for the standard variables
$(u,B,\Ga)$. To metrize a bounded subchart, fix the
extension family $E_\Om$ from
\cite[Proposition~A.9]{IPTT-MHD}.  These operators are uniformly bounded
$H^s(\Om)\to H^s(\R^3)$ on the subcharts under consideration, and for any
fixed $F\in H^s(\R^3)$,
\begin{equation}\label{eq:extension-continuity}
E_{\Om_n}(F|_{\Om_n})\longrightarrow E_\Om(F|_\Om)
\quad\text{in }H^s(\R^3)
\end{equation}
whenever the graphs converge in $H^s$.

\begin{definition}
\label{def:extmetric}
For two states
$U=(u,B,\Ga)$ and $\widetilde U=(\widetilde u,\widetilde B,\widetilde\Ga)$
in the same bounded subchart, let $\Om$ and $\widetilde\Om$ be the domains
enclosed by their boundaries and define
\begin{equation*}
d(U,\widetilde U):=
\|\eta_\Ga-\eta_{\widetilde\Ga}\|_{H^s(\Ga_*)}
+\|E_\Om u-E_{\widetilde\Om}\widetilde u\|_{H^s(\R^3)}+\|E_\Om B-E_{\widetilde\Om}\widetilde B\|_{H^s(\R^3)}.
\end{equation*} 
\end{definition}

By \cref{lem:extmetric-top}, $d$ generates the $\X^s$ convergence in
\cref{def:stdtop}. We equip $C([0,T];\X^s)$ with the corresponding
supremum metric on bounded subcharts.

\begin{remark}
The metric $d$ is the strong Eulerian metric used in the main theorems.
The nonnegative comparison functional $\D$ defined in \cref{sec:lowerreg}
measures squared differences at the $L^2$ level on two moving domains.
Its evolution controls weak differences. The frequency estimates in
\cref{lem:telescoping} recover strong $H^s$ convergence in $d$.
\end{remark}

\subsection{Hadamard well-posedness and nonuniform dependence}

\begin{theorem}
\label{thm:main}
Let $5/2<s<4$. Suppose that $\Om_0$ is bounded,
$\overline\Om_0$ is diffeomorphic to the closed ball, and
$\Ga_0=\partial\Om_0\in H^s$ in the reference chart. Let $N_0$ be its
outward unit normal. Assume
\begin{equation*}
u_0,B_0\in H^s(\Om_0),
\qquad \diver u_0=\diver B_0=0,
\qquad B_0\cdot N_0=0\quad\text{on }\Ga_0,
\end{equation*}
and let the associated total pressure satisfy
\begin{equation*}
-\partial_{N_0}p_0\ge c_0>0\quad\text{on }\Ga_0.
\end{equation*}
Then there are $T>0$ and an open $\X^s$ neighborhood
$\mathscr U$ of the initial state such that every datum in $\mathscr U$
generates a unique solution
\begin{equation*}
U=(u,B,\Ga)\in C([0,T];\X^s).
\end{equation*}
These solutions remain in one controlled collar and preserve a common
positive Taylor lower bound on $[0,T]$. Their graphs have uniform
$C^{1,\alpha_s}$ bounds, where
$\alpha_s=\tfrac12+\tfrac12\min\{s-\tfrac52,\tfrac12\}$.

The trajectory map
\begin{equation*}
\Phi:\mathscr U\longrightarrow C([0,T];\X^s)
\end{equation*}
is continuous. We write $\Phi(t)U_0:=\Phi(U_0)(t)$.
\end{theorem}

The solution map also satisfies the following nonuniform dependence result.

\begin{theorem}
\label{thm:nonuniform}
Let $U_0$ be any datum covered by \cref{thm:main}. After shrinking the
neighborhood and lifespan in \cref{thm:main}, there are a bounded open
neighborhood $\mathscr U$ of $U_0$ and $T>0$ with the following property.
For every $0<\tau\le T$ and every nonempty open set
$\mathscr O\subset\mathscr U$, there are sequences
$U_k,\widetilde U_k\in\mathscr O$ such that
\begin{equation*}
d(U_k,\widetilde U_k)\longrightarrow0\qquad \text{and}\qquad
\liminf_{k\to\infty}
d\bigl(\Phi(\tau)U_k,
\Phi(\tau)\widetilde U_k\bigr)>0.
\end{equation*}
Thus, the map at time $\tau$ is not uniformly continuous on any nonempty open
subset of $\mathscr U$.
\end{theorem}

In particular, the solution map has no local Lipschitz or positive H\"older
modulus on any nonempty open subset. 

\subsection{Coefficient estimates and pressure continuity}

Fix the smooth reference domain $\Om_*$, with boundary $\Ga_*$,
and a sufficiently small $C^{1,\epsilon_0}$ collar of $\Ga_*$,
where $0<\epsilon_0<1/8$.
Domains $\Om$ have boundary $\Ga$ in this collar. Boundary Sobolev
norms are defined in these fixed coordinates.
Constants may depend on the fixed chart, its transversality bounds, and
its finitely many coordinate neighborhoods.

We fix a linear lifting $\mathscr L_*$ of vector fields on $\Ga_*$.
It is obtained by componentwise harmonic extension into $\Om_*$,
multiplication by a fixed smooth cutoff equal to one near $\Ga_*$
and supported in an interior collar, and a fixed compactly supported
extension from $\overline{\Om_*}$ to $\R^3$. The extension is chosen
bounded in both Sobolev and H\"older norms. Thus, for $5/2<r\le4$,
\begin{equation*}
(\mathscr L_*g)|_{\Ga_*}=g,\qquad
\|\mathscr L_*g\|_{H^{r+1/2}(\R^3)}
\le C_r\|g\|_{H^r(\Ga_*)},
\end{equation*}
and
\begin{equation*}
\|\mathscr L_*g\|_{C^{1,\epsilon_0}(\R^3)}
\le C\|g\|_{C^{1,\epsilon_0}(\Ga_*)}.
\end{equation*}
After decreasing the fixed collar, we have
$\|D\mathscr L_*(\eta X_*)\|_{L^\infty(\R^3)}<1/2$.
Consequently
\begin{equation}\label{eq:reference-domain-map}
\Xi_\eta
:=\bigl(\mathrm{Id}+\mathscr L_*(\eta X_*)\bigr)|_{\Om_*}
:\Om_*\longrightarrow\Om_\eta
\end{equation}
is an orientation-preserving $C^1$ diffeomorphism with uniformly
bounded derivative and inverse derivative. Indeed, its ambient
extension is a global diffeomorphism, maps $\Ga_*$ onto $\Ga_\eta$
by $G_\eta$, and therefore maps the bounded enclosed domain onto
$\Om_\eta$. In particular, $\Xi_\eta|_{\Ga_*}=G_\eta$.
The linearity and Sobolev bounds of $\mathscr L_*$ give continuous
dependence of $\Xi_\eta$ in $H^{r+1/2}(\Om_*)$ on $\eta$ in
$H^r(\Ga_*)$.

Set $M_m=1+\|U\|_{\X^m}$ for $m=s,4$, and assume
\begin{equation}\label{eq:low-bound}
M_s\le R,
\qquad \diver u=\diver B=0,
\qquad B\cdot N=0\quad\hbox{on }\Ga.
\end{equation}
For the time dependent estimates, the domain moves with normal velocity $u\cdot N$
and $(u,B,p)$ obey
\begin{equation*}
D_tu=B\cdot\nabla B-\nabla p,\qquad
D_tB=B\cdot\nabla u,\qquad
D_t=\partial_t+u\cdot\nabla,\qquad D_B=B\cdot\nabla.
\end{equation*}
We also write $D_t^\pm=D_t\pm D_B$.
Write $A=\nabla u$ and $\mathsf{C}=\nabla B$, with matrix entries
$A_{ij}=\partial_ju_i$ and $\mathsf{C}_{ij}=\partial_jB_i$. The pressure solves
\begin{equation}\label{eq:pressure-dirichlet}
\Delta p=\operatorname{tr}(\mathsf{C}^2)-\operatorname{tr}(A^2)
\quad\hbox{in }\Om,\qquad p|_{\Ga}=0.
\end{equation}
The constants $C_R$ may depend on $s$, $R$, the fixed collar, and a
fixed volume bound. For the energy estimates, we also assume
$a=-\partial_Np\ge c_0>0$, and their constants may additionally depend
on $c_0^{-1}$. All constants are independent of $M_4$.
Here $N$ is the outward unit normal and
$\Pi=\mathrm{Id}-N\otimes N$ is the orthogonal projection onto the tangent plane
of $\Ga$.

We use the balanced Dirichlet, trace, harmonic extension, and
Dirichlet-to-Neumann estimates of \cite[Appendix A.3--A.4]{IPTT-MHD}. We write the
geometric factors explicitly as $1+\|\Ga\|_{H^\rho}$, since
$\|\Ga\|_{H^\rho}$ denotes the graph norm in our convention;
\cite[Remark~A.5]{IPTT-MHD} suppresses the additive constant.

We also use, with constants uniform on the fixed collar, the following
standard consequences of the Sobolev calculus.  In
three dimensions, $H^\rho(\Om)$ is an algebra for $\rho>3/2$, and
multiplication satisfies
\begin{equation}\label{eq:uniform-product}
\|fg\|_{H^r(\Om)}\le C_R\|f\|_{H^{r_1}(\Om)}\|g\|_{H^{r_2}(\Om)}
\end{equation}
for nonnegative Sobolev orders $r,r_1,r_2$ satisfying
$r\le\min\{r_1,r_2\}$ and $r_1+r_2-r>3/2$.  On the two-dimensional boundary the corresponding
threshold is $1$.  We also use the interpolation inequality
\begin{equation*}
\|v\|_{H^{(1-\theta)r_0+\theta r_1}}
\le C_R\|v\|_{H^{r_0}}^{1-\theta}\|v\|_{H^{r_1}}^\theta,
\qquad 0\le\theta\le1.
\end{equation*}
For $x>0$, $0\le\theta<1$, and $\varepsilon>0$, we also use
$x^\theta\le \varepsilon x+C_{\varepsilon,\theta}$. 

In particular, with $r\ge0$, $0\le\alpha\le1$, and
$r+\alpha=1$, the Dirichlet and trace bounds are
\begin{align}
\|v\|_{H^\tau(\Om)}
&\lesssim \|\Delta v\|_{H^{\tau-2}(\Om)}
+\|v|_{\Ga}\|_{H^{\tau-1/2}(\Ga)}
+(1+\|\Ga\|_{H^{\tau+r-1/2}})\|v\|_{C^\alpha(\Om)},
&&\tau\ge2,\label{eq:balanced-dir}\\
\|v|_{\Ga}\|_{H^{\tau-1/2}(\Ga)}
&\lesssim \|v\|_{H^\tau(\Om)}
+(1+\|\Ga\|_{H^{\tau+r-1/2}})\|v\|_{C^\alpha(\Om)},
&&\tau>1/2.\label{eq:balanced-trace}
\end{align}

Estimates \eqref{eq:balanced-dir} and \eqref{eq:balanced-trace} follow from
\cite[Propositions~A.14 and A.7]{IPTT-MHD}, respectively. For
$(r,\alpha)=(1,0)$,
the trace estimate uses the $L^\infty$ norm of $v$ and the boundary norm
$H^{\tau+1/2}$.

\begin{lemma}\label{lem:elliptic-estimates}
Let $\Ga$ lie in the fixed collar and satisfy $\|\Ga\|_{H^s}\le R$.

\textup{(i)} The Dirichlet estimates are
\begin{align}
\|v\|_{H^2(\Om)}&\le C_R
\bigl(\|\Delta v\|_{L^2(\Om)}+
\|v|_{\Ga}\|_{H^{3/2}(\Ga)}\bigr),\label{eq:dir2}\\
\|v\|_{H^s(\Om)}&\le C_R\|\Delta v\|_{H^{s-2}(\Om)},
\qquad v|_{\Ga}=0.\label{eq:dir-s}
\end{align}
Estimates \eqref{eq:dir2}--\eqref{eq:dir-s} also hold for weak Dirichlet solutions.

\textup{(ii)} Suppose in addition that $\Ga$ is smooth. Let
$f\in H^{-1/2}(\Ga)$ satisfy $\langle f,1\rangle=0$, and let $\psi$ be
the weak harmonic solution with $\partial_N\psi=f$ and
$\int_\Ga\psi\,dS=0$. Then
\begin{align}
\|\nabla\psi\|_{H^m}
&\le C_R\left(\|f\|_{H^{m-1/2}(\Ga)}
+(1+\|\Ga\|_{H^{m+3/4}})
\|f\|_{H^{s-3/2}(\Ga)}\right),
\label{eq:neumann-tame}\\
\|\nabla\psi\|_{L^2}&\le C_R\|f\|_{H^{-1/2}(\Ga)}.
\label{eq:neumann-low}
\end{align}
Here \eqref{eq:neumann-tame} holds for $m\in\{s,4\}$ whenever
$f\in H^{m-1/2}(\Ga)$.
\end{lemma}
\begin{proof}
The proof is divided into four steps.

\par\nobreak\smallskip
\noindent\textit{Step 1. The $H^2$ Dirichlet estimate.}\par\nobreak

Let $F=\Delta v$ and $g=v|_\Ga$.  Apply
\eqref{eq:balanced-dir} with
$(\tau,r,\alpha)=(2,3/4,1/4)$.  Since
$2+r-1/2=9/4$ and $\|\Ga\|_{H^{9/4}}\le C_R$, we obtain
\begin{equation}\label{eq:dir2-preabsorb}
\|v\|_{H^2}
\le C_R\bigl(\|F\|_{L^2}+\|g\|_{H^{3/2}}+\|v\|_{C^{1/4}}\bigr).
\end{equation}
After applying a fixed Stein extension to $\R^3$, the three-dimensional
Sobolev embedding $H^{15/8}(\R^3)\hookrightarrow C^{1/4}(\R^3)$ and
interpolation between $H^1$ and $H^2$ give, for every $\varepsilon>0$,
\begin{equation}\label{eq:Cquarter-interp}
\|v\|_{C^{1/4}}
\le C\|v\|_{H^{15/8}}
\le \varepsilon\|v\|_{H^2}+C_\varepsilon\|v\|_{H^1}.
\end{equation}
A bounded trace lifting $\mathcal L:H^{1/2}(\Ga)\to H^1(\Om)$
and the energy estimate for $v-\mathcal Lg\in H^1_0(\Om)$ give
\begin{equation*}
\|v\|_{H^1}
\le C_R\bigl(\|F\|_{H^{-1}}+\|g\|_{H^{1/2}}\bigr)
\le C_R\bigl(\|F\|_{L^2}+\|g\|_{H^{3/2}}\bigr).
\end{equation*}
Insert this bound into \eqref{eq:Cquarter-interp} and then
\eqref{eq:dir2-preabsorb}. Choosing $\varepsilon$ small enough to absorb
$\|v\|_{H^2}$ proves \eqref{eq:dir2}.

\par\smallskip
\noindent\textit{Step 2. The homogeneous $H^s$ Dirichlet estimate.}\par\nobreak

Assume $v|_\Ga=0$.  Use \eqref{eq:balanced-dir} with
$(\tau,r,\alpha)=(s,0,1)$:
\begin{equation}\label{eq:dirs-preabsorb}
\|v\|_{H^s}
\le C_R\bigl(\|\Delta v\|_{H^{s-2}}+\|v\|_{C^1}\bigr),
\end{equation}
where $1+\|\Ga\|_{H^{s-1/2}}\le C_R$. With
$\kappa=(s+5/2)/2\in(5/2,s)$, Sobolev embedding and interpolation give
\begin{equation*}
\|v\|_{C^1}
\le C_s\|v\|_{H^\kappa}
\le \varepsilon\|v\|_{H^s}+C_{\varepsilon,s}\|v\|_{H^1}.
\end{equation*}
The zero boundary energy estimate
$\|v\|_{H^1}\le C_R\|F\|_{H^{-1}}\le C_R\|F\|_{H^{s-2}}$
then allows absorption in \eqref{eq:dirs-preabsorb}, proving
\eqref{eq:dir-s}, also componentwise for vector and tensor fields.
To obtain the estimates for weak solutions, choose smooth graphs
$\eta_n\to\eta$ in $H^s(\Ga_*)$ and write
$\Om_n=\Om_{\eta_n}$, $\Xi_n=\Xi_{\eta_n}$, and $\Xi=\Xi_\eta$
using \eqref{eq:reference-domain-map}. The maps converge in $C^1$
and have uniform bounds on their derivatives and inverse derivatives.
Extend $F$ from $\Om$ by a Stein operator, smooth this ambient
extension, and restrict it to $\Om_n$ to obtain $F_n$.
The ambient approximations converge in $L^2$, and in $H^{s-2}$
for \eqref{eq:dir-s}, with uniformly bounded norms.
For \eqref{eq:dir2}, approximate $g\circ G_\eta$ by smooth
functions $\widehat g_n$ in $H^{3/2}(\Ga_*)$ and prescribe
$g_n\circ G_{\eta_n}=\widehat g_n$; for \eqref{eq:dir-s}, take $g_n=0$.
Let $v_n$ solve the resulting smooth Dirichlet problems.

After pullback to $\Om_*$, the Dirichlet matrices converge uniformly
and remain uniformly elliptic. The pulled-back sources converge in
$L^2$, and the boundary values converge in $H^{1/2}(\Ga_*)$.
Subtracting fixed-domain $H^1$ liftings of these boundary values and
applying the variational difference estimate gives
$v_n\circ\Xi_n\to v\circ\Xi$ in $H^1(\Om_*)$, where variational
uniqueness identifies $v$ with the prescribed weak solution.
The estimates already proved give uniform $H^2(\Om_n)$ or
$H^s(\Om_n)$ bounds for $v_n$. Apply uniform Stein extensions
to these solutions and extract a weakly convergent subsequence in
the corresponding ambient Sobolev space. On each compact subset
of $\Om$, the fixed-domain $H^1$ convergence identifies its
restriction with $v$. Weak lower semicontinuity and restriction
therefore give \eqref{eq:dir2}--\eqref{eq:dir-s} on $\Om$.

\par\smallskip
\noindent\textit{Step 3. The low Neumann estimate.}\par\nobreak

The compatibility condition $\langle f,1\rangle=0$ gives a unique harmonic
solution modulo constants. Its weak formulation on $H^1(\Om)/\R$ is, for every $\phi\in H^1(\Om)$,
\begin{equation*}
\int_\Om\nabla\psi\cdot\nabla\phi\,dx
=\langle f,\phi|_\Ga\rangle_{H^{-1/2},H^{1/2}}.
\end{equation*}
Taking $\phi=\psi$ after fixing the boundary mean and using the trace and
Poincar\'e inequalities gives
\begin{equation*}
\|\nabla\psi\|_2^2
\le C_R\|f\|_{H^{-1/2}}\|\psi\|_{H^1}
\le C_R\|f\|_{H^{-1/2}}\|\nabla\psi\|_2.
\end{equation*}
This proves \eqref{eq:neumann-low}. For $h=\psi|_\Ga$,
\eqref{eq:neumann-low} and the trace theorem give
\begin{equation}\label{eq:h-L2-low}
\|h\|_{L^2(\Ga)}\le C_R\|f\|_{H^{-1/2}(\Ga)}.
\end{equation}

\par\smallskip
\noindent\textit{Step 4. Boundary regularity through the
Dirichlet-to-Neumann map.}\par\nobreak

Let $\mathcal N$ denote the Dirichlet-to-Neumann operator,
$\mathcal N g=\partial_N(\mathcal H_{\Om}g)$, where
$\mathcal H_{\Om}g$ is the harmonic extension of boundary data $g$.
Then $\mathcal N h=f$. Apply the one-derivative
Dirichlet-to-Neumann estimate from
\cite[Proposition~A.19]{IPTT-MHD} with
$(r,\alpha)=(1/4,3/4)$ and base index $s-3/2$.  It gives
\begin{equation}\label{eq:DN-base}
\|h\|_{H^{s-1/2}}
\le C_R\Bigl(
\|f\|_{H^{s-3/2}}+\|h\|_{L^2}
+(1+\|\Ga\|_{H^{s-1/4}})\|h\|_{C^{3/4}}\Bigr).
\end{equation}
Choose $2<\rho<s-1/2$. Boundary Sobolev embedding and interpolation give
\begin{equation*}
\|h\|_{C^{3/4}}
\le C\|h\|_{H^\rho}
\le \varepsilon\|h\|_{H^{s-1/2}}+C_\varepsilon\|h\|_{L^2}.
\end{equation*}
The bound $1+\|\Ga\|_{H^{s-1/4}}\le C_R$ permits absorption in
\eqref{eq:DN-base}; together with \eqref{eq:h-L2-low}, this yields
\begin{equation}\label{eq:h-base-bound}
\|h\|_{H^{s-1/2}}
\le C_R\|f\|_{H^{s-3/2}}.
\end{equation}
Choose $0<\beta<\min\{s-5/2,1/4\}$.  The boundary embedding gives
$h\in C^{1,\beta}$. The bound $\|\Ga\|_{H^s}\le R$ and $s>5/2$
give a uniform $C^{1,1/2}$ boundary bound. The harmonic Dirichlet
Schauder estimate therefore yields
\begin{equation*}
\|\psi\|_{C^1(\overline\Om)}
\le C_R\|h\|_{C^{1,\beta}(\Ga)}
\le C_R\|f\|_{H^{s-3/2}}.
\end{equation*}

At base index $m-1/2$, for $m\in\{s,4\}$,
\cite[Proposition~A.19]{IPTT-MHD} and \eqref{eq:h-base-bound} give
\begin{align*}
\|h\|_{H^{m+1/2}}
&\le C_R\Bigl(\|f\|_{H^{m-1/2}}
+(1+\|\Ga\|_{H^{m+3/4}})
(\|h\|_{C^{3/4}}+\|h\|_{L^2})\Bigr)\\
&\le C_R\Bigl(\|f\|_{H^{m-1/2}}
+(1+\|\Ga\|_{H^{m+3/4}})\|f\|_{H^{s-3/2}}\Bigr).
\end{align*}
Finally apply \eqref{eq:balanced-dir} to the harmonic function $\psi$
at order $\tau=m+1$ with $(r,\alpha)=(0,1)$.  Since $\Delta\psi=0$,
\begin{align*}
\|\psi\|_{H^{m+1}}
&\le C_R\Bigl(
\|h\|_{H^{m+1/2}}
+(1+\|\Ga\|_{H^{m+1/2}})\|\psi\|_{C^1}\Bigr)\\
&\le C_R\left(
\|f\|_{H^{m-1/2}}
+(1+\|\Ga\|_{H^{m+3/4}})\|f\|_{H^{s-3/2}}
\right).
\end{align*}
Taking one derivative proves \eqref{eq:neumann-tame}.
\end{proof}

Let $Q=D_tp$ be the material pressure derivative, and write
\begin{equation*}
a=-\partial_Np,\qquad a_Q=-\partial_NQ
\end{equation*}
for the normal derivative coefficients of $p$ and $Q$.
Set
\begin{equation*}
\theta_0=\frac{(3-s)_+}{4-s},\qquad
\theta_1=\frac{(7/2-s)_+}{4-s},
\qquad x_+:=\max\{x,0\}.
\end{equation*}
For $5/2<s<4$, they satisfy
\begin{equation}\label{eq:strict-exponents}
0\le\theta_0<\frac13,\qquad
0\le\theta_1<\frac23,\qquad
2\theta_0\le\theta_1,\qquad
\theta_0+\theta_1<1,\qquad 3\theta_0<1.
\end{equation}

\begin{lemma}\label{lem:pressure-estimates}
Let $(u,B,\Ga)$ satisfy \eqref{eq:low-bound}, and let $p$ solve the
problem \eqref{eq:pressure-dirichlet}. Here the Taylor coefficient $a$ may change sign.

\textup{(i)} Along smooth solutions,
\begin{align}
&\|p\|_{H^{s+1/2}}+\|D_Bp\|_{H^s}
+\|a\|_{H^{s-1}(\Ga)}
+\|Q\|_{H^s}+\|a_Q\|_{H^{s-3/2}(\Ga)}
\le C_R,\notag\\
&\|p\|_{H^{9/2}}+\|Q\|_{H^4}\le C_RM_4.
\label{eq:pressure-high}
\end{align}
The estimates for $p$, $D_Bp$, and $a$ hold for smooth states.
Their low-order bounds also hold for $H^s$ states.
Along smooth solutions,
\begin{equation}\label{eq:Dt-a}
D_ta=a_Q-a\,N\cdot A^TN,
\qquad \|D_ta\|_{L^\infty(\Ga)}\le C_R.
\end{equation}

\textup{(ii)} For every $5/2<s'<s$, the map
$U\mapsto a\circ G_{\eta_\Ga}$ is uniformly continuous into
$C^0(\Ga_*)$ on the class defined by \eqref{eq:low-bound},
with respect to the metric in \cref{def:extmetric}
with $s$ replaced by $s'$.

Consequently, strict Taylor sign is open in the standard $H^s$
topology. Along smooth solutions, $a(0)\ge c_0>0$ implies
$a(t)\ge c_0/2$ for $0\le t\le c_0/(2C_R)$
while \eqref{eq:low-bound} holds.
\end{lemma}
\begin{proof}
The proof is divided into seven steps.

\par\nobreak\smallskip
\noindent\textit{Step 1. Low regularity of the Dirichlet pressure.}\par\nobreak

Set $f=\operatorname{tr}(\mathsf{C}^2)-\operatorname{tr}(A^2)$,
where $A=\nabla u$ and $\mathsf{C}=\nabla B$.
Since $s-1>3/2$, the algebra property of $H^{s-1}(\Om)$
gives
\begin{equation*}
\|f\|_{H^{s-1}}
\le C_R\bigl(\|u\|_{H^s}^2+\|B\|_{H^s}^2\bigr)
\le C_R.
\end{equation*}
As $\Delta p=f$ and $p|_\Ga=0$, \eqref{eq:dir-s}
and the Sobolev embedding for $s>5/2$ yield
\begin{equation*}
\|p\|_{H^s}+\|p\|_{C^1}
\le C_R\|f\|_{H^{s-2}}\le C_R.
\end{equation*}
Applying \eqref{eq:balanced-dir} with
$(\tau,r,\alpha)=(s+1/2,0,1)$, we obtain
\begin{equation*}
\|p\|_{H^{s+1/2}}
\le C\Bigl(\|f\|_{H^{s-3/2}}
+(1+\|\Ga\|_{H^s})\|p\|_{C^1}\Bigr)
\le C_R.
\end{equation*}
Sobolev embedding then gives
\begin{equation}\label{eq:p-C1alpha}
\|p\|_{C^{1,1/4}(\overline\Om)}\le C_R.
\end{equation}

At order $9/2$, \eqref{eq:balanced-dir} gives
\begin{equation}\label{eq:p-high-pre}
\|p\|_{H^{9/2}}
\le C\Bigl(\|f\|_{H^{5/2}}
+(1+\|\Ga\|_{H^4})\|p\|_{C^1}\Bigr).
\end{equation}
The tame product estimate and the low $H^s$ bound imply
\begin{equation*}
\|f\|_{H^{5/2}}
\le C_R\bigl(\|u\|_{H^{7/2}}+
\|B\|_{H^{7/2}}\bigr)
\le C_RM_4.
\end{equation*}
Consequently,
$\|p\|_{H^{9/2}}\le C_RM_4$.

\par\smallskip
\noindent\textit{Step 2. Equation for the material pressure derivative.}\par\nobreak

For any smooth vector field $v$, the commutation identity
\begin{equation*}
D_t(\partial_jv_i)=\partial_j(D_tv_i)-A_{kj}\partial_kv_i
\end{equation*}
and the MHD equations yield
\begin{equation}\label{eq:At-Ct}
D_tA=D_B\mathsf{C}+\mathsf{C}^2-\nabla^2p-A^2,
\qquad
D_t\mathsf{C}=D_BA+A\mathsf{C}-\mathsf{C}A.
\end{equation}
Moreover,
\begin{equation}\label{eq:laplace-Dt-comm}
\Delta(D_tp)-D_t(\Delta p)
= (\Delta u)\cdot\nabla p
+2\partial_ku_i\,\partial_{ki}p.
\end{equation}
Differentiate $f=\operatorname{tr}(\mathsf{C}^2)-\operatorname{tr}(A^2)$
materially and use \eqref{eq:At-Ct}:
\begin{align*}
D_tf
&=2\operatorname{tr}(\mathsf{C}D_t\mathsf{C})-2\operatorname{tr}(AD_tA)\\
&=2\operatorname{tr}(\mathsf{C}D_BA)-2\operatorname{tr}(AD_B\mathsf{C})
-2\operatorname{tr}(A\mathsf{C}^2)
+2\operatorname{tr}(A\nabla^2p)
+2\operatorname{tr}(A^3).
\end{align*}
The terms $\operatorname{tr}(\mathsf{C}A\mathsf{C})$ and
$\operatorname{tr}(\mathsf{C}^2A)$ cancel by cyclicity of the trace.  Adding the
commutator \eqref{eq:laplace-Dt-comm} therefore gives
\begin{align}
\Delta Q={}&(\Delta u)\cdot\nabla p
+4\operatorname{tr}(A\nabla^2p)
+2\operatorname{tr}(A^3)-2\operatorname{tr}(A\mathsf{C}^2)\notag\\
&+2\operatorname{tr}(\mathsf{C}D_BA)
-2\operatorname{tr}(AD_B\mathsf{C}),
\qquad Q|_\Ga=0.\label{eq:Q-equation}
\end{align}
The boundary value is zero because $p=0$ on the material free boundary,
so the material derivative of its trace is zero.

To estimate the source in $H^{s-2}$, \eqref{eq:uniform-product} gives
\begin{align}
\|(\Delta u)\cdot\nabla p\|_{H^{s-2}}
&\le C\|\Delta u\|_{H^{s-2}}\|\nabla p\|_{H^{s-1/2}}
\le C_R,\notag\\
\|A\nabla^2p\|_{H^{s-2}}
&\le C\|A\|_{H^{s-1}}\|\nabla^2p\|_{H^{s-3/2}}
\le C_R.\label{eq:pressure-products-low}
\end{align}
The cubic terms are bounded in $H^{s-1}$ because $s-1>3/2$.
For the magnetic terms, write
$\mathsf{C}D_BA=\mathsf{C}B_k\partial_kA$ and $AD_B\mathsf{C}=AB_k\partial_k\mathsf{C}$.
Here $AB_k,\mathsf{C}B_k\in H^{s-1}$ and
$\partial_kA,\partial_k\mathsf{C}\in H^{s-2}$, so \eqref{eq:uniform-product} yields
\begin{equation*}
\|\mathsf{C}D_BA\|_{H^{s-2}}+\|AD_B\mathsf{C}\|_{H^{s-2}}\le C_R.
\end{equation*}
Thus $\|\Delta Q\|_{H^{s-2}}\le C_R$.
By \eqref{eq:dir-s},
\begin{equation}\label{eq:Q-low}
\|Q\|_{H^s}\le C_R.
\end{equation}
Since $s>5/2$, $H^s(\Om)\hookrightarrow W^{1,\infty}(\Om)$, and hence
$\|Q\|_{C^1}\le C_R$.

\par\smallskip
\noindent\textit{Step 3. The magnetic pressure derivative.}\par\nobreak

Commuting $D_B=B\cdot\nabla$ with the Laplacian gives
\begin{equation*}
\Delta(D_Bp)
=D_Bf+(\Delta B)\cdot\nabla p
+2\partial_kB\cdot\nabla\partial_kp.
\end{equation*}
Magnetic tangency and $p|_\Ga=0$ give $D_Bp|_\Ga=0$.
Each term in $D_Bf$ has the form $B_kA\partial_kA$ or
$B_k\mathsf{C}\partial_k\mathsf{C}$, and hence belongs to $H^{s-2}$ with norm at most $C_R$.
For the remaining two terms, use \eqref{eq:pressure-products-low}
with $u$ replaced by $B$. Therefore
\eqref{eq:dir-s} gives
\begin{equation*}
\|D_Bp\|_{H^s}\le C_R.
\end{equation*}
Together with \eqref{eq:Q-low}, this gives
$\|D_t^\pm p\|_{H^s}\le C_R$.

\par\smallskip
\noindent\textit{Step 4. The high $H^4$ estimate for $Q$.}\par\nobreak

Choose
\begin{equation*}
0<\epsilon<\min\{s-5/2,1/4,(4-s)/4\},
\qquad e:=\frac{\epsilon}{4-s}.
\end{equation*}
Interpolation between the low $H^s$ bounds and the $H^4$ state norm,
using the low bound directly when the target order is at most $s$, gives
\begin{align}
\|V\|_{H^3}&\le C_RM_4^{\theta_0},
&V&\in\{u,B\},\label{eq:interp-H3-fields}\\
\|V\|_{H^{7/2+\epsilon}}
&\le C_RM_4^{\theta_1+e},
&\|V\|_{W^{2,\infty}}&\le C_RM_4^{\theta_1+e},\notag
\end{align}
using $H^{7/2+\epsilon}\hookrightarrow W^{2,\infty}$.
Interpolation between $H^{s+1/2}$ and $H^{9/2}$ similarly gives
\begin{align}
\|p\|_{H^4}&\le C_RM_4^{\theta_1},\label{eq:p-H4-interp}\\
\|p\|_{H^{7/2+\epsilon}}
&\le C_RM_4^{\theta_0+e},
&\|p\|_{W^{2,\infty}}&\le C_RM_4^{\theta_0+e}.\notag
\end{align}
The exponents in \eqref{eq:interp-H3-fields} and \eqref{eq:p-H4-interp}
also apply when the target order is at most the low Sobolev order. Also $\|p\|_{H^3}\le C_R$, since $s+1/2>3$.
For $s\le3$, use $\epsilon<s-5/2$; for $s>3$, use
$\theta_0=0$, $\theta_1\le1/2$, and $e<1/4$. These give
\begin{equation*}
\theta_1+e<1,
\qquad \theta_0+\theta_1+e<1,
\qquad 2\theta_0+e<1.
\end{equation*}

To estimate the source of \eqref{eq:Q-equation} in $H^2$, use the
integer tame product rule,
\begin{align*}
\|(\Delta u)\cdot\nabla p\|_{H^2}
&\le C\Bigl(
\|u\|_{H^4}\|\nabla p\|_{L^\infty}
+\|u\|_{W^{2,\infty}}\|p\|_{H^3}\Bigr)\le C_R\bigl(M_4+M_4^{\theta_1+e}\bigr)
\le C_RM_4,\\
\|A\nabla^2p\|_{H^2}
&\le C\Bigl(
\|A\|_{L^\infty}\|p\|_{H^4}
+\|u\|_{H^3}\|p\|_{W^{2,\infty}}\Bigr)\le C_R\bigl(M_4^{\theta_1}+M_4^{2\theta_0+e}\bigr)
\le C_RM_4.
\end{align*}
For the directional terms, retain the undifferentiated factor $B_k$.
By \eqref{eq:interp-H3-fields},
$\|AB_k\|_{H^2}+\|\mathsf{C}B_k\|_{H^2}\le C_RM_4^{\theta_0}$,
while their $L^\infty$ norms are bounded by $C_R$. Thus
\begin{align*}
&\|\mathsf{C}B_k\partial_kA\|_{H^2}+\|AB_k\partial_k\mathsf{C}\|_{H^2}\\
&\quad\le C\Bigl[
(\|\mathsf{C}B_k\|_\infty+\|AB_k\|_\infty)
(\|\nabla A\|_{H^2}+\|\nabla \mathsf{C}\|_{H^2})\\
&\hspace{32mm}
+(\|\mathsf{C}B_k\|_{H^2}+\|AB_k\|_{H^2})
(\|\nabla A\|_\infty+\|\nabla \mathsf{C}\|_\infty)\Bigr]\\
&\quad\le C_R\Bigl(M_4+M_4^{\theta_0+\theta_1+e}\Bigr)
\le C_RM_4.
\end{align*}
The cubic terms satisfy
\begin{equation*}
\|A^3\|_{H^2}+\|A\mathsf{C}^2\|_{H^2}
\le C_R(\|A\|_{H^2}+\|\mathsf{C}\|_{H^2})
\le C_RM_4^{\theta_0}\le C_RM_4.
\end{equation*}
Consequently $\|\Delta Q\|_{H^2}\le C_RM_4$.
Applying \eqref{eq:balanced-dir} at $\tau=4$ and using
\eqref{eq:Q-low} with $H^s\hookrightarrow C^1$,
\begin{equation*}
\|Q\|_{H^4}\le C\Bigl(\|\Delta Q\|_{H^2}
+(1+\|\Ga\|_{H^{7/2}})\|Q\|_{C^1}\Bigr)\le C_RM_4.
\end{equation*}
This proves \eqref{eq:pressure-high}.

\par\smallskip
\noindent\textit{Step 5. Normal traces.}\par\nobreak

Apply \eqref{eq:balanced-trace} to $\nabla p$ with
$(\tau,r,\alpha)=(s-1/2,1,0)$.  Since
$\nabla p\in H^{s-1/2}(\Om)$,
\begin{equation*}
\|\nabla p|_\Ga\|_{H^{s-1}}
\le C_R\Bigl(
\|p\|_{H^{s+1/2}}+(1+\|\Ga\|_{H^s})\|\nabla p\|_\infty\Bigr)
\le C_R.
\end{equation*}
Similarly, for $\nabla Q$ with
$(\tau,r,\alpha)=(s-1,1,0)$,
\begin{equation*}
\|\nabla Q|_\Ga\|_{H^{s-3/2}}
\le C_R\Bigl(
\|Q\|_{H^s}+(1+\|\Ga\|_{H^{s-1/2}})\|\nabla Q\|_\infty\Bigr)
\le C_R.
\end{equation*}
The graph formula and Moser estimates give $\|N\|_{H^{s-1}(\Ga)}\le C_R$.
Since $s-3/2>1$, the boundary product estimates yield
\begin{equation}\label{eq:pressure-normal-low}
\|a\|_{H^{s-1}(\Ga)}+\|a_Q\|_{H^{s-3/2}(\Ga)}\le C_R.
\end{equation}
For an $H^s$ state, apply the approximation argument in the proof of
\cref{lem:elliptic-estimates}\textup{(i)} to the graph and to a
Stein extension of the pressure source $f\in H^{s-1}(\Om)$.
The smooth Dirichlet solutions have the uniform $H^{s+1/2}$ bound
from Step~1. Their uniform Stein extensions converge subsequentially
weakly in this space and strongly in $C^1$ on a fixed ball containing
the domains. The fixed-domain variational limit identifies their
restriction with $p$, which therefore has the same Sobolev bound.
The normals converge uniformly in the reference coordinates, so the
pulled-back normal pressure traces converge uniformly to
$a\circ G_\eta$. The trace estimates above bound these traces
uniformly in $H^{s-1}(\Ga_*)$. Weak lower semicontinuity gives
the asserted bound for $a$ in \eqref{eq:pressure-normal-low}.
The equation in Step~3 for $D_Bp$ holds distributionally. Since
$B\in C^1$, $p\in C^1$, $B\cdot N=0$, and $p|_\Ga=0$,
the trace of $D_Bp$ is zero. Moreover, $D_Bp\in H^{s-1/2}(\Om)
\subset H^1(\Om)$ by the Sobolev product estimate, so it is the
weak zero Dirichlet solution of that equation. Its source belongs
to $H^{s-2}$ by the estimates in Step~3, and \eqref{eq:dir-s}
yields the stated $H^s$ bound.

\par\smallskip
\noindent\textit{Step 6. Continuity of the normal pressure coefficient.}\par\nobreak

Let $U_n=(u_n,B_n,\Ga_n)\to U$ in the weaker $H^{s'}$ topology
of the statement, with a common $H^s$ bound. Denote by $\Om_n$ the domain
enclosed by $\Ga_n$, by $N_n$ its outward unit normal, by $p_n$ its
pressure, and by $f_n=\operatorname{tr}((\nabla B_n)^2)
-\operatorname{tr}((\nabla u_n)^2)$ the pressure source.
Use the fixed maps \eqref{eq:reference-domain-map}, with
$\Xi_n=\Xi_{\eta_{\Ga_n}}$ and $\Xi=\Xi_{\eta_\Ga}$.
Their construction gives
\begin{equation}\label{eq:Xi-conv}
\Xi_n\to\Xi\quad\hbox{in }H^{s'+1/2}\cap C^1,
\qquad
\sup_n\bigl(\|\Xi_n\|_{H^{s+1/2}}+\|\Xi_n\|_{C^{1,\alpha}}
+\|\Xi_n^{-1}\|_{W^{1,\infty}}\bigr)\le C_R
\end{equation}
for a fixed $0<\alpha<1/4$. \Cref{def:extmetric}, the embedding
$H^{s'}(\R^3)\hookrightarrow C^1(\R^3)$, and \eqref{eq:Xi-conv}
imply $C^1$ convergence of the pulled back fields. Hence their pressure sources satisfy
\begin{equation}\label{eq:source-unif-conv}
f_n\circ\Xi_n\longrightarrow f\circ\Xi
\quad\hbox{uniformly, and in particular in }L^2(\Om_*).
\end{equation}

Set
\begin{equation*}
v_n=p_n\circ\Xi_n,
\quad v=p\circ\Xi,
\quad J_n=\det D\Xi_n,
\quad H_n=J_n(D\Xi_n)^{-1}(D\Xi_n)^{-T},
\end{equation*}
and define $J,H$ analogously. Here $D\Xi_n$ is the Jacobian matrix,
and $(D\Xi_n)^{-T}$ is the transpose of its inverse. The matrices
$H_n,H$ are the coefficients of the pulled-back Dirichlet forms.
The weak equations on the fixed domain are, for $\varphi\in H^1_0(\Om_*)$,
\begin{align*}
\int_{\Om_*}H_n\nabla v_n\cdot\nabla\varphi
&=-\int_{\Om_*}J_n(f_n\circ\Xi_n)\varphi,\\
\int_{\Om_*}H\nabla v\cdot\nabla\varphi
&=-\int_{\Om_*}J(f\circ\Xi)\varphi.
\end{align*}
Put $w_n=v_n-v$ and take $\varphi=w_n$ in the difference. Let
$\lambda_R>0$ be a common lower ellipticity bound for $H_n$. Then
\begin{equation*}
\lambda_R\|\nabla w_n\|_2^2
\le \|H_n-H\|_\infty
\|\nabla v\|_2\|\nabla w_n\|_2+C\|J_n(f_n\circ\Xi_n)-J(f\circ\Xi)\|_{H^{-1}}
\|\nabla w_n\|_2.
\end{equation*}
The two coefficients on the right tend to zero by
\eqref{eq:Xi-conv}--\eqref{eq:source-unif-conv}.  Poincar\'e's inequality
therefore yields
\begin{equation}\label{eq:pressure-H1-conv}
v_n\longrightarrow v\quad\hbox{in }H^1_0(\Om_*).
\end{equation}
By \eqref{eq:p-C1alpha} and \eqref{eq:Xi-conv}, the sequence $(v_n)$
is bounded in $C^{1,\beta}$ for some $\beta>0$. Compactness in $C^1$
and uniqueness of the $H^1$ limit in \eqref{eq:pressure-H1-conv} yield
\begin{equation}\label{eq:pressure-C1-conv}
v_n\longrightarrow v\quad\hbox{in }C^1(\overline{\Om_*}).
\end{equation}
Writing $a_n=-\partial_{N_n}p_n$, we have on $\partial\Om_*$,
\begin{equation}\label{eq:normal-pressure-pullback}
a_n\circ\Xi_n
=-(N_n\circ\Xi_n)\cdot (D\Xi_n)^{-T}\nabla v_n.
\end{equation}
Every factor on the right converges uniformly, so the pulled back Taylor
coefficients converge in $C^0$.

For the uniform assertion, suppose that there are $\varepsilon_0>0$
and pairs $U_n,\widetilde U_n$ satisfying \eqref{eq:low-bound} whose
distance at order $s'$ tends to zero, while
\begin{equation*}
\|a_n\circ G_{\eta_n}-\widetilde a_n\circ G_{\widetilde\eta_n}
\|_{C^0(\Ga_*)}\ge\varepsilon_0.
\end{equation*}
Here $\eta_n,\widetilde\eta_n$ are their graph functions and
$a_n,\widetilde a_n$ their Taylor coefficients. The common $H^s$
bound gives subsequential compactness of the graphs and pulled back
fields in $C^1$. The distance at order $s'$ controls the graph and
ambient field discrepancies in $C^1$, so the two sequences have the
same limits. Their pulled back Dirichlet matrices and pressure sources
therefore converge to the same limits. Subtracting the two weak pressure
equations, as in \eqref{eq:pressure-H1-conv}, gives convergence of the
pressure difference in $H^1_0(\Om_*)$. The uniform $C^{1,\beta}$ bounds
then give convergence in $C^1$, and \eqref{eq:normal-pressure-pullback}
makes the displayed discrepancy tend to zero. This contradiction proves
the uniform assertion.

\par\smallskip
\noindent\textit{Step 7. Evolution of the Taylor coefficient.}\par\nobreak

On the free surface, $\nabla p=-aN$ and $N\cdot D_tN=0$.
Using $D_t\nabla p=\nabla Q-A^T\nabla p$, we obtain
\begin{equation*}
D_ta=-D_t(N\cdot\nabla p)
=-\partial_NQ-a\,N\cdot A^TN
=a_Q-a\,N\cdot A^TN.
\end{equation*}
Estimates \eqref{eq:low-bound} and \eqref{eq:pressure-normal-low}, together
with $H^{s-3/2}(\Ga)\hookrightarrow L^\infty(\Ga)$, give
$\|D_ta\|_\infty\le C_R$, proving \eqref{eq:Dt-a}.  Along each material boundary trajectory,
\begin{equation*}
a(t)\ge a(0)-C_Rt.
\end{equation*}
Thus $a(0)\ge c_0$ implies $a(t)\ge c_0/2$ for
$t\le c_0/(2C_R)$.  The convergence in \eqref{eq:pressure-C1-conv} and the trace formula
\eqref{eq:normal-pressure-pullback} give openness of the strict Taylor
sign in the standard $H^s$ topology.
\end{proof}

\section{Smooth approximations and weak stability}
\label{sec:lowerreg}

In this section, we construct smooth approximations of $H^s$
initial data that preserve the divergence and magnetic tangency
constraints, with approximation errors controlled by frequency
envelopes. We then establish weak stability estimates for the
corresponding smooth solutions on moving domains, uniformly
on intervals where the $H^s$ norms and collar geometry remain bounded
and the Taylor coefficient has a common positive lower bound.
These estimates provide the low-order comparison used with the fourth
order energy bound in \cref{sec:fourth-order}. Here $R$ is a fixed bound for $1+\|U\|_{\X^s}$.
Constants $C_R$ may depend on $R$, $s$, the fixed collar and its geometric
bounds, and the reciprocal of the common Taylor lower bound, but are
independent of the approximation scale.

For $r>5/2$, define the state space with additional magnetic-direction
regularity by
\begin{gather*}
\mathcal Y^r=\{U=(u,B,\Ga)\in\X^r:
D_Bu,D_BB\in H^{r-1/2}(\Om)\},\\
\|U\|_{\mathcal Y^r}^2=\|U\|_{\X^r}^2
+\|D_Bu\|_{H^{r-1/2}(\Om)}^2+\|D_BB\|_{H^{r-1/2}(\Om)}^2.
\end{gather*}
On our local set of domains, this is the state space of
\cite[Section~3.3]{IPTT-MHD}, with the corresponding ambient comparison
topology for the fields and their magnetic derivatives.
As with $\|U\|_{\X^r}$, $\|U\|_{\mathcal Y^r}$ measures the size of a
state in a nonlinear space. Standard $H^s$ data only give
\begin{equation*}
\|D_Bu\|_{H^{s-1}}+\|D_BB\|_{H^{s-1}}
\le C\|B\|_{H^s}\bigl(\|u\|_{H^s}+\|B\|_{H^s}\bigr).
\end{equation*}
The existence theorem \cite[Theorem~1.9]{IPTT-MHD} in $\mathcal Y^r$
supplies the smooth approximating solutions.
For continuation, the standard $H^4$ bound \eqref{eq:H4bound} controls
their $\mathcal Y^{7/2}$ size through \eqref{eq:enh-72-bound}.

For two states $U_i=(u_i,B_i,\Ga_i)$, $i=1,2$, let $\Om_i$ be the
domain enclosed by $\Ga_i$, let $N_i$ be its outward unit normal, and
let $p_i$ be the associated pressure. With $a_i=-\partial_{N_i}p_i$,
we use the weak comparison functional of
\cite[(4.2)--(4.3)]{IPTT-MHD}:
\begin{equation*}
\D(U_1,U_2)= 
\int_{\Om_1\cap\Om_2}\bigl(|u_1-u_2|^2+|B_1-B_2|^2\bigr)\,dx +\int_{\Ga_1\cap\Om_2}a_1^{-1}|p_2|^2\,dS
+\int_{\Ga_2\cap\Om_1}a_2^{-1}|p_1|^2\,dS.
\end{equation*}
Its boundary part is locally equivalent to the squared $L^2$ distance
between the graph functions by \cref{prop:Dprop}\textup{(i)}. Thus
$\D$ supplies the weaker comparison propagated by the equations,
whereas $d$ measures the strong $H^s$ separation in the main theorems.

\subsection{Compatible smoothing and frequency envelopes}

Fix
\begin{equation*}
0<\sigma<\min\left\{\frac1{64},\frac{s-5/2}{64},\frac{4-s}{64}\right\}.
\end{equation*}
Choose a smooth cutoff $\chi:[0,\infty)\to[0,1]$, equal to one on
$[0,1]$ and zero on $[2,\infty)$. On $\R^3$, let
$\Lambda=\sqrt{-\Delta}$, and on $\Ga_*$ let
$\Lambda=\sqrt{-\Delta_{\Ga_*}}$, where $\Delta_{\Ga_*}$ is the
Laplace--Beltrami operator of the fixed reference surface. Define
\begin{equation*}
L_j=\chi(2^{-j}\Lambda)\quad(j\ge0),\qquad
P_0=L_0,\qquad P_j=L_j-L_{j-1}\quad(j\ge1).
\end{equation*}
These are Fourier multipliers on $\R^3$ and spectral multipliers on
$\Ga_*$. We use the same notation on both spaces; the argument
specifies which operator is intended. Thus $L_j=\sum_{k=0}^jP_k$
is the low frequency cutoff at scale $2^j$, and $(P_j)_{j\ge0}$ is the
associated inhomogeneous Littlewood--Paley decomposition.
For $U_0=(u,B,\Ga)$, write
$\eta=\eta_\Ga$ and define, using the extension family from
\cref{def:extmetric},
\begin{equation}\label{eq:dyadic-coefficients}
\alpha_j^2:=
2^{2sj}\Bigl(
\|P_jE_\Om u\|_{L^2(\R^3)}^2
+\|P_jE_\Om B\|_{L^2(\R^3)}^2
+\|P_j\eta\|_{L^2(\Ga_*)}^2\Bigr).
\end{equation}
Here $\alpha_j\ge0$. The sequence $\alpha=(\alpha_j)_{j\ge0}$ belongs
to $\ell^2$, and its norm is equivalent
on the collar to the $H^s$ size.  Set
\begin{equation}
c_j:=2^{-\sigma j}
+\left(\sum_{k\ge0}2^{-2\sigma|j-k|}\alpha_k^2\right)^{1/2}.
\label{eq:envelope}
\end{equation}
Then $c=(c_j)_{j\ge0}\in\ell^2$, $c_j\to0$, and
$c_j\le C2^{\sigma|j-k|}c_k$.
Since $c_j\ge2^{-\sigma j}$ and $\sigma<4-s$, we also have
$2^{(4-s)j}c_j\ge1$.
If $U_0^{(n)}\to U_0$ in $\X^s$, let $\alpha^{(n)},c^{(n)}$ be the
corresponding sequences in \eqref{eq:dyadic-coefficients}--\eqref{eq:envelope}.
Uniform extension bounds and \eqref{eq:extension-continuity}, applied
to the common smooth comparison fields in \cref{def:stdtop}, give
convergence of the ambient extensions in $H^s$. Together with graph
convergence, the square function characterization of $H^s$ yields
$\alpha^{(n)}\to\alpha$ in $\ell^2$.
The reverse triangle inequality in \eqref{eq:envelope} and Fubini's
theorem give
\begin{equation}\label{eq:envelope-continuity}
\|c^{(n)}-c\|_{\ell^2}
\le C_\sigma\|\alpha^{(n)}-\alpha\|_{\ell^2}\longrightarrow0.
\end{equation}
Thus the envelope depends continuously on the datum, and its tails are
uniformly small along every convergent sequence.

\begin{lemma}\label{lem:compatible-regularization}
Let $U_0$ range over a bounded set of admissible initial data
in $\X^s$, with uniform geometric bounds and a common positive
lower bound for the Taylor coefficient. Then there exist an
integer $j_0$, independent of $U_0$, and smooth compatible
approximations $S_jU_0=(u_j,B_j,\Ga_j)$, $j\ge j_0$. Writing
$\Om_j$ for the domain enclosed by $\Ga_j$, these satisfy
\begin{align}
&S_jU_0\longrightarrow U_0\quad\hbox{in }\X^s,
\qquad
\sup_{j\ge j_0}\|S_jU_0\|_{\X^s}\le C,
\label{eq:regconv}\\
&\|S_jU_0\|_{\X^4}\le C2^{(4-s)j}c_j,
\label{eq:reghigh}\\
&\|u_{j+1}-u_j\|_{L^2(\Om_j\cap\Om_{j+1})}
+\|B_{j+1}-B_j\|_{L^2(\Om_j\cap\Om_{j+1})}
+\|\eta_{\Ga_{j+1}}-\eta_{\Ga_j}\|_{L^2(\Ga_*)}
\le C2^{-sj}c_j,
\label{eq:reglow}\\
&\D(S_jU_0,S_{j+1}U_0)^{1/2}
\le C2^{-sj}c_j.
\label{eq:D0}
\end{align}
The constant $C$ is uniform over the prescribed set, and $c_j$ is the
envelope in \eqref{eq:envelope}. For each fixed $j\ge j_0$ and finite
$m>5/2$, the map $S_j$ is continuous from this set into $\X^m$.
\end{lemma}

\begin{proof}
The proof is divided into seven steps. 
\par\nobreak\smallskip
\noindent\textit{Step 1. Smoothing the graph.}\par\nobreak

Set
\begin{equation*}
\eta_j=L_j\eta,
\qquad \Ga_j=\Ga_{\eta_j},
\end{equation*}
and let $\Om_j$ be the enclosed domain, with outward unit normal
$N_j$. If $k\le j+O(1)$, then
\begin{equation*}
2^{sk}\|P_k\eta\|_2\le \alpha_k\le c_k
\le C2^{\sigma(j-k)}c_j.
\end{equation*}
Summing the geometric series gives, for every fixed $m>s+\sigma$,
\begin{equation}
\|L_j\eta\|_{H^m}
\le C_R+C c_j2^{\sigma j}
\sum_{k<j+O(1)}2^{(m-s-\sigma)k}
\le C_R+C2^{(m-s)j}c_j.                 \label{eq:eta-sm-high}
\end{equation}
In particular, \eqref{eq:eta-sm-high} applies at $m=4$ and every
$m\ge s+1/2$.
At $m=s$, the square function estimate gives
\begin{equation*}
\|\eta_j\|_{H^s}\le C_R.                 
\end{equation*}
Because $L_{j+1}-L_j$ is supported at frequencies comparable with
$2^j$,
\begin{equation}
\|\eta_{j+1}-\eta_j\|_{L^2(\Ga_*)}
\le C2^{-sj}c_j.                         \label{eq:eta-increment-L2}
\end{equation}
Bernstein's inequality on the reference surface and slow variation give
\begin{equation}
\|\eta-\eta_j\|_{L^\infty}
\le C\sum_{k\ge j-O(1)}2^k\|P_k\eta\|_2
\le C2^{-(s-1)j}c_j.                    \label{eq:eta-tail-Linf}
\end{equation}
Since $s>5/2$, this is $o(2^{-3j/2})$, uniformly on the bounded local
set. Let $K_j$ denote the scalar regularizer at scale $2^{-j}$
from \cite[Proposition~A.32]{IPTT-MHD}, applied componentwise to
vector fields. We fix its partition of unity and inward sampling
directions throughout this set. For all sufficiently large $j$,
$K_jF$ is defined on both $\Om$ and $\Om_j$ whenever $F$ is defined
on $\Om$. The kernel is given by \eqref{eq:scalar-kernel}.

\par\smallskip
\noindent\textit{Step 2. Smoothing the bulk fields and restoring
$\diver F=0$.}\par\nobreak

Let $F$ denote either $u$ or $B$.  The scalar kernel estimates, together
with the frequency envelope, give
\begin{align}
\|K_jF\|_{H^s(\Om_j)}&\le C_R,
\label{eq:kernel-Hs}\\
\|K_jF\|_{H^4(\Om_j)}
&\le C_R+C2^{(4-s)j}c_j,                 \label{eq:kernel-H4}\\
\|(K_{j+1}-K_j)F\|_{L^2(\Om_j\cap\Om_{j+1})}
&\le C2^{-sj}c_j.                         \label{eq:kernel-inc}
\end{align}
Fix an integer $M\ge8$ for the moment conditions below.
Convolution commutes with $(1-\Delta)^{s/2}$; Young's inequality
and boundedness of multiplication by the fixed smooth cutoffs
give the first bound.
For the other two, the moment conditions and rapid Fourier decay give
\begin{align*}
\|K_jF\|_{H^4(\Om_j)}
&\le C\sum_{m\le j}2^{(4-s)m}c_m
+C\sum_{m>j}2^{(4-s)m-M(m-j)}c_m
\le C2^{(4-s)j}c_j,\\
\|(K_{j+1}-K_j)F\|_2
&\le C\sum_{m\ge0}2^{-M|m-j|}2^{-sm}c_m
\le C2^{-sj}c_j.
\end{align*}
The sums converge by slow variation, $4-s>\sigma$, and $M>s+\sigma$.

The localized scalar kernel introduces a divergence defect. Put
\begin{equation*}
g_j:=\diver(K_jF).
\end{equation*}
In the local convolution formula from
\cite[proof of Proposition~A.32]{IPTT-MHD}, let $(\chi_a)_a$ be the
fixed finite smooth partition of unity and let $\kappa_a$ be the
corresponding smooth compactly supported kernels. Then
\begin{equation}\label{eq:scalar-kernel}
K_jF(x)
=\sum_a\chi_a(x)(\kappa_{a,j}*F)(x),
\qquad \kappa_{a,j}(z)=2^{3j}\kappa_a(2^jz).
\end{equation}
Here $*$ denotes convolution in $\R^3$.
Each kernel $\kappa_a$ has mass one and vanishing moments of orders
$1,\ldots,M-1$. On boundary patches, its reflected support lies in the
fixed inward sampling cone. Thus, wherever $\chi_a$ is nonzero on $\Om_j$, convolution samples
only points of $\Om$. Since convolution commutes with constant
coefficient differentiation and $\diver F=0$ there,
\begin{equation*}
g_j=\sum_a\nabla\chi_a\cdot(\kappa_{a,j}*F).
\end{equation*}
Let $F^e=E_\Om F$ and $F_m=P_mF^e$. The same formula holds with $F^e$
in place of $F$. For $m\le j$, the identity
$\sum_a\nabla\chi_a=0$ allows us to replace
$\kappa_{a,j}*F_m$ by $(\kappa_{a,j}*-\mathrm{Id})F_m$.
The moment conditions give a factor $2^{-M(j-m)}$. For $m>j$, we
retain the original convolution and use the rapid decay of
$\widehat\kappa_a(2^{-j}\xi)$ on $|\xi|\simeq2^m$.
Here $\widehat\kappa_a$ is the Fourier transform of $\kappa_a$.
The fixed smooth coefficients $\nabla\chi_a$ are bounded multipliers
on $H^r$, so for $0\le r\le3$ these two estimates give
\begin{equation*}
\|g_j\|_{H^r(\Om_j)}
\le C\sum_{m\ge0}2^{rm}2^{-M|m-j|}\|F_m\|_2\le C2^{(r-s)j}c_j
\sum_{m\ge0}2^{-(M-|r-s|-\sigma)|m-j|}.
\end{equation*}
Since $M-|r-s|-\sigma>0$ for $s<4$, $0\le r\le3$, and $M\ge8$,
the sum is uniformly bounded, giving
\begin{equation}
\|g_j\|_{H^r(\Om_j)}
\le C2^{(r-s)j}c_j,\qquad 0\le r\le3.
\label{eq:gj-all}
\end{equation}

We solve
\begin{equation}
\Delta q_j=g_j\quad\hbox{in }\Om_j,
\qquad q_j=0\quad\hbox{on }\Ga_j,
\qquad
\widetilde F_j:=K_jF-\nabla q_j.
\end{equation}
Then $\diver\widetilde F_j=0$. Write $q_j^u$ and $q_j^B$ for the
potentials obtained with $F=u$ and $F=B$, respectively.

Applying \eqref{eq:balanced-dir} with
$(\tau,r,\alpha)=(11/4,0,1)$ and using
$\|\Ga_j\|_{H^{9/4}}\le C_R$ gives
\begin{equation*}
\|q_j\|_{H^{11/4}}
\le C_R\bigl(\|g_j\|_{H^{3/4}}+\|q_j\|_{C^1}\bigr).
\end{equation*}
Choose $5/2<\rho<11/4$.  The embedding
$H^\rho(\Om_j)\hookrightarrow C^1(\overline\Om_j)$, interpolation
between $H^1$ and $H^{11/4}$, and the variational estimate
$\|q_j\|_{H^1}\le C_R\|g_j\|_{H^{-1}}$ imply, for any
$\varepsilon>0$,
\begin{equation*}
\|q_j\|_{C^1}
\le \varepsilon\|q_j\|_{H^{11/4}}
+C_{R,\varepsilon}\|g_j\|_{H^{-1}}.
\end{equation*}
After absorption and \eqref{eq:gj-all},
\begin{equation}\label{eq:qj-C1}
\|q_j\|_{C^1}
\le C_R\|g_j\|_{H^{3/4}}
\le C_R2^{(3/4-s)j}c_j.
\end{equation}
Applying \eqref{eq:balanced-dir} with
$(\tau,r,\alpha)=(5,0,1)$ gives
\begin{equation*}
\|q_j\|_{H^5}
\le C_R\Bigl(
\|g_j\|_{H^3}
+(1+\|\Ga_j\|_{H^{9/2}})\|q_j\|_{C^1}\Bigr)\le C_R\Bigl(
2^{(3-s)j}c_j
+(1+2^{(9/2-s)j}c_j)2^{(3/4-s)j}c_j\Bigr).
\end{equation*}
The exponent of the product of the two high frequency factors is
\begin{equation*}
\left(\frac92-s\right)+\left(\frac34-s\right)
=3-s+\left(\frac94-s\right)<3-s,
\end{equation*}
because $s>5/2$.  Since $c_j$ is bounded,
\begin{equation}
\|q_j\|_{H^5}\le C_R2^{(3-s)j}c_j,
\qquad
\|\nabla q_j\|_{H^4}\le C_R2^{(3-s)j}c_j.              \label{eq:qj-H5}
\end{equation}
At order $s+1$, \eqref{eq:gj-all} and \eqref{eq:eta-sm-high} give
$\|g_j\|_{H^{s-1}}\le C2^{-j}c_j$ and
$\|\Ga_j\|_{H^{s+1/2}}\le C_R+C2^{j/2}c_j$. Thus
\eqref{eq:balanced-dir} and \eqref{eq:qj-C1} yield
\begin{equation}                   \label{eq:qj-Hs}
\|\nabla q_j\|_{H^s}
\le C_R\Bigl(
2^{-j}c_j+(1+2^{j/2}c_j)2^{(3/4-s)j}c_j\Bigr)\le C_R2^{-j}c_j.    
\end{equation}
Here the second high frequency product has exponent $5/4-s<-1$
and is therefore bounded by $C_R2^{-j}c_j$.
The variational estimate, together with
$\|g_j\|_{H^{-1}}\le C\|g_j\|_2$, also gives
\begin{equation}
\|\nabla q_j\|_2\le C_R\|g_j\|_2
\le C_R2^{-sj}c_j.                       \label{eq:qj-L2}
\end{equation}
Set
\begin{equation*}
u_j:=K_ju-\nabla q_j^u,
\qquad \widetilde B_j:=K_jB-\nabla q_j^B.
\end{equation*}
Combining \eqref{eq:kernel-Hs}--\eqref{eq:qj-Hs} gives
\begin{equation}
\|\eta_j\|_{H^m}+\|u_j\|_{H^m(\Om_j)}
+\|\widetilde B_j\|_{H^m(\Om_j)}
\le C_R\bigl(1+2^{(m-s)j}c_j\bigr),
\qquad m\in\{s,4\}.                     \label{eq:stdregbound}
\end{equation}

\par\smallskip
\noindent\textit{Step 3. Quantitative normal defect of the magnetic
field.}\par\nobreak

We next estimate the normal component of the solenoidal field
$\widetilde B_j$ on the new boundary. Set
\begin{equation*}
f_j:=\widetilde B_j\cdot N_j\quad\hbox{on }\Ga_j.
\end{equation*}
The defect satisfies
\begin{equation}
\|f_j\|_{H^r(\Ga_j)}
\le C_R2^{(r-s+1/2)j}c_j,
\qquad -\frac12\le r\le\frac72.          \label{eq:normal-defect}
\end{equation}
We prove the estimates at $r=-1/2$ and $r=7/2$ and then interpolate.

We first record a trace estimate that keeps the frequency envelope.
Let $F^e=E_\Om B$ and let $\Sigma$ denote either $\Ga$ or $\Ga_j$.
For a field $G$ defined near $\Sigma$, denote its trace pulled back to
$\Ga_*$ by
\begin{equation*}
(\mathcal T_\Sigma G)(y)
=G\bigl(y+\eta_\Sigma(y)X_*(y)\bigr),\qquad y\in\Ga_*.
\end{equation*}
For $G=F^e$ or $G=K_\ell F^e$, with any regularization index $\ell$,
we claim that
\begin{equation}
\|P_k\mathcal T_\Sigma G\|_2
\le C_R2^{-(s-1/2)k}c_k,\qquad k\ge0.
\label{eq:trace-envelope}
\end{equation}
To prove this, split the ambient input as
$F^e=L_kF^e+(\mathrm{Id}-L_k)F^e$ and apply the same operator, either the
identity or $K_\ell$, to both terms. Denote the corresponding outputs
by $G_{\rm lo}$ and $G_{\rm hi}$. The kernels in
\eqref{eq:scalar-kernel} satisfy
$\|\kappa_{a,\ell}\|_{L^1}=\|\kappa_a\|_{L^1}$.
Convolution commutes with $(1-\Delta)^{r/2}$ and with first derivatives.
Young's inequality therefore bounds each convolution on $H^r$ and on
$C^1$, uniformly in $\ell$. Multiplication by the fixed smooth cutoffs
$\chi_a$ preserves these bounds for $K_\ell$. The low frequency contribution therefore satisfies
\begin{equation*}
\|G_{\rm lo}\|_{H^{s+1/2}(\R^3)}
\le C_R2^{k/2}c_k,
\qquad \|G_{\rm lo}\|_{C^1}\le C_R.
\end{equation*}
Apply the balanced trace estimate \eqref{eq:balanced-trace} with
$\tau=s+1/2$, $r=0$, and $\alpha=1$. Since
$\|\Sigma\|_{H^s}\le C_R$, it gives
\begin{equation*}
\|\mathcal T_\Sigma G_{\rm lo}\|_{H^s}
\le C_R(2^{k/2}c_k+1)
\le C_R2^{k/2}c_k.
\end{equation*}
The last inequality uses $c_k\ge2^{-\sigma k}$ and $\sigma<1/2$.
Applying $P_k$ and using $\|P_kh\|_2\le C2^{-sk}\|h\|_{H^s}$ gives
$\|P_k\mathcal T_\Sigma G_{\rm lo}\|_2
\le C_R2^{-(s-1/2)k}c_k$.
For the high frequency contribution, the envelope and slow variation
give
\begin{equation*}
\|G_{\rm hi}\|_2\le C2^{-sk}c_k,
\qquad
\|G_{\rm hi}\|_{H^1}\le C2^{-(s-1)k}c_k.
\end{equation*}
The multiplicative trace inequality on the uniformly Lipschitz domains
then yields
\begin{equation*}
\|\mathcal T_\Sigma G_{\rm hi}\|_2
\le C_R\|G_{\rm hi}\|_2^{1/2}
\|G_{\rm hi}\|_{H^1}^{1/2}
\le C_R2^{-(s-1/2)k}c_k.
\end{equation*}
This proves \eqref{eq:trace-envelope}. The argument uses only the
standard $H^s$ bounds for the field and the surface.

Let $Y_*(y)$ be a smooth local parametrization of $\Ga_*$, with
oriented coordinates $y=(y_1,y_2)$. Writing $X_*(y)$ for the pullback
of the fixed transverse field and $\partial_a=\partial/\partial y_a$, set
\begin{equation*}
G_\eta(y)=Y_*(y)+\eta(y)X_*(y),\qquad
\nu(\eta)=\frac{\partial_1G_\eta\times\partial_2G_\eta}
{\left|\partial_1Y_*\times\partial_2Y_*\right|}.
\end{equation*}
The denominator is the reference area density. Thus $\nu(\eta)$ is
an ambient vector field along $\Ga_*$, independent of the chosen
oriented coordinates.
Choose the orientation so that $\nu(\eta)\cdot(N\circ G_\eta)>0$. The collar bounds
keep $|\nu(\eta)|$ and its reciprocal uniformly bounded. The vector
$\nu(\eta)$ is affine in $\nabla\eta$, with coefficients smooth in
$y$ and polynomial of degree at most two in $\eta$. Expanding the cross product expresses it as a sum of
$1,\eta,\eta^2,\partial_a\eta$, and $\partial_a(\eta^2)$,
$a=1,2$, multiplied by fixed smooth vector fields. Set $\nu=\nu(\eta)$,
$\nu_j=\nu(\eta_j)$,
$\mathcal T_j=\mathcal T_{\Ga_j}$, $\mathcal T=\mathcal T_\Ga$, and
$\widetilde f_j=(\mathcal T_j\widetilde B_j)\cdot\nu_j$.
Thus $f_j\circ G_{\eta_j}=|\nu_j|^{-1}\widetilde f_j$.
In the trace estimates, we identify functions on $\Ga_j$ with their
pullbacks to the reference surface in the fixed surface coordinates.

Because $B$ is tangent to the original boundary, we have
\begin{equation*}
\widetilde f_j=(\mathcal T_j\widetilde B_j)\cdot\nu_j
-L_j\bigl((\mathcal TB)\cdot\nu\bigr).
\end{equation*}

To separate the trace error from the surface products, set
$g=\mathcal TB$.
The trace envelope estimate \eqref{eq:trace-envelope} gives
\begin{equation*}
\|P_kg\|_2\le C_R2^{-(s-1/2)k}c_k,
\qquad
\|P_k\nu\|_2\le C_R2^{-(s-1)k}c_k.
\end{equation*}
The second estimate follows from the polynomial representation for
$\nu$ and the surface product decomposition. Fix
$1/2<\beta<\min\{1,s-2\}$, so that
$\|g\|_{C^1}+\|\nu\|_{C^\beta}\le C_R$.
Define
\begin{equation*}
b_j=\mathcal T_j\widetilde B_j-L_jg,\qquad
r_j=(L_jg)\cdot(L_j\nu)-L_j(g\cdot\nu),\qquad
e_j=\nu_j-L_j\nu.
\end{equation*}
Since $g\cdot\nu=0$, the following identity is exact:
\begin{equation}\label{eq:upper-defect-decomposition}
\widetilde f_j=b_j\cdot\nu_j+r_j+(L_jg)\cdot e_j.
\end{equation}
We estimate the three terms separately.

Fix an integer $L>s+5$. We first establish the spectral remainder
estimate used in the product bounds below.
For a fixed smooth coefficient $A$, fixed coordinate differential
operators $D_1,\ldots,D_p$, and $1\le p\le3$, put
\begin{equation*}
 V_j=A\prod_{\iota=1}^p D_\iota(L_jF_\iota),\qquad F_\iota\in L^2(\Ga_*).
\end{equation*}
There is a fixed integer $C_0$ such that, for all $a,b\ge0$,
\begin{equation}\label{eq:surface-separated-tail}
 \|P_kV_j\|_2
 \le C_{a,b}2^{-aj}2^{-b(k-j)}
             \prod_{\iota=1}^p\|F_\iota\|_2,
 \qquad k\ge j+C_0.
\end{equation}
The constants may depend on $A$, the $D_\iota$, and the fixed surface.
To prove this, set $h=2^{-j}$, $h_k=2^{-k}$, and
$\varepsilon=h_k/h$. Smooth functional calculus on $\Ga_*$
\cite[Theorem~14.9 and its proof]{Zworski2012} gives local kernels
\begin{equation*}
 (2\pi h)^{-2}\int e^{i(x-y)\cdot\xi/h}
                   a_h(x,y,\xi)\,d\xi
\end{equation*}
for $L_j$, modulo remainders $O(h^T)$ in each fixed kernel
$C^m$ norm, for arbitrary $T,m$. The amplitudes have fixed compact support in $(x,y,\xi)$,
with $|\xi|\le R_0$, and uniformly bounded derivatives. Indeed the
resolvent symbol expansion, integrated against an almost-analytic
extension of a smooth compactly supported continuation of
$\lambda\mapsto\chi(\sqrt\lambda)$ from $[0,\infty)$ to $\R$, has coefficients
supported where $|\xi|_g\le2$; truncating it at arbitrarily high order
gives these remainder bounds. The corresponding kernels of $P_k$
have scale $h_k$ and frequency support $c\le|\zeta|\le R_0$, since
$\chi(\sqrt\lambda)-\chi(2\sqrt\lambda)$ vanishes near zero.
All coordinate cutoffs, volume densities, and $A$ enter as fixed
smooth amplitude factors.

Insert these $p$ input kernels and the output kernel. In a common
coordinate chart the oscillatory factor is $e^{i\Phi/h_k}$, with prefactor
$h_k^{-2}h^{-2p}$ and
\begin{equation*}
 \Phi=(x-w)\cdot\zeta
       +\varepsilon\sum_{\iota=1}^p(w-y_\iota)\cdot\xi_\iota.
\end{equation*}
Choose $C_0$ so that $3R_02^{-C_0}\le c/2$. Then
$|\nabla_w\Phi|=|-\zeta+\varepsilon\sum_\iota\xi_\iota|\ge c/2$.
Integrating $M$ times with
$(h_k/i)|\nabla_w\Phi|^{-2}\nabla_w\Phi\cdot\nabla_w$
gives, if $d_0=\sum_\iota\operatorname{ord}D_\iota$,
\begin{equation*}
 C_Mh^{-2p-d_0}h_k^{M-2}
   =C_Mh^{M-2(p+1)-d_0}\varepsilon^{M-2}
\end{equation*}
as a bound for the multilinear kernel. All remaining integrations
are over fixed compact sets; hence this also bounds its
$(L^2)^p\to L^2$ norm. Taking $M$ large proves
\eqref{eq:surface-separated-tail} for the oscillatory kernels.

An input remainder $\mathcal R_h$ can be made
$O(h^T)$ as an operator $L^2\to W^{m,\infty}$ for any fixed $T,m$.
Choose an integer $m\ge b$. The inequalities
\begin{equation*}
 \|L_jF\|_{W^{m,\infty}}
 \le C\|L_jF\|_{H^{m+2}}\le Ch^{-m-2}\|F\|_2,
 \qquad \|P_k v\|_2\le Ch_k^m\|v\|_{H^m}
\end{equation*}
bound every term with an input remainder by
$C h^{T-(p-1)(m+2)-d_0+m}\varepsilon^m
\prod_\iota\|F_\iota\|_2$.
An output remainder with kernel $O(h_k^T)$ contributes at most
$C h^{T-d_0-2(p-2)_+}\varepsilon^T
\prod_\iota\|F_\iota\|_2$: use two $L^2$ factors and, when $p=3$,
the $H^2\hookrightarrow L^\infty$ bound for the third.
Here $(p-2)_+=\max\{p-2,0\}$; for $p=1$ use the bounded volume of
the chart. Increasing $T$ proves the same estimate for all remainders.
Finite coordinate sums and off-diagonal kernels obey these bounds as well.

For any fixed $r\ge0$ and $L>0$, the spectral square function estimate
and \eqref{eq:surface-separated-tail}, with $a>L+r$ and $b>r$, give
\begin{equation*}
 \|(\mathrm{Id}-L_{j+C_0+2})V_j\|_{H^r}^2
 \le C2^{-2(a-r)j}
       \sum_{k\ge j+C_0}2^{-2(b-r)(k-j)}
                         \prod_\iota\|F_\iota\|_2^2
 \le C2^{-2Lj}\prod_\iota\|F_\iota\|_2^2.
\end{equation*}
For the low--low terms, choose a larger fixed integer $C_1$ and set
$\ell=j-C_1\ge0$, $F_{\rm lo}=L_\ell F$,
$G_{\rm lo}=L_\ell G$. Both inputs then satisfy
$L_jF_{\rm lo}=F_{\rm lo}$ and $L_jG_{\rm lo}=G_{\rm lo}$.
Applying \eqref{eq:surface-separated-tail} at scale $\ell$ and summing
over the output blocks $k\ge j-1$ gives
\begin{equation*}
 \|(\mathrm{Id}-L_j)(AF_{\rm lo}G_{\rm lo})\|_{H^r}
 \le C_{L,r}2^{-Lj}\|F\|_2\|G\|_2.
\end{equation*}
This controls the contribution where both input frequencies are below
$j-C_1$, including fixed coordinate derivatives. The finitely many
smaller $j$ are absorbed into the constant.

The complementary separated interaction has one annular input $P_kF$,
the other input $L_mG$ with $m\le k-C_0$, and output $L_j$ with
$j\le k-C_0$. Put $h_m=2^{-m}$. Its oscillatory factor is
$e^{i\Psi/h_k}$, where
\begin{equation*}
\Psi=\frac{h_k}{h}(x-w)\cdot\zeta
+(w-y)\cdot\xi+\frac{h_k}{h_m}(w-z)\cdot\theta.
\end{equation*}
Here $|\xi|\ge c$, while the other covectors are bounded, so
$|\nabla_w\Psi|\ge c/2$ after increasing $C_0$.
The same integration gives a kernel bound
$C_Mh^{-2}h_k^{M-2}h_m^{-2}\le C_Mh_k^{M-6}$.
Derivatives of the output cost fixed powers of $h^{-1}\le h_k^{-1}$.
For a remainder in the output or low input, place its smooth kernel
in the amplitude and keep the annular input phase; integration in $w$
gives the same arbitrary power of $h_k$. A remainder in the annular
input already has that decay. Consequently this interaction is bounded
by $C_{N,r}2^{-Nk}\|F\|_2\|G\|_2$ in $H^r$, for every $N$.
Summing over $k\ge j+C_0$ gives $C_{L,r}2^{-Lj}\|F\|_2\|G\|_2$.
All these statements apply componentwise to vector fields.
Since $c_j\ge2^{-\sigma j}$,
$2^{-Lj}\le2^{-(L-\sigma)j}c_j$. Our choice $L>s+5$ absorbs the
remainders in every estimate below, including after the required derivative.

We now split $r_j$ into low--high, high--low, and comparable-frequency
interactions, taking the separation between the low and high inputs
larger than the fixed width of the transition bands.
If the high input has index $k=j+O(1)$, the low input is
unchanged by $L_j$, and the corresponding difference is a commutator.
The kernel formula with the factor $a(y)-a(x)$ gives
\begin{equation*}
\|[L_j,a]v\|_2\le C2^{-j}\|a\|_{C^1}\|v\|_2,
\qquad
\|[L_j,a]v\|_2\le C2^{-\beta j}\|a\|_{C^\beta}\|v\|_2.
\end{equation*}
If $k<j-C$, both inputs are below the cutoff, and their difference
is the low--low remainder just estimated. If $k>j+C$, $L_j$
annihilates the high input, and the cutoff of the low--high product
is the complementary separated interaction. Thus only a fixed number
of bands around $j$ enters the two commutator sums. For comparable
inputs, use the surface Bernstein bound
$\|P_kv\|_\infty\le C2^k\|P_kv\|_2$.
Taking $C$ large enough to include all fixed overlaps, we obtain
\begin{align*}
\|r_j\|_2
\le{}& C_R2^{-j}\sum_{|k-j|\le C}\|P_k\nu\|_2
+C_R2^{-\beta j}\sum_{|k-j|\le C}\|P_kg\|_2\\
&+C\sum_{k\ge j-C}2^k\|P_kg\|_2
       \sum_{|\ell-k|\le C}\|P_\ell\nu\|_2
+C_{L,R}2^{-Lj}\\
\le{}& C_R\Bigl(2^{-sj}c_j+2^{-(s-1/2+\beta)j}c_j
+\sum_{k\ge j-C}2^{-(2s-5/2)k}c_k^2\Bigr)
+C_{L,R}2^{-Lj}
\le C_R2^{-sj}c_j.
\end{align*}
The last line uses slow variation, $\sup_kc_k\le C_R$,
$\beta>1/2$, and $2s-5/2>s$.
In $r_j$, the term $L_j(g\cdot\nu)$ has spectrum below $2^{j+1}$,
and \eqref{eq:surface-separated-tail} controls the upper tail of
$(L_jg)\cdot(L_j\nu)$. Bernstein's inequality on the remaining
output bands therefore yields
\begin{equation*}
\|r_j\|_{H^{7/2}}\le C_R2^{(7/2-s)j}c_j.
\end{equation*}

For the normal commutator, the polynomial representation above gives
\begin{equation*}
\nu(\eta)=A_0+\mathcal Q_1\eta+\mathcal Q_2(\eta^2),
\end{equation*}
where $A_0$ is a fixed smooth vector field and $\mathcal Q_1,\mathcal Q_2$
are fixed smooth vector-valued differential operators of order at most
one on $\Ga_*$. The fixed coordinate partition is included in their
coefficients. With $w_j=(L_j\eta)^2-L_j(\eta^2)$, we have the exact identity
\begin{equation*}
e_j=(\mathrm{Id}-L_j)A_0
+[\mathcal Q_1,L_j]\eta+[\mathcal Q_2,L_j](\eta^2)
+\mathcal Q_2w_j.
\end{equation*}
For each fixed smooth first order operator $\mathcal Q$, the symbol of
$[\mathcal Q,L_j]$ is uniformly of order zero on the transition bands.
Away from those bands, the symbol of $L_j$ is constant, so the
commutator is a separated-frequency remainder. Thus
\begin{equation*}
\|[\mathcal Q,L_j]F\|_2
\le C_{\mathcal Q}\sum_{|k-j|\le C}\|P_kF\|_2
+C_{L,\mathcal Q}2^{-Lj}\|F\|_2.
\end{equation*}
The product decomposition shows that $\eta^2$ has the same $H^s$
envelope bound as $\eta$: low--high terms use $\|\eta\|_\infty$,
and comparable-frequency terms are bounded by
$C\sum_{k\ge m-C}2^k\|P_k\eta\|_2
\sum_{|\ell-k|\le C}\|P_\ell\eta\|_2
\le C_R2^{-sm}c_m$ in the output band $m$.
The two commutators therefore have $L^2$ norm at most $C_R2^{-sj}c_j$.
Applying the preceding low--high decomposition to $w_j$ gives
\begin{align*}
\|w_j\|_2
\le{}& C_R2^{-j}\sum_{|k-j|\le C}\|P_k\eta\|_2
+C\sum_{k\ge j-C}2^k\|P_k\eta\|_2
       \sum_{|\ell-k|\le C}\|P_\ell\eta\|_2
+C_{L,R}2^{-Lj}\\
\le{}& C_R2^{-(s+1)j}c_j,
\end{align*}
since $2s-1>s+1$. Its upper spectral tail is controlled by
\eqref{eq:surface-separated-tail}, so one derivative costs $C2^j$
on the remaining bands. The same tail estimate applies to
$\mathcal QL_jF$, while $L_j\mathcal QF$ is spectrally localized.
The fixed smooth term $(\mathrm{Id}-L_j)A_0$ decays to arbitrary
order in every fixed Sobolev norm. Combining these estimates gives
\begin{equation*}
\|e_j\|_{H^r}\le C_R2^{(r-s)j}c_j,\qquad r\in\{0,7/2\}.
\end{equation*}
This is the analogue, in the present collar coordinates, of the normal
comparison in
\cite[Appendix~A.6.1, proof of Proposition~A.37, step~(e)]{IPTT-MHD}.
Finally, $\|L_jg\|_\infty\le C_R$ gives
$\|(L_jg)\cdot e_j\|_2\le C_R2^{-sj}c_j$.
The identity
\begin{equation*}
e_j=A_0+\mathcal Q_1L_j\eta
+\mathcal Q_2((L_j\eta)^2)-L_j\nu
\end{equation*}
expresses this product using at most three cutoff factors and one derivative, so
\eqref{eq:surface-separated-tail} and Bernstein's inequality yield
\begin{equation*}
\|(L_jg)\cdot e_j\|_{H^{7/2}}
\le C_R2^{(7/2-s)j}c_j.
\end{equation*}

We next estimate $b_j\cdot\nu_j$ in $H^{-1/2}$. Set
$\ell=\lfloor3j/2\rfloor$ and
\begin{equation*}
G=K_\ell F^e,\qquad
B'=G-\nabla q',\qquad
\Delta q'=\diver G\ \hbox{in }\Om_j,\qquad q'|_{\Ga_j}=0.
\end{equation*}
By \eqref{eq:eta-tail-Linf}, the distance between $\Ga$ and $\Ga_j$
is $o(2^{-\ell})$. The finer kernel is therefore defined along the
transverse segments joining these surfaces and samples only the
original domain. The proof of \eqref{eq:gj-all}, applied at scale
$\ell$ with target domain $\Om_j$, gives
$\|\diver G\|_2\le C_R2^{-s\ell}c_\ell$.
The Dirichlet estimate \eqref{eq:dir2} gives
$\|\nabla q'\|_{H^1}\le C_R\|\diver G\|_2$.
The uniform trace map $H^1(\Om_j)\to H^{1/2}(\Ga_j)$,
followed by $H^{1/2}(\Ga_j)\hookrightarrow L^2(\Ga_j)$, then gives
\begin{equation}\label{eq:qprime-trace}
\|\nabla q'\|_{H^1(\Om_j)}
+\|\mathcal T_j\nabla q'\|_2
\le C_R2^{-s\ell}c_\ell.
\end{equation}

The same moment estimates for the kernels, now applied to
$\mathrm{Id}-K_\ell$, give approximation bounds
$C2^{-(s-r)\ell}c_\ell$ in $H^r$ for $r=0,1$. A fixed cutoff equal
to one on the collar may be inserted when taking ambient norms.
The multiplicative trace inequality consequently gives
\begin{equation*}
\|\mathcal T(G-B)\|_2
\le C_R2^{-(s-1/2)\ell}c_\ell
\le C_R2^{-sj}c_j.
\end{equation*}
In the last step,
$(s-1/2)\ell-\sigma(\ell-j)-sj
=((s-3/2-\sigma)/2)j+O(1)>0$.
Moreover, $\|G\|_{C^1}\le C_R$, and integration along transverse
segments gives
\begin{equation*}
\|\mathcal T_jG-\mathcal TG\|_2
\le C_R\|\eta_j-\eta\|_2
\le C_R2^{-sj}c_j.
\end{equation*}
Together with \eqref{eq:qprime-trace}, these estimates show that
\begin{equation}\label{eq:trace-replacement}
\|\mathcal T_jB'-g\|_2\le C_R2^{-sj}c_j.
\end{equation}

Since $b_j=\mathcal T_j\widetilde B_j-L_jg$, we have the exact
decomposition
\begin{equation}\label{eq:fhat-decomp}
b_j\cdot\nu_j
=\bigl(\mathcal T_j(\widetilde B_j-B')\bigr)\cdot\nu_j
+\bigl((\mathrm{Id}-L_j)\mathcal T_jB'\bigr)\cdot\nu_j
+\bigl(L_j(\mathcal T_jB'-g)\bigr)\cdot\nu_j.
\end{equation}
The last term is bounded by $C_R2^{-sj}c_j$ in $L^2$ by
\eqref{eq:trace-replacement}.

For the first term, both fields are divergence-free on $\Om_j$.
The kernel approximation estimates and the Dirichlet corrections give
\begin{equation*}
\|\widetilde B_j-B'\|_{L^2(\Om_j)}
\le C_R2^{-sj}c_j.
\end{equation*}
For $\varphi\in H^{1/2}(\Ga_j)$, let
$\mathcal E_j\varphi\in H^1(\Om_j)$ be a uniformly bounded extension.
Green's formula then gives
\begin{equation*}
\left|\left\langle(\widetilde B_j-B')\cdot N_j,
\varphi\right\rangle\right|
=\left|\int_{\Om_j}(\widetilde B_j-B')\cdot
\nabla\mathcal E_j\varphi\,dx\right|\le C_R2^{-sj}c_j\|\varphi\|_{H^{1/2}(\Ga_j)}.
\end{equation*}
Thus
\begin{equation*}
\|(\widetilde B_j-B')\cdot N_j\|_{H^{-1/2}(\Ga_j)}
\le C_R2^{-sj}c_j.
\end{equation*}
Multiplication by $|\nu_j|$ preserves this estimate, because
$|\nu_j|$ is uniformly bounded in $C^\beta$ with $\beta>1/2$.

For the second term in \eqref{eq:fhat-decomp}, the proved trace
envelope estimate \eqref{eq:trace-envelope} for $G$ gives
\begin{align*}
\|(\mathrm{Id}-L_j)\mathcal T_jG\|_{H^{-1/2}}^2
&\le C\sum_{k>j-O(1)}2^{-k}
\|P_k\mathcal T_jG\|_2^2\\
&\le C_R2^{-2sj}\sum_{k>j-O(1)}
2^{-2s(k-j)}c_k^2\\
&\le C_R2^{-2sj}c_j^2
\sum_{k>j-O(1)}2^{-2(s-\sigma)(k-j)}
\le C_R2^{-2sj}c_j^2.
\end{align*}
For its divergence correction, \eqref{eq:qprime-trace} yields
\begin{equation*}
\|(\mathrm{Id}-L_j)\mathcal T_j\nabla q'\|_{H^{-1/2}}
\le C2^{-j/2}\|\mathcal T_j\nabla q'\|_2
\le C_R2^{-j/2-s\ell}c_\ell
\le C_R2^{-sj}c_j.
\end{equation*}
Consequently,
\begin{equation*}
\|(\mathrm{Id}-L_j)\mathcal T_jB'\|_{H^{-1/2}}
\le C_R2^{-sj}c_j.
\end{equation*}
Multiplication by $\nu_j$ is bounded on $H^{-1/2}$, so
\eqref{eq:fhat-decomp} gives
$\|b_j\cdot\nu_j\|_{H^{-1/2}}\le C_R2^{-sj}c_j$.
Combining this with the $L^2$ estimates for $r_j$ and $e_j$ in
\eqref{eq:upper-defect-decomposition} gives
\begin{equation}\label{eq:fhat-neg}
\|\widetilde f_j\|_{H^{-1/2}}\le C_R2^{-sj}c_j.
\end{equation}

At the upper endpoint, \eqref{eq:balanced-trace} at order $4$, together
with \eqref{eq:eta-sm-high} and \eqref{eq:stdregbound}, yields
\begin{align*}
\|\mathcal T_j\widetilde B_j\|_{H^{7/2}}
&\le C_R\bigl(\|\widetilde B_j\|_{H^4}
+(1+\|\eta_j\|_{H^{7/2}})\|\widetilde B_j\|_{C^1}\bigr)
\le C_R2^{(4-s)j}c_j.
\end{align*}
The trace envelope and slow variation give
\begin{equation*}
\|L_jg\|_{H^{7/2}}^2
\le C_R\sum_{k\le j+O(1)}2^{(8-2s)k}c_k^2
\le C_R2^{(8-2s)j}c_j^2
\sum_{k\le j+O(1)}2^{-2(4-s-\sigma)(j-k)}
\le C_R2^{(8-2s)j}c_j^2,
\end{equation*}
since $4-s>\sigma$. Thus
\begin{equation*}
\|b_j\|_{H^{7/2}}\le C_R2^{(4-s)j}c_j.
\end{equation*}
To estimate its supremum norm, write
\begin{align*}
b_j={}&(\mathcal T_j-\mathcal T)K_jB
+\mathcal T(K_jB-B)
+(\mathrm{Id}-L_j)g-\mathcal T_j\nabla q_j^B.
\end{align*}
The first term is bounded by
$C_R\|\eta_j-\eta\|_\infty$ using the uniform Lipschitz bound for the
regularized field. Kernel approximation in $L^\infty$,
\eqref{eq:trace-envelope}, and \eqref{eq:qj-C1} treat the other terms. Hence
\begin{equation*}
\|b_j\|_\infty\le C_R2^{-(s-3/2)j}c_j.
\end{equation*}
Finally,
\begin{equation*}
\|b_j\cdot\nu_j\|_{H^{7/2}}
\le C_R\|b_j\|_{H^{7/2}}
+C_R(1+\|\eta_j\|_{H^{9/2}})\|b_j\|_\infty\le C_R\bigl(2^{(4-s)j}c_j
+2^{(6-2s)j}c_j^2\bigr)
\le C_R2^{(4-s)j}c_j,
\end{equation*}
since $6-2s<4-s$ for $s>2$. Combining these bounds in
\eqref{eq:upper-defect-decomposition} proves
\begin{equation}\label{eq:fhat-high}
\|\widetilde f_j\|_{H^{7/2}}\le C_R2^{(4-s)j}c_j.
\end{equation}

Real interpolation between \eqref{eq:fhat-neg} and
\eqref{eq:fhat-high} gives
\begin{equation}
\|\widetilde f_j\|_{H^r}
\le C_R2^{(r-s+1/2)j}c_j,
\qquad -\frac12\le r\le\frac72.          \label{eq:fhat-all}
\end{equation}
It remains to pass from the unnormalized to the unit normal.  At the
negative endpoint multiplication by $|\nu_j|^{-1}$ is bounded on
$H^{-1/2}$ because $|\nu_j|^{-1}\in C^\beta$ uniformly, with
$1/2<\beta<\min\{1,s-2\}$.  At
the upper endpoint the tame surface product estimate gives
\begin{align}
\|f_j\|_{H^{7/2}}
&\le C_R\|\widetilde f_j\|_{H^{7/2}}
+C_R\||\nu_j|^{-1}\|_{H^{7/2}}
\|\widetilde f_j\|_{L^\infty}.       \label{eq:unit-normal-tame}
\end{align}
Because $\nu_j$ depends smoothly on the fixed coordinates, $\eta_j$,
and $\nabla\eta_j$,
\begin{equation*}
\||\nu_j|^{-1}\|_{H^{7/2}}
\le C_R\bigl(1+2^{(9/2-s)j}c_j\bigr).
\end{equation*}
Taking $r=s-3/2>1$ in \eqref{eq:fhat-all} and using the two-dimensional
embedding $H^{s-3/2}\hookrightarrow L^\infty$ gives
\begin{equation*}
\|\widetilde f_j\|_{L^\infty}\le C_R2^{-j}c_j.
\end{equation*}
Thus the second term in \eqref{eq:unit-normal-tame} is bounded by
\begin{equation*}
C_R\left(2^{-j}c_j+2^{(7/2-s)j}c_j^2\right)
\le C_R2^{(4-s)j}c_j.
\end{equation*}
Interpolation proves \eqref{eq:normal-defect} for $f_j$.

\par\smallskip
\noindent\textit{Step 4. Harmonic Neumann correction and magnetic
tangency.}\par\nobreak

By the divergence theorem, $\int_{\Ga_j}f_j\,dS=0$. Hence the
following harmonic Neumann problem has a unique solution:
\begin{equation*}
\Delta\psi_j=0\quad\hbox{in }\Om_j,
\qquad \partial_{N_j}\psi_j=f_j\quad\hbox{on }\Ga_j,
\qquad \int_{\Ga_j}\psi_j\,dS=0.
\end{equation*}
Set
\begin{equation*}
B_j:=\widetilde B_j-\nabla\psi_j.
\end{equation*}
Then
\begin{equation*}
\diver B_j=0\quad\hbox{in }\Om_j,
\qquad B_j\cdot N_j=0\quad\hbox{on }\Ga_j.
\end{equation*}
For $m=s$ and $m=4$, \cref{lem:elliptic-estimates}\textup{(ii)} and
\eqref{eq:normal-defect} give
\begin{align*}
\|\nabla\psi_j\|_{H^m}
&\le C_R\Bigl(
\|f_j\|_{H^{m-1/2}}
+(1+\|\Ga_j\|_{H^{m+3/4}})
\|f_j\|_{H^{s-3/2}}\Bigr)\\
&\le C_R\Bigl(2^{(m-s)j}c_j
+(1+2^{(m+3/4-s)j}c_j)2^{-j}c_j\Bigr).
\end{align*}
The product term has exponent $m-s-1/4$, so
\begin{equation}
\|\nabla\psi_j\|_{H^m(\Om_j)}
\le C_R2^{(m-s)j}c_j,
\qquad m=s,4.                              \label{eq:Bcorr-high}
\end{equation}
The variational Neumann estimate and the negative endpoint of
\eqref{eq:normal-defect} give
\begin{equation}
\|\nabla\psi_j\|_{L^2(\Om_j)}
\le C_R\|f_j\|_{H^{-1/2}(\Ga_j)}
\le C_R2^{-sj}c_j.                         \label{eq:Bcorr-L2}
\end{equation}
Taking $m=s$ in \eqref{eq:Bcorr-high} gives
$\|\nabla\psi_j\|_{H^s}\le C_Rc_j\to0$.

\par\smallskip
\noindent\textit{Step 5. Consecutive estimates and strong convergence.}\par\nobreak

For either Dirichlet correction $q_j=q_j^u$ or $q_j^B$,
\eqref{eq:qj-L2}, \eqref{eq:Bcorr-L2}, the triangle inequality,
and slow variation give
\begin{equation*}
\|\nabla q_{j+1}-\nabla q_j\|_{L^2(\Om_j\cap\Om_{j+1})}
+\|\nabla\psi_{j+1}-\nabla\psi_j\|_{L^2(\Om_j\cap\Om_{j+1})}
\le C_R2^{-sj}c_j.
\end{equation*}
Together with \eqref{eq:eta-increment-L2} and \eqref{eq:kernel-inc},
this proves \eqref{eq:reglow}.

The uniform $H^s$ bound follows from \eqref{eq:stdregbound} and
\eqref{eq:Bcorr-high}. The scalar regularizations converge strongly in the
moving domain $H^s$ topology, while \eqref{eq:qj-Hs} and
\eqref{eq:Bcorr-high} make both corrections tend to zero since
$c_j\to0$. Thus $S_jU_0\to U_0$ in $\X^s$.
The fourth order bounds \eqref{eq:eta-sm-high}, \eqref{eq:kernel-H4},
\eqref{eq:qj-H5}, and \eqref{eq:Bcorr-high} give \eqref{eq:reghigh}.

\par\smallskip
\noindent\textit{Step 6. Uniform geometric bounds and positivity of the Taylor coefficient.}\par\nobreak

Fix $s'$ with $5/2<s'<s$. For $F=u,B$, write
$F^e=E_{\Om}F$ and let $F_j$ denote the corresponding regularized field.
The scalar kernel estimate and interpolation between
\eqref{eq:qj-L2} and \eqref{eq:qj-Hs}, and between
\eqref{eq:Bcorr-L2} and \eqref{eq:Bcorr-high}, give
\begin{equation*}
\|\eta_j-\eta\|_{H^{s'}}
+\sum_{F=u,B}\|F_j-F^e\|_{H^{s'}(\Om_j)}
\le C_R2^{-(s-s')j}.
\end{equation*}
For the maps $\Xi_j:\Om_*\to\Om_j$ and $\Xi:\Om_*\to\Om$
from the proof of \cref{lem:pressure-estimates}\textup{(ii)}, the graph estimate gives
$\Xi_j\to\Xi$ uniformly in $C^1$ on the local set. Using
$\|F^e\|_{C^{1,\gamma}}\le C_R$ for
$0<\gamma<\min\{1,s-5/2\}$ and $H^{s'}\hookrightarrow C^1$, we obtain
\begin{equation*}
\sum_{F=u,B}
\|F_j\circ\Xi_j-F\circ\Xi\|_{C^1(\overline{\Om_*})}
\longrightarrow0
\end{equation*}
uniformly on that set. The $C^1$ convergence of the domain maps
preserves the uniform transversality and injectivity bounds.
Let $p_j$ be the pressure associated with $(u_j,B_j,\Ga_j)$.
The pressure-continuity argument in the proof of
\cref{lem:pressure-estimates}\textup{(ii)}, applied to these pulled back
fields and domains, gives uniform convergence of the Taylor coefficients.
Consequently,
one $j_0$ works throughout the local set, with
\begin{equation*}
-\partial_{N_j}p_j\ge \frac12c_*>0.
\end{equation*}
Consecutive regularizations are uniformly close in $C^0$.
On the portion where $\eta_j<\eta_{j+1}$,
$p_{j+1}=0$ on $\Ga_{j+1}$ and the uniform $C^1$ pressure bound give
\begin{equation*}
|p_{j+1}(G_{\eta_j}(y))|
\le C_R|\eta_{j+1}(y)-\eta_j(y)|.
\end{equation*}
The analogous bound holds with the indices interchanged on the
remaining portion. The uniform surface Jacobian bounds and Taylor
lower bound control the boundary part of $\D$ by
$C_R\|\eta_{j+1}-\eta_j\|_2^2$.
Together with \eqref{eq:reglow}, this proves \eqref{eq:D0}.

\par\smallskip
\noindent\textit{Step 7. Continuity of $S_j$ for fixed $j$.}\par\nobreak

Fix $j$ and let $U_0^{(n)}\to U_0$ in $\X^s$ within the local set.
Continuity of the extension family gives
\begin{equation*}
E_{\Om_n}u_n\longrightarrow E_\Om u,
\qquad E_{\Om_n}B_n\longrightarrow E_\Om B
\quad\hbox{in }H^s(\R^3).
\end{equation*}
The fixed smooth kernels satisfy, for every finite $m\ge0$,
\begin{equation*}
\|K_jG\|_{H^m(\R^3)}
\le C_{j,m}\|G\|_{L^2(\R^3)}.
\end{equation*}
Here the local convolution formula is extended to a fixed ambient
neighborhood and multiplied by a fixed cutoff. It agrees with the
inward regularizer on each regularized domain, since every sample point
lies in the original domain. Thus the smoothed fields converge in every
finite ambient Sobolev norm. Likewise, $L_j\eta_n\to L_j\eta$
in every finite Sobolev norm on $\Ga_*$.

Extend the regularized graph displacements through the fixed collar
to obtain smooth reference-domain maps. At fixed $j$, these maps and
their inverses converge in every finite smooth norm. The pulled back
Dirichlet equations for $q_j^u$ and $q_j^B$ therefore have convergent
sources and uniformly elliptic, convergent coefficients. The variational
estimate for their difference gives $H^1_0$ convergence, and successive
elliptic estimates give convergence of the potentials, corrected fields,
and normal traces in every finite Sobolev norm.

For the Neumann problems, variational estimates on
$H^1(\Om_*)/\R$ and successive elliptic estimates give convergence of
the potentials modulo constants.
The boundary-mean normalization also converges, since both the boundary
Jacobians and the traces of the potentials converge. Thus the Neumann
corrections, and hence $S_jU_0^{(n)}$, converge in every Sobolev topology
of finite order. Constants here may depend on $j$ and the target order. This completes the proof.
\end{proof}

\begin{remark}
Smoothing the graph changes the domain and its normal. The Dirichlet
correction restores zero divergence on $\Om_j$. The harmonic Neumann
correction then restores magnetic tangency and preserves zero divergence.
Its compatibility condition is exactly
$\int_{\Ga_j}\widetilde B_j\cdot N_j\,dS=0$.
The $H^{-1/2}$ estimate in \eqref{eq:normal-defect} supplies the
$L^2$ error needed for weak stability, while the higher estimates
make the correction vanish in $H^s$.
\end{remark}

\subsection{Weak comparison along smooth solutions}

For each $j\ge j_0$, let $U_j(t)$ be the maximal smooth solution from
$S_jU_0$, provided by \cite[Theorem~1.9]{IPTT-MHD}.  Its lifespan $T_j>0$ may depend on $j$. Write
$U_j=(u_j,B_j,\Ga_j)$, $\eta_j=\eta_{\Ga_j}$ and
$W_j^\pm=u_j\pm B_j$.

\begin{proposition}\label{prop:Dprop}
\textup{(i)} Fix a bounded reference collar, $R<\infty$, and $c>0$.
There is $h=h(R,c)>0$ such that any two states $U_1,U_2$ in this
collar with $\|U_i\|_{\X^s}\le R$, Taylor coefficients at least $c$, and
\begin{equation}\label{eq:weak-band}
\|\eta_1-\eta_2\|_{C^0(\Ga_*)}\le h
\end{equation}
satisfy
\begin{equation*}
\D(U_1,U_2)\simeq_{R,c}
\|u_1-u_2\|_{L^2(\Om_1\cap\Om_2)}^2
+\|B_1-B_2\|_{L^2(\Om_1\cap\Om_2)}^2+\|\eta_1-\eta_2\|_{L^2(\Ga_*)}^2.
\end{equation*}
The width $h$ is chosen so that the transverse segments between the two
graphs remain in the fixed reference collar.

\textup{(ii)} Let $U_j,U_{j+1}$ be the smooth solutions with initial
data $S_jU_0,S_{j+1}U_0$, defined on $[0,T_*]$, where $0<T_*\le1$.
Assume that both remain in a common enlarged collar, have $\X^s$ norms
at most $R$ and Taylor coefficients at least $c>0$, and satisfy
\eqref{eq:weak-band} throughout this interval. Then
\begin{align}
\D(U_j(t),U_{j+1}(t))^{1/2}
&\le C_R2^{-sj}c_j,\label{eq:Dprop}\\
\|W_{j+1}^\pm-W_j^\pm\|_{L^2(\Om_j\cap\Om_{j+1})}
+\|\eta_{j+1}-\eta_j\|_{L^2(\Ga_*)}
&\le C_R2^{-sj}c_j,\qquad 0\le t\le T_*.
\label{eq:L2prop}
\end{align}
The constant may depend on $c$ and the fixed collar, but is independent
of $j$ and the length of this interval.
\end{proposition}

\begin{proof}
The proof is divided into three steps.

\par\nobreak\smallskip
\noindent\textit{Step 1. Transverse coordinates in the thin graph band.}\par\nobreak

Use the graph maps
\begin{equation*}
G_{\eta_i}(y)=y+\eta_i(y)X_*(y),\qquad y\in\Ga_*,\quad i=1,2.
\end{equation*}
The fixed collar assumptions imply
\begin{equation*}
0<\mu\le X_*(y)\cdot N_i(G_{\eta_i}(y))\le C_R
\end{equation*}
for both surfaces, where $\mu>0$ is a uniform transversality constant.
By \cref{lem:pressure-estimates}\textup{(i)},
\begin{equation*}
\|p_i\|_{C^{1,1/4}(\overline\Om_i)}
+\|a_i\|_{L^\infty(\Ga_i)}\le C_R,
\qquad a_i=-\partial_{N_i}p_i\ge c>0.
\end{equation*}
At the boundary,
\begin{equation*}
\nabla p_i(G_{\eta_i}(y))\cdot X_*(y)
=-a_i(G_{\eta_i}(y))N_i(G_{\eta_i}(y))\cdot X_*(y)
\le-c\mu.
\end{equation*}
The uniform $C^{1/4}$ modulus of $\nabla p_i$ allows us to choose
$h=h(R,c)>0$ so that the variation of $\nabla p_i\cdot X_*(y)$ is
at most $c\mu/2$ along each interior transverse segment of length
at most $h\|X_*\|_\infty$. Hence, throughout such a segment,
\begin{equation}
-C_R\le \nabla p_i\cdot X_*\le-\frac{c\mu}{2}.            \label{eq:transverse-p-band}
\end{equation}

\par\smallskip
\noindent\textit{Step 2. The boundary part of $\D$ is equivalent to the
graph $L^2$ distance.}\par\nobreak

For $y\in\Ga_*$ with $\eta_1(y)<\eta_2(y)$, the point $G_{\eta_1}(y)$ lies inside $\Om_2$, while
$p_2(G_{\eta_2}(y))=0$.  The fundamental theorem of calculus along the
transverse segment gives
\begin{align*}
p_2(G_{\eta_1}(y))
&=-\int_{\eta_1(y)}^{\eta_2(y)}
\nabla p_2(y+rX_*(y))\cdot X_*(y)\,dr.
\end{align*}
By \eqref{eq:transverse-p-band},
\begin{equation*}
\frac{c\mu}{2}|\eta_2(y)-\eta_1(y)|
\le |p_2(G_{\eta_1}(y))|
\le C_R|\eta_2(y)-\eta_1(y)|.
\end{equation*}
When $\eta_2(y)<\eta_1(y)$, the same calculation with the indices interchanged gives
the corresponding estimate for $p_1(G_{\eta_2}(y))$.

Let $j_i$ be the surface Jacobian of $G_{\eta_i}$, so that
$G_{\eta_i}^*(dS_i)=j_i(y)dS_*(y)$. The collar bounds give
$C_R^{-1}\le j_i\le C_R$, and
$c\le a_i\le C_R$.  Therefore
\begin{align*}
&\int_{\Ga_1\cap\Om_2}a_1^{-1}|p_2|^2\,dS
+\int_{\Ga_2\cap\Om_1}a_2^{-1}|p_1|^2\,dS\\
&\qquad\simeq_{R,c}
\int_{\{\eta_1<\eta_2\}}|p_2(G_{\eta_1}(y))|^2\,dS_*(y)
+\int_{\{\eta_2<\eta_1\}}|p_1(G_{\eta_2}(y))|^2\,dS_*(y)\\
&\qquad\simeq_{R,c}
\|\eta_1-\eta_2\|_{L^2(\Ga_*)}^2.                       
\end{align*}
Together with the bulk terms in $\D$, this proves part~\textup{(i)}.

\par\smallskip
\noindent\textit{Step 3. Uniform coefficient bounds in the weak
stability estimate.}\par\nobreak

For each solution, set $W_i^\pm=u_i\pm B_i$. In formulas indexed by $i$,
write $D_t=\partial_t+u_i\cdot\nabla$, $D_B=B_i\cdot\nabla$, and
$D_t^\pm=D_t\pm D_B=\partial_t+W_i^\pm\cdot\nabla$.
Fix a sufficiently small $0<\epsilon<1/4$.  The prescribed common collar is chosen within a
$C^{1,\epsilon}$ collar as in \cite[Theorem~4.1]{IPTT-MHD}.
Since $s>5/2$, uniform Stein extension and the bulk and surface
Sobolev embeddings give
\begin{equation*}
\|u_i\|_{W^{1,\infty}}+\|B_i\|_{W^{1,\infty}}
+\|\Ga_i\|_{C^{1,1/2}}\le C_R.
\end{equation*}
Similarly, \cref{lem:pressure-estimates}\textup{(i)} gives
$\|p_i\|_{C^{1,1/2}}\le C_R$ from its $H^{s+1/2}$ bound, and
\begin{equation*}
\|D_tp_i\|_{H^s(\Om_i)}+\|D_Bp_i\|_{H^s(\Om_i)}\le C_R.
\end{equation*}
Thus $D_t^\pm p_i=D_tp_i\pm D_Bp_i$ are uniformly bounded in
$W^{1,\infty}$. The two control quantities in
\cite[Theorem~4.1]{IPTT-MHD} therefore satisfy
\begin{align*}
&\|W_i^+\|_{C^{1/2+\epsilon}}
+\|W_i^-\|_{C^{1/2+\epsilon}}
+\|\Ga_i\|_{C^{1,\epsilon}}\le C_R,\\
&\|W_i^+\|_{W^{1,\infty}}+\|W_i^-\|_{W^{1,\infty}}
+\|\Ga_i\|_{C^{1,1/2}}
+\|D_t^+p_i\|_{W^{1,\infty}}
+\|D_t^-p_i\|_{W^{1,\infty}}\le C_R.
\end{align*}
For $T_*\le1$, \cite[Theorem~4.1]{IPTT-MHD} gives
\begin{equation*}
\frac{d}{dt}\D(U_j(t),U_{j+1}(t))
\le C_R\D(U_j(t),U_{j+1}(t)).
\end{equation*}
The constant $C_R$ depends only on the two control bounds above,
the fixed collar, and the reciprocal of the Taylor lower bound.
Gr\"onwall's inequality, $T_*\le1$, and \eqref{eq:D0} give
\eqref{eq:Dprop}. Part~\textup{(i)} then yields \eqref{eq:L2prop}.
\end{proof}

\section{Fourth order estimates with \texorpdfstring{$H^s$}{Hs} low norm control}
\label{sec:fourth-order}

In this section, we establish a fourth order energy estimate whose
constants depend only on the standard $H^s$ bound, the fixed collar
geometry, and a positive lower bound for the Taylor coefficient.
Using the elliptic and pressure estimates from
Section~\ref{sec:phase-spaces}, we first construct a harmonic extension of
the boundary normal to control the differentiated energy coefficients. We then combine the bulk and boundary energies,
using the cancellation of the leading pressure terms to close the
fourth order estimate with coefficients controlled by the standard $H^s$ norm.
Finally, we introduce a comparison regularizer with a kernel independent
of the moving domain and combine the fourth order estimate with the
weak stability bounds from \cref{sec:lowerreg}. Frequency-envelope
summation then yields a common lifespan, strong convergence in the
moving-domain $H^s$ topology, uniqueness, and continuous dependence.

\subsection{A harmonic normal extension}

\begin{lemma}\label{lem:harmonic-normal}
Along a smooth solution satisfying \eqref{eq:low-bound}, let $F$ be the
componentwise harmonic extension of the outward unit normal:
\begin{equation*}
\Delta F=0\quad\hbox{in }\Om,\qquad F=N\quad\hbox{on }\Ga.
\end{equation*}
Then
\begin{equation}\label{eq:F-bounds}
\|N\|_{H^{s-1}(\Ga)}+\|F\|_{H^{s-1/2}(\Om)}
+\|D_tF\|_{H^{s-1}(\Om)}+\|D_BF\|_{H^{s-1}(\Om)}\le C_R.
\end{equation}
\end{lemma}
\begin{proof}
The proof is divided into two steps.

\par\nobreak\smallskip
\noindent\textit{Step 1. The normal and its boundary derivatives.}\par\nobreak

In an oriented graph chart $G(y)=Y_*(y)+\eta(y)X_*(y)$, the normal is
$N=(\partial_1G\times\partial_2G)/|\partial_1G\times\partial_2G|$.
Transversality bounds the denominator away from zero, so Moser estimates give
\begin{equation}\label{eq:N-Hsminus1}
\|N\|_{H^{s-1}(\Ga)}\le C_R.
\end{equation}
Since $s-1>3/2$ on the two-dimensional boundary, this also gives
$\|N\|_{C^{1/2}(\Ga)}\le C_R$.
Apply \cite[Proposition~A.16]{IPTT-MHD} with output order
$s-1/2$ and $(r,\alpha)=(1/2,1/2)$:
\begin{equation*}
\|F\|_{H^{s-1/2}(\Om)}
\le C\Bigl(
\|N\|_{H^{s-1}(\Ga)}
+(1+\|\Ga\|_{H^{s-1/2}})\|N\|_{C^{1/2}(\Ga)}\Bigr)\le C_R.
\end{equation*}

For a tangent vector $X$ transported by the differential of the fluid
flow, one has $D_tX=AX$.
Differentiating $N\cdot X=0$ and $|N|^2=1$ therefore gives
$D_tN=-\Pi A^TN$. For the magnetic derivative, write $D_X=X\cdot\nabla$. The symmetry
of the second fundamental form and the tangency of $B$ imply
\begin{equation*}
X\cdot D_BN=B\cdot D_XN=-N\cdot D_XB=-X\cdot \mathsf{C}^TN.
\end{equation*} 
Since $D_BN$ is tangent, we obtain
\begin{equation}\label{eq:normal-identities}
D_tN=-\Pi A^TN,
\qquad
D_BN=-\Pi \mathsf{C}^TN.
\end{equation}

Apply \eqref{eq:balanced-trace} to $A$ and $\mathsf{C}$ with
$(\tau,r,\alpha)=(s-1,1,0)$. By \eqref{eq:low-bound},
\begin{equation*}
\|A|_\Ga\|_{H^{s-3/2}}+\|\mathsf{C}|_\Ga\|_{H^{s-3/2}}\le C_R.
\end{equation*}
Since $H^{s-3/2}(\Ga)$ is an algebra, \eqref{eq:N-Hsminus1} and
\eqref{eq:normal-identities} give
\begin{equation}\label{eq:DtN-DBN-bound}
\|D_tN\|_{H^{s-3/2}(\Ga)}
+\|D_BN\|_{H^{s-3/2}(\Ga)}\le C_R.
\end{equation}

\par\smallskip
\noindent\textit{Step 2. Commuting the harmonic extension with $D_t$ and
$D_B$.}\par\nobreak

Let $(Z,V)$ stand for $(D_tF,u)$ or $(D_BF,B)$.
Commuting $D_t$ or $D_B$ with $\Delta$ and using $\Delta F=0$ gives
\begin{equation*}
\Delta Z=(\Delta V)\cdot\nabla F
+2\partial_kV\cdot\nabla\partial_kF,
\qquad Z|_\Ga=D_tN\ \hbox{or}\ D_BN.
\end{equation*}
Since $s>5/2$, the embeddings
$H^{s-2}\hookrightarrow L^3$, $H^{s-3/2}\hookrightarrow L^6$ and
$H^{s-1}\hookrightarrow L^\infty$ give
\begin{equation*}
\|\Delta Z\|_2
\le C\Bigl(\|\Delta V\|_3\|\nabla F\|_6
+\|\nabla V\|_\infty\|\nabla^2F\|_2\Bigr)\le C_R\|V\|_{H^s}\|F\|_{H^{s-1/2}}\le C_R.
\end{equation*}

Write $Z=Z_0+Z_h$, where
\begin{equation*}
\Delta Z_0=\Delta Z,
\quad Z_0|_\Ga=0,
\qquad
\Delta Z_h=0,
\quad Z_h|_\Ga=Z|_\Ga.
\end{equation*}
By \eqref{eq:dir2}, $\|Z_0\|_{H^2}\le C_R$.
First suppose $5/2<s\le3$. Then
$\|Z_0\|_{H^{s-1}}\le C_R$ because $s-1\le2$.
For the harmonic part, the variational lifting and the $H^2$ Dirichlet
estimate give the endpoint bounds
\begin{equation*}
\|Z_h\|_{H^1}\le C_R\|Z|_\Ga\|_{H^{1/2}},
\qquad
\|Z_h\|_{H^2}\le C_R\|Z|_\Ga\|_{H^{3/2}}.
\end{equation*}
Complex or real interpolation with parameter
$\vartheta=s-2\in(1/2,1]$ gives
\begin{equation*}
\|Z_h\|_{H^{s-1}}
\le C_R\|Z|_\Ga\|_{H^{s-3/2}}.
\end{equation*}
This proves the required bound in this range.

For $3<s<4$, the same source has the stronger estimate
\begin{align*}
\|\Delta Z\|_{H^{s-3}}
&\le C_R\Bigl(
\|\Delta V\|_{H^{s-2}}\|\nabla F\|_{H^{s-3/2}}
+\|\nabla V\|_{H^{s-1}}\|\nabla^2F\|_{H^{s-5/2}}\Bigr)
\le C_R.
\end{align*}
In each product, the target index $s-3$ is below both input indices,
whose sum minus $s-3$ equals $s-1/2>3/2$.
Moreover, $H^2(\Om)\hookrightarrow L^\infty(\Om)$ controls $Z_0$,
while the componentwise maximum principle and
$H^{s-3/2}(\Ga)\hookrightarrow L^\infty(\Ga)$ control $Z_h$.
Consequently $\|Z\|_{L^\infty}\le C_R$. Apply
\eqref{eq:balanced-dir} with $(\tau,r,\alpha)=(s-1,1,0)$ to obtain
\begin{align*}
\|Z\|_{H^{s-1}}
&\le C_R\Bigl(
\|\Delta Z\|_{H^{s-3}}
+\|Z|_\Ga\|_{H^{s-3/2}}
+(1+\|\Ga\|_{H^{s-1/2}})\|Z\|_{L^\infty}\Bigr)
\le C_R.
\end{align*}
Using \eqref{eq:DtN-DBN-bound}, the two ranges together prove
\begin{equation*}
\|D_tF\|_{H^{s-1}}+\|D_BF\|_{H^{s-1}}\le C_R.
\end{equation*}
This completes the proof of \eqref{eq:F-bounds}.
\end{proof}

We now extend the tangential quadratic form into the fluid domain.
The subharmonicity of $|F|^2$ gives $|F|\le1$.
By \cite[Corollary~A.12]{IPTT-MHD}, $F$ has a uniform $C^{1/2}$ modulus. Since $F=N$ on the
boundary and $|N|=1$, there is a width $\delta_R>0$ such that
$|F|\ge3/4$ at interior points whose distance from $\Ga$ is less
than $\delta_R$.

Fix $R$ and the corresponding $\delta_R$. Around the initial boundary $\Ga_0$, choose a smooth
ambient cutoff $0\le\chi\le1$ equal to one on a sufficiently thin
distance band and supported in a slightly wider band of width less
than $\delta_R/4$. Shrink the state neighborhood and time interval
so that every boundary stays in the inner band and is sufficiently
close to $\Ga_0$ in Hausdorff distance. Thus
\begin{equation}\label{eq:cutoff-support}
\chi|_\Ga=1,\qquad
\operatorname{supp}\chi\cap\Om
\subset\{x\in\Om:\operatorname{dist}(x,\Ga)<\delta_R\}.
\end{equation}
The cutoff derivatives may depend on $R$.
Choose a smooth map from $\R^3$ to the closed unit ball that agrees with
$z\mapsto z/|z|$ for $|z|\ge1/2$, and let $n$ be its composition with $F$.
Thus $n=F/|F|$ on the support of $\chi$. Set
$\gamma=\mathrm{Id}-\chi^2n\otimes n$.
Composition and multiplication in $H^{s-1/2}$ and $H^{s-1}$,
both above the three-dimensional algebra threshold, give
\begin{equation*}
\|n\|_{H^{s-1/2}}+\|D_tn\|_{H^{s-1}}+\|D_Bn\|_{H^{s-1}}
+\|\gamma\|_{H^{s-1/2}}
+\|D_t\gamma\|_{H^{s-1}}+\|D_B\gamma\|_{H^{s-1}}
\le C_R.
\end{equation*}
Derivatives of $\chi$ contribute $u\cdot\nabla\chi$ or
$B\cdot\nabla\chi$, which have the required bounds.
Since $H^{s-1}\hookrightarrow L^\infty$ and
$H^{s-3/2}\hookrightarrow L^6$, it follows that
\begin{equation}\label{eq:gamma-bounds}
0\le\gamma\le\mathrm{Id},\qquad \gamma|_\Ga=\Pi,\qquad
\|\nabla\gamma\|_{L^6}+\|D_t\gamma\|_{L^\infty}
+\|D_B\gamma\|_{L^\infty}\le C_R.
\end{equation}
Away from the cutoff support, $\gamma=\mathrm{Id}$.

\subsection{The fourth order energy}

We use the bulk and boundary quadratic forms from
\cite[Section~5]{Hao2017}, using the harmonic normal from
\cref{lem:harmonic-normal} and the weight $\gamma$ satisfying
\eqref{eq:gamma-bounds}.

For an integer $r\ge2$ and ordered multi-indices
$I=(i_1,\ldots,i_r)$ and $J=(j_1,\ldots,j_r)$ with entries in
$\{1,2,3\}$, put
\begin{equation*}
\partial_I=\partial_{i_1}\cdots\partial_{i_r},\qquad
g^{IJ}=\prod_{\ell=1}^r\gamma^{i_\ell j_\ell}.
\end{equation*}
The order of the tensor $g$ is determined by the length of $I$ and $J$.
Repeated Cartesian indices and multi-indices are summed; indices are
raised and lowered with the Euclidean metric. Write
$U_I=\partial_Iu$, $V_I=\partial_IB$, $P_I=\partial_Ip$,
and $T=\Pi^{\otimes r}\nabla^rp$ on $\Ga$, where the projection acts on
each derivative index. Define
\begin{align*}
E_r={}&1+\frac12\int_\Om g^{IJ}
(U_I\cdot U_J+V_I\cdot V_J)\,dx\\
&+\frac12\int_\Om
(|\nabla^{r-1}\curl u|^2+|\nabla^{r-1}\curl B|^2)\,dx
+\frac12\int_\Ga a^{-1}|T|^2\,dS.
\end{align*}
Write $\mathcal K_r$, $\mathcal W_r$, and $\mathcal S_r$ for the bulk,
curl, and boundary terms, respectively, so that
$E_r=1+\mathcal K_r+\mathcal W_r+\mathcal S_r$.

We first differentiate the bulk energy.
For a vector field $v$ and a scalar or vector field $f$, define
\begin{equation*}
C_r(v,f)_I=
\sum_{\varnothing\ne S\subset\{1,\ldots,r\}}
(\partial_{I_S}v^k)\partial_{I_{S^c}}\partial_k f.
\end{equation*}
Here $I_S$ is the ordered subtuple with positions in $S$,
$S^c=\{1,\ldots,r\}\setminus S$, and $\partial_{I_\varnothing}$
is the identity operator.
Thus $[\partial_I,D_t]f=C_r(u,f)_I$ and
$[\partial_I,D_B]f=C_r(B,f)_I$.
Set
\begin{equation*}
R^u=C_r(B,B)-C_r(u,u),\qquad
R^B=C_r(B,u)-C_r(u,B).
\end{equation*}
The differentiated equations are
\begin{equation*}
D_tU_I=D_BV_I-\nabla P_I+R_I^u,
\qquad D_tV_I=D_BU_I+R_I^B.
\end{equation*}
By symmetry of $g^{IJ}$, the magnetic transport terms combine as
$g^{IJ}D_B(U_I\cdot V_J)$. Since $\diver B=0$ and $B\cdot N=0$,
their integral is $-\int_\Om(D_Bg^{IJ})U_I\cdot V_J$.
For the pressure term, integration by parts gives
\begin{align*}
-\int_\Om g^{IJ}U_I^k\partial_kP_J
={}&\int_\Om(\partial_kg^{IJ})U_I^kP_J
+\int_\Om g^{IJ}(\partial_kU_I^k)P_J
-\int_\Ga g^{IJ}U_I\cdot N\,P_J\,dS\\
={}&\int_\Om(\partial_kg^{IJ})U_I^kP_J
-\int_\Ga T^I\partial_Iu\cdot N\,dS,
\end{align*}
where $\partial_kU_I^k=0$ and $\gamma|_\Ga=\Pi$.
Reynolds' formula and $\diver u=0$ therefore yield
\begin{align}
\frac d{dt}\mathcal K_r
={}&\frac12\int_\Om (D_t g^{IJ})
(U_I\cdot U_J+V_I\cdot V_J)
-\int_\Om(D_B g^{IJ})U_I\cdot V_J\notag\\
&+\int_\Om g^{IJ}(U_I\cdot R^u_J+V_I\cdot R^B_J)
+\int_\Om(\partial_kg^{IJ})U_I^kP_J\notag\\
&-\int_\Ga T^I\partial_Iu\cdot N\,dS.\label{eq:bulk-energy}
\end{align}
Unmarked integrals are with respect to $dx$. The coefficient bounds follow
from the product defining $g^{IJ}$ and \eqref{eq:gamma-bounds}.

For the boundary energy, we use two commutator cancellations.
The scalar pressure satisfies $D_tP_I=\partial_IQ-C_r(u,p)_I$.
Separate the terms with $|S|=1$ and $|S|=r$, and write the remaining
terms as
\begin{equation*}
\widetilde C_r(u,p)_I=
\sum_{\substack{S\subset\{1,\ldots,r\}\\2\le|S|\le r-1}}
(\partial_{I_S}u^k)\partial_{I_{S^c}}\partial_kp.
\end{equation*}
Since $T=\Pi^{\otimes r}\nabla^rp$,
\begin{equation*}
D_tT=(D_t\Pi^{\otimes r})\nabla^rp
+\Pi^{\otimes r}(\nabla^rQ-C_r(u,p)).
\end{equation*}
Fix a slot $\ell\in\{1,\ldots,r\}$. Since
\begin{equation*}
D_t\Pi=-D_tN\otimes N-N\otimes D_tN,
\qquad T^IN_{i_\ell}=0,
\end{equation*}
only the first term contributes to the contraction with $T^I$, giving
\begin{equation*}
T^I(D_t\Pi)_{i_\ell}{}^{j_\ell}
\prod_{h\ne\ell}\Pi_{i_h}{}^{j_h}P_J=-T^I(D_tN)_{i_\ell}N^{j_\ell}
\prod_{h\ne\ell}\Pi_{i_h}{}^{j_h}P_J
=T^I(\partial_{i_\ell}u\cdot N)N^{j_\ell}
\prod_{h\ne\ell}\Pi_{i_h}{}^{j_h}P_J.
\end{equation*}
Here products over $h\ne\ell$ run through $\{1,\ldots,r\}\setminus\{\ell\}$.
The tensor $T$ is tangential in every slot and
$D_tN=-\Pi(\nabla u)^TN$. The part with normal output of the $|S|=1$
term in $-T\cdot\Pi^{\otimes r}C_r(u,p)$ is exactly
\begin{equation*}
-T^I(\partial_{i_\ell}u\cdot N)N^{j_\ell}
\prod_{h\ne\ell}\Pi_{i_h}{}^{j_h}P_J.
\end{equation*}
These terms cancel slot by slot. Define the tangential velocity gradient by
\begin{equation*}
(A_\Ga)_i{}^k=\Pi_i{}^j(\partial_ju^m)\Pi_m{}^k.
\end{equation*}
The remaining tangential term is
$-T^I(A_\Ga)_{i_\ell}{}^kT_{I(i_\ell\mapsto k)}$, where
$I(i_\ell\mapsto k)$ replaces the indicated index.
The remaining $r$th pressure derivatives are therefore fully tangential.

For the other cancellation, $\nabla p=-aN$ on $\Ga$ gives
\begin{equation*}
-a^{-1}T^I(\partial_Iu\cdot\nabla p)
=T^I\partial_Iu\cdot N.
\end{equation*}
This is the $|S|=r$ term. Its boundary integral cancels the pressure
flux in \eqref{eq:bulk-energy}.

It remains to differentiate the surface measure and the Taylor weight.
The surface divergence is $\diver_\Ga u=\operatorname{tr}(\Pi\nabla u)$.
The identities
\begin{equation*}
D_ta=-\partial_NQ-a(N\cdot\nabla u\,N),\qquad
\diver_\Ga u=-N\cdot\nabla u\,N
\end{equation*}
imply
\begin{equation*}
\diver_\Ga u-\frac{D_ta}{a}
=\frac{\partial_NQ}{a}.
\end{equation*}
Applying the surface transport identity
\begin{equation*}
\frac{d}{dt}\int_{\Ga_t}f\,dS
=\int_{\Ga_t}\bigl(D_tf+f\,\diver_\Ga u\bigr)\,dS
\end{equation*}
to $f=\frac12a^{-1}|T|^2$ gives
\begin{equation}\begin{aligned}
\frac d{dt}\mathcal S_r
={}&\int_\Ga a^{-1}T\cdot\Pi^{\otimes r}\nabla^rQ
+\int_\Ga T^I\partial_Iu\cdot N\\
&-\int_\Ga a^{-1}T\cdot\Pi^{\otimes r}\widetilde C_r(u,p)
-\sum_{\ell=1}^r\int_\Ga a^{-1}
T^I (A_\Ga)_{i_\ell}{}^kT_{I(i_\ell\mapsto k)}\\
&+\frac12\int_\Ga a^{-2}(\partial_NQ)|T|^2.
\end{aligned}\label{eq:boundary-energy}\end{equation}

We next express the fourth tangential derivatives in the fixed collar.
Fix a reference chart, with smooth maps $Y_*,X_*$,
and write the moving boundary as
\begin{equation*}
G(y)=Y_*(y)+\eta(y)X_*(y).
\end{equation*}
Here $y=(y_1,y_2)$ denotes the surface coordinates;
$D^mG$ collects all order-$m$ derivatives in $y$, and subscripts
$\alpha_1,\ldots,\alpha_m\in\{1,2\}$ denote their individual components.
The tangent vectors $G_\alpha=\partial_\alpha G$ have a uniformly
nondegenerate metric and $|N\cdot X_*|\ge\mu>0$.  The sign of
$N\cdot X_*$ may be fixed by orienting the reference transverse field.
For any scalar $q$ whose boundary trace vanishes, differentiate
$q\circ G=0$ four times. Grouping terms by partitions of the four labeled
derivatives gives the Fa\`a di Bruno identity
\begin{align}
0={}&\nabla^4q(G_{\alpha_1},G_{\alpha_2},G_{\alpha_3},G_{\alpha_4})
+\nabla q\cdot G_{\alpha_1\alpha_2\alpha_3\alpha_4}\notag\\
&+\sum_{\substack{S\subset\{1,2,3,4\}\\|S|=2}}
\nabla^3q(G_{\alpha_S},G_{\alpha_b},G_{\alpha_c})\notag\\
&+\sum_{\substack{S\subset\{1,2,3,4\}\\|S|=3}}
\nabla^2q(G_{\alpha_S},G_{\alpha_b})
+\sum_{\{S,S'\}\in\mathcal P}
\nabla^2q(G_{\alpha_S},G_{\alpha_{S'}}).\label{eq:chain-four}
\end{align}
Here $G_{\alpha_S}$ means the derivative of $G$ with respect to the
indices in the block $S$. In the second line $\{b,c\}$ is the complement
of $S$, so there are six $(2,1,1)$ terms.  The first sum in the third line
has four $(3,1)$ terms, with $\{b\}=S^c$, and $\mathcal P$ is
the set of the three unordered partitions of $\{1,2,3,4\}$ into
two pairs.  Every derivative of $q$ is evaluated at
$G(y)$.  Since each first derivative $G_\alpha$ is tangent to the
boundary, the first term is exactly the contraction of
$\Pi^{\otimes4}\nabla^4q$ with four tangent vectors.

For $w\in H^1(\Om)$, we use the uniform first order trace estimate
\begin{equation}\label{eq:trace-L4}
\|w|_\Ga\|_{L^4(\Ga)}\le C_R\|w\|_{H^1(\Om)}.
\end{equation}
This follows from the Lipschitz trace into $H^{1/2}$ and the
surface embedding $H^{1/2}\hookrightarrow L^4$, and also applies to physical
derivatives. Uniform Stein extensions give the bulk interpolation estimates.

Since $G=Y_*+\eta X_*$, Leibniz' rule gives
\begin{equation}\label{eq:graph-four}
G_{\alpha_1\alpha_2\alpha_3\alpha_4}
=X_*\partial_{\alpha_1\alpha_2\alpha_3\alpha_4}\eta+\mathcal R,
\qquad
\|\mathcal R\|_2\le C_RM_4^{\theta_0},
\end{equation}
where $\mathcal R$ collects the terms containing at most three derivatives of
$\eta$.
Interpolation when $m>s$, and the low norm bound when $m\le s$, give
$\|\eta\|_{H^m}\le C_RM_4^{(m-s)_+/(4-s)}$ for $m\in\{3,7/2\}$.
Thus
\begin{equation*}
\|\eta\|_{H^3}\le C_RM_4^{\theta_0},
\qquad
\|\eta\|_{H^{7/2}}\le C_RM_4^{\theta_1}.
\end{equation*}
The low $H^s$ bound controls $DG$ in $L^\infty$ and $D^2G$ in $L^4$.
The surface embeddings $H^1\hookrightarrow L^8$ and
$H^{1/2}\hookrightarrow L^4$ then give
\begin{equation*}
\|DG\|_\infty+\|D^2G\|_4\le C_R,
\qquad
\|D^2G\|_8\le C_RM_4^{\theta_0},
\qquad
\|D^3G\|_4\le C_RM_4^{\theta_1}.
\end{equation*}

\begin{lemma}\label{lem:E4equiv}
\textup{(i)} For every integer $r\ge1$ and smooth vector field $v$, the
following pointwise estimate holds for each unit vector $n$, with
$\Pi_n=\mathrm{Id}-n\otimes n$:
\begin{equation*}
|\nabla^rv|\le C(r)\bigl(
|\Pi_n^{\otimes r}\nabla^rv|
+|\nabla^{r-1}\curl v|+|\nabla^{r-1}\diver v|\bigr).
\end{equation*}
The constant $C(r)$ depends only on $r$.

\textup{(ii)} For smooth states satisfying \eqref{eq:low-bound},
$a\ge c_0>0$, and \eqref{eq:cutoff-support},
\begin{equation}\label{eq:energy-equivalence}
C_R^{-1}M_4^2\le E_4\le C_RM_4^2.
\end{equation}
\end{lemma}
\begin{proof}
The proof is divided into three steps.

\par\nobreak\smallskip
\noindent\textit{Step 1. Pointwise div--curl algebra.}\par\nobreak

Fix a point and rotate coordinates so that $n=e_3$, where $e_3$ is the
third Cartesian unit vector. Suppose all three quantities on the right
side of the asserted inequality vanish. The curl condition gives, for
an ordered multi-index $I$ of length $r-1$ and $i,b\in\{1,2,3\}$,
\begin{equation}\label{eq:T-curl-symmetry}
\partial_I\partial_i v_b=\partial_I\partial_b v_i.
\end{equation}
Together with symmetry of the derivative indices, \eqref{eq:T-curl-symmetry} makes
$\nabla^rv$ symmetric in all $r+1$ indices, including the vector index.
The divergence condition gives
\begin{equation}\label{eq:T-tracefree}
\partial_I\partial_1v_1+\partial_I\partial_2v_2
+\partial_I\partial_3v_3=0.
\end{equation}
The vanishing tangential projection and full symmetry imply that every
component with at most one normal index vanishes. If all components with
$q$ normal indices vanish, choose $I$ in \eqref{eq:T-tracefree} with
$q$ normal indices. Its first two terms vanish, so the component with
$q+2$ normal indices also vanishes. Induction from $q=0,1$ gives
$\nabla^rv=0$. Thus the tangential projection together with the curl
and divergence contractions defines an injective linear map on the
finite dimensional space of tensors symmetric in the derivative indices.
Finite dimensional norm equivalence gives the estimate, and rotational
invariance makes its constant independent of $n$.

\par\smallskip
\noindent\textit{Step 2. The fourth order boundary chain rule.}\par\nobreak

Apply \eqref{eq:chain-four} with $q=p$.  The coefficient of the
highest graph derivative is obtained from
\eqref{eq:graph-four}:
\begin{equation*}
\nabla p\cdot G_{\alpha_1\alpha_2\alpha_3\alpha_4}
=(\nabla p\cdot X_*)\partial_{\alpha_1\alpha_2\alpha_3\alpha_4}\eta
+\nabla p\cdot \mathcal R.
\end{equation*}
Because $p=0$ on $\Ga$, $\nabla p=-aN$, and therefore
\begin{equation}\label{eq:graph-leading-coeff}
|\nabla p\cdot X_*|=a|N\cdot X_*|\ge c_0\mu.
\end{equation}
All other terms in \eqref{eq:chain-four} contain at most three
derivatives of the graph.  The pressure trace bounds follow from
\eqref{eq:trace-L4}:
\begin{align*}
\|\nabla^3p|_\Ga\|_4
&\le C_R\|p\|_{H^4}
\le C_RM_4^{\theta_1},\\
\|\nabla^2p|_\Ga\|_4
&\le C_R\|p\|_{H^3}\le C_R,
\qquad
\|\nabla p\|_\infty\le C_R.
\end{align*}
Here we used the pressure interpolation from
\eqref{eq:p-H4-interp} and the low $H^{s+1/2}$ bound.

H\"older's inequality bounds the $(2,1,1)$, $(3,1)$, and $(2,2)$
classes in $L^2(\Ga)$ by
\begin{align*}
\|\nabla^3p|_\Ga\|_4\|D^2G\|_4\|DG\|_\infty^2
&\le C_RM_4^{\theta_1},\\
\|\nabla^2p|_\Ga\|_4\|D^3G\|_4\|DG\|_\infty
&\le C_RM_4^{\theta_1},\\
\|\nabla^2p|_\Ga\|_4\|D^2G\|_8^2
&\le C_RM_4^{2\theta_0}.
\end{align*}
The remaining term satisfies
$\|\nabla p\cdot \mathcal R\|_2\le C_RM_4^{\theta_0}$.
Since $2\theta_0\le\theta_1<1$, solving
\eqref{eq:chain-four} for $D^4\eta$ using
\eqref{eq:graph-leading-coeff} and summing over the finite atlas gives
\begin{equation}\label{eq:graph-recovery}
\|\Ga\|_{H^4}
\le C_R\left(
1+\|\Pi^{\otimes4}\nabla^4p\|_{L^2(\Ga)}
+M_4^{\theta_1}\right).
\end{equation}
Conversely, solve \eqref{eq:chain-four} for its first term. Since
$\|D^4\eta\|_2\le M_4$ and every lower term is bounded by $C_RM_4$, we obtain
\begin{equation}\label{eq:T-upper-by-M4}
\|\Pi^{\otimes4}\nabla^4p\|_{L^2(\Ga)}\le C_RM_4.
\end{equation}
This proves the upper bound for the boundary part of $E_4$.

\par\smallskip
\noindent\textit{Step 3. Recovery of the full fourth order bulk norms.}\par\nobreak

Apply the pointwise div--curl inequality with the extended normal
$n(x)$ on the support of $\chi$. There
$|n|=1$ and
\begin{equation*}
\gamma=\mathrm{Id}-\chi^2n\otimes n
\ge\mathrm{Id}-n\otimes n.
\end{equation*}
Thus the bulk energy controls the tangential projection in part~(i).
Off the support of $\chi$, $\gamma=\mathrm{Id}$ controls all fourth derivatives.
Using $\diver u=\diver B=0$, we obtain
\begin{equation}\label{eq:H4-top-deriv}
\|\nabla^4u\|_2+\|\nabla^4B\|_2
\le C_R E_4^{1/2}.
\end{equation}
Uniform interpolation, the $L^2$ bounds in \eqref{eq:low-bound}, and
\eqref{eq:H4-top-deriv} give
\begin{equation*}
\|u\|_{H^4}+\|B\|_{H^4}
\le C_RE_4^{1/2}.
\end{equation*}
Combining this with \eqref{eq:graph-recovery} and the bounds
$c_0\le a\le C_R$ gives
\begin{equation*}
M_4\le C_RE_4^{1/2}+C_RM_4^{\theta_1}.
\end{equation*}
Young's inequality, $\theta_1<1$, and $E_4\ge1$ give
$M_4\le C_RE_4^{1/2}$. The reverse bound follows from
\eqref{eq:T-upper-by-M4}, the bulk and curl energies, and $a\ge c_0$,
proving \eqref{eq:energy-equivalence}.
\end{proof}

\begin{proposition}
Let a smooth solution remain in the fixed collar on $[0,\tau]$ and
satisfy \eqref{eq:cutoff-support}. Assume
$M_s(t)\le R$ and $a(t)\ge c_0>0$ throughout this interval. Then
\begin{equation*}
\frac{d}{dt}E_4(t)\le C_RE_4(t),\qquad
E_4(t)\le e^{C_Rt}E_4(0).
\end{equation*}
Consequently,
\begin{equation}\label{eq:H4-growth}
M_4(t)\le C_Re^{C_Rt}M_4(0),\qquad 0\le t\le\tau.
\end{equation}
The constants depend only on $s$, $R$, $c_0^{-1}$, the fixed collar,
and the fixed volume bound.
\end{proposition}
\begin{proof}
The proof is divided into four steps.

\par\nobreak\smallskip
\noindent\textit{Step 1. Bulk terms.}\par\nobreak

Adding \eqref{eq:bulk-energy} and \eqref{eq:boundary-energy} cancels
the pressure flux. The product rule and \eqref{eq:gamma-bounds} give
\begin{equation*}
\|D_tg\|_\infty+\|D_Bg\|_\infty\le C_R.
\end{equation*}
Together with \eqref{eq:H4-top-deriv}, these bounds control the first
two coefficient terms in \eqref{eq:bulk-energy} by $C_RE_4$.

For the commutators, $s>5/2$ gives
$H^{s-2}(\Om)\hookrightarrow H^{1/2}(\Om)\hookrightarrow L^3(\Om)$,
while $H^1(\Om)\hookrightarrow L^6(\Om)$.  Consequently, for
$V,W\in\{u,B\}$,
\begin{equation}\label{eq:bulk-tame-product}
\|\nabla^2V\,\nabla^3W\|_2
\le \|\nabla^2V\|_3\|\nabla^3W\|_6
\le C_R\|V\|_{H^s}\|W\|_{H^4}
\le C_RM_4.
\end{equation}
The derivative distributions in $C_4(V,W)$ are $(1,4)$ and $(2,3)$,
up to interchange of the factors. The low $W^{1,\infty}$ bound handles
the first, and \eqref{eq:bulk-tame-product} handles the second. Hence
\begin{equation*}
\|R^u\|_2+\|R^B\|_2\le C_RM_4.
\end{equation*}
By \cref{lem:E4equiv}\textup{(ii)}, these terms contribute at most
$C_RM_4E_4^{1/2}\le C_RE_4$.

For the pressure coefficient term, \eqref{eq:gamma-bounds} gives
$\|\nabla g\|_6\le C_R$. Hence
\begin{equation*}
\left|\int_\Om(\partial_kg^{IJ})U_I^kP_J\,dx\right|
\le C\|\nabla g\|_6
\|\nabla^4u\|_2\|\nabla^4p\|_3\le C_RM_4^2.
\end{equation*}
Since $p\in H^{9/2}$,
$\nabla^4p\in H^{1/2}\hookrightarrow L^3$ and
$\|\nabla^4p\|_3\le C_RM_4$.  By \eqref{eq:energy-equivalence}, this contribution is bounded by $C_RE_4$.

\par\smallskip
\noindent\textit{Step 2. Middle boundary commutators.}\par\nobreak

For $r=4$, the middle commutator
$\widetilde C_4(u,p)$ contains only the derivative distributions
$(2,3)$ and $(3,2)$.  By \eqref{eq:trace-L4},
\eqref{eq:interp-H3-fields}, and \eqref{eq:p-H4-interp},
\begin{align*}
\|\nabla^2u|_\Ga\|_4
&\le C_R\|u\|_{H^3}
\le C_RM_4^{\theta_0},\\
\|\nabla^3p|_\Ga\|_4
&\le C_R\|p\|_{H^4}
\le C_RM_4^{\theta_1}.
\end{align*}
Thus the $(2,3)$ terms have $L^2(\Ga)$ norm at most
$C_RM_4^{\theta_0+\theta_1}$.  For the $(3,2)$ terms,
\begin{align*}
\|\nabla^3u|_\Ga\|_4
&\le C_R\|u\|_{H^4}\le C_RM_4,\\
\|\nabla^2p|_\Ga\|_4
&\le C_R\|p\|_{H^3}\le C_R,
\end{align*}
so their $L^2$ norm is at most $C_RM_4$.  Since
$\theta_0+\theta_1<1$ and $M_4\ge1$,
\begin{equation*}
\|\widetilde C_4(u,p)\|_{L^2(\Ga)}\le C_RM_4.
\end{equation*}
Multiplication by $a^{-1}T$ and Cauchy--Schwarz therefore give a
contribution bounded by $C_RM_4E_4^{1/2}\le C_RE_4$.

The remaining quadratic boundary terms are bounded by $C_RE_4$, since
$\|A_\Ga\|_\infty+\|\partial_NQ\|_\infty\le C_R$ and $a\ge c_0$.

\par\smallskip
\noindent\textit{Step 3. The projected fourth derivative of $Q$.}\par\nobreak

Since $Q|_\Ga=0$, apply \eqref{eq:chain-four} with $q=Q$.
Estimates \eqref{eq:Q-low}, \eqref{eq:pressure-high}, and
\eqref{eq:trace-L4} give
\begin{align*}
\|\nabla Q\|_\infty\|D^4G\|_2&\le C_RM_4,\\
\|\nabla^3Q|_\Ga\|_4\|D^2G\|_4\|DG\|_\infty^2
&\le C_R\|Q\|_{H^4}\le C_RM_4.
\end{align*}
Interpolation between $Q\in H^s$ and $Q\in H^4$ when $s<3$,
and the low $H^s$ bound when $s\ge3$, give
\begin{equation*}
\|Q\|_{H^3}\le C_RM_4^{\theta_0},
\qquad
\|\nabla^2Q|_\Ga\|_4\le C_RM_4^{\theta_0}.
\end{equation*}
The $(3,1)$ and $(2,2)$ terms are therefore bounded by
\begin{align*}
\|\nabla^2Q|_\Ga\|_4\|D^3G\|_4\|DG\|_\infty
&\le C_RM_4^{\theta_0+\theta_1},\\
\|\nabla^2Q|_\Ga\|_4\|D^2G\|_8^2
&\le C_RM_4^{3\theta_0}.
\end{align*}
By \eqref{eq:strict-exponents} and $M_4\ge1$, all these terms are
bounded by $C_RM_4$. Solving the chain rule for its fully projected first term gives
\begin{equation*}
\|\Pi^{\otimes4}\nabla^4Q\|_{L^2(\Ga)}\le C_RM_4.
\end{equation*}
Thus the first term in \eqref{eq:boundary-energy} is bounded by
\begin{equation*}
C_R\|T\|_2M_4\le C_RM_4E_4^{1/2}\le C_RE_4.
\end{equation*}

\par\smallskip
\noindent\textit{Step 4. Curl energy.}\par\nobreak

Set $\omega=\curl u$ and $\jmath=\curl B$. Let $\epsilon_{ikl}$ be
the alternating symbol, normalized by
$\epsilon_{123}=1$. Taking curls of the MHD equations gives
\begin{align*}
D_t\omega_i-D_B\jmath_i
&=\epsilon_{ikl}\bigl(\partial_kB^m\partial_mB^l
-\partial_ku^m\partial_mu^l\bigr),\\
D_t\jmath_i-D_B\omega_i
&=\epsilon_{ikl}\bigl(\partial_kB^m\partial_mu^l
-\partial_ku^m\partial_mB^l\bigr).
\end{align*}
Apply three spatial derivatives and retain the commutators with $D_t,D_B$
on the right. In the sum of the two $L^2$ energy identities,
the leading magnetic terms cancel:
\begin{equation*}
\int_\Om \nabla^3\omega\cdot D_B\nabla^3\jmath
+\int_\Om \nabla^3\jmath\cdot D_B\nabla^3\omega=\int_\Om D_B(\nabla^3\omega\cdot\nabla^3\jmath)=0.
\end{equation*}
The last equality uses $\diver B=0$ and $B\cdot N=0$.
The commutators and differentiated quadratic terms have distributions
$(1,4)$ and $(2,3)$, up to interchange. The Lipschitz bounds from
\eqref{eq:low-bound} and the product estimate \eqref{eq:bulk-tame-product} give
\begin{equation*}
\left|\frac d{dt}\mathcal W_4\right|\le C_RM_4^2
\le C_RE_4.
\end{equation*}

Adding the estimates gives $\frac d{dt}E_4\le C_RE_4$.
Gr\"onwall's inequality and \cref{lem:E4equiv}\textup{(ii)} yield
\eqref{eq:H4-growth}.
\end{proof}

For the regularized solutions, combining
\eqref{eq:H4-growth} with \eqref{eq:reghigh} gives
\begin{equation}\label{eq:E4local}
\sup_{0\le t\le\tau}\|U_j(t)\|_{\X^4}
\le C_R\bigl(1+\|S_jU_0\|_{\X^4}\bigr)
\le C_R2^{(4-s)j}c_j.
\end{equation}
The constants are uniform for $0\le\tau\le1$ and independent of $j$.
\subsection{Comparison on moving domains}

To sum differences of smooth solutions on varying domains, we use
operators $\mathcal J_j$ with a fixed sampling kernel. They preserve
divergence and act on bulk fields beyond each boundary, at a distance
comparable to the smoothing scale. The graph estimate
\eqref{eq:Linfgeomtail} places all later domains in this enlargement.

\begin{lemma}
\label{lem:fixed-comparison-regularizer}
Fix a bounded $\X^s$ subchart of a sufficiently small graph collar, and
write $\Om^h=\{x:\operatorname{dist}(x,\Om)<h\}$ within the fixed
ambient neighborhood. There are linear operators $\mathcal J_j$,
$j\ge j_0$, with a domain-independent kernel and a fixed $\kappa>0$.
The kernels, $\kappa$, and the operator constants depend only on the collar;
$j_0$ may also depend on the common $H^s$ bound. These operators have
the following properties.

\begin{enumerate}[label=\textup{(\roman*)},leftmargin=2.5em]
\item \emph{Inward sampling and preservation of divergence.}
For every domain $\Om$ in the subchart and every locally integrable
vector field $F$ on $\Om$, the field $\mathcal J_jF$ is smooth on
$\Om^{\kappa2^{-j}}$ and uses only values of $F$ in $\Om$.
If $F$ is solenoidal in $\Om$, then $\mathcal J_jF$ is solenoidal
on this enlargement.

\item \emph{Sobolev bounds and approximation.}
For $F\in H^r(\Om;\R^3)$ and $0\le q\le r\le4$,
\begin{align}
\|\mathcal J_jF\|_{H^r(\Om^{\kappa 2^{-j}})}
&\le C\|F\|_{H^r(\Om)},
\label{eq:fixed-smoother-bounded}\\
\|(\mathrm{Id}-\mathcal J_j)F\|_{H^q(\Om)}
&\le C2^{-(r-q)j}\|F\|_{H^r(\Om)},
\label{eq:fixed-smoother-error}\\
\|(\mathcal J_{j+1}-\mathcal J_j)F\|_{H^q(V)}
&\le C2^{-(r-q)j}\|F\|_{H^r(\Om)}
\quad\text{if }V\subset\Om^{\kappa 2^{-j-1}}\text{ is open}.
\label{eq:fixed-smoother-difference}
\end{align}
For every fixed $F\in H^r(\Om)$, one also has
$\mathcal J_jF\to F$ in $H^r(\Om)$ as $j\to\infty$.
\item \emph{Local $L^2$ bound.}
The kernel satisfies
\begin{equation}
\|\mathcal J_jF\|_{L^2(V)}\le C\|F\|_{L^2(D)}
\label{eq:fixed-smoother-core}
\end{equation}
for every open set $V$ on which $\mathcal J_jF$ is defined, whenever
all points sampled from $V$ lie in a measurable set $D\subset\Om$.
The same kernel applies on intersections and on all subsequent domains.
\end{enumerate}
\end{lemma}

\begin{proof}

\par\nobreak\smallskip
\noindent\textit{(i) Inward sampling and preservation of divergence.}\par\nobreak

Write the fixed collar map and the defining function of $\Ga_\eta$ as
\begin{equation}\label{eq:collar-coordinates}
G(y,r)=y+rX_*(y),\qquad G^{-1}=(\pi_*,r_*),\qquad
\rho_\eta=r_*-\eta\circ\pi_*.
\end{equation}
Here $\pi_*$ and $r_*$ are the surface and transverse coordinate
projections. Then $\rho_\eta<0$ on the interior side and
$N=\nabla\rho_\eta/|\nabla\rho_\eta|$ on $\Ga_\eta$.
Choose $0<\alpha<\min\{1,s-2\}$. On a smaller collar containing all the
surfaces, the bounded $H^s$ graph norm and Sobolev embedding give uniform
constants $L,C_\rho>0$ such that
\begin{equation*}
\|\nabla\rho_\eta\|_{L^\infty}\le L,\qquad
[\nabla\rho_\eta]_{C^\alpha}\le C_\rho.
\end{equation*}
Also $\nabla\rho_\eta(G(y,r))\cdot X_*(y)=1$, so
$|\nabla\rho_\eta|\ge m_0>0$ uniformly.

Choose a smooth compactly supported vector field $V_*$ pointing outward
and transverse to all the graphs. Shrinking the fixed $C^1$ collar gives
\begin{equation*}
V_*\cdot N\ge\mu_0>0\qquad\hbox{on every }\Ga_\eta.
\end{equation*}
Set $\mu_1=m_0\mu_0/2$. The uniform H\"older bound gives $d_0>0$
such that the region $\operatorname{dist}(x,\Ga_\eta)<2d_0$ lies in the
smaller collar and $\nabla\rho_\eta(x)\cdot V_*(x)\ge\mu_1$ there.
Choose $0<r_0\le1$ with $Lr_0\le\mu_1/4$ and
$\varrho\in C_c^\infty(B(0,r_0))$ with $\int\varrho=1$.
The choices of $V_*,r_0,\varrho$ depend only on the fixed $C^1$ collar;
the $H^s$ bound affects $d_0$ and the admissible range of $h$.

For $h>0$ and $z\in\operatorname{supp}\varrho$, define the sampling map
and its Piola matrix by
\begin{equation*}
\mathcal F_{h,z}(x)=x+h(z-V_*(x)),\qquad
\mathsf M_h(x)=\det(\mathrm{Id}-hDV_*(x))(\mathrm{Id}-hDV_*(x))^{-1}.
\end{equation*}
When $h\|DV_*\|_\infty<1/2$, the equation
$\mathcal F_{h,z}(x)=y$ is solved by contraction in
$x=y-hz+hV_*(x)$. Thus $\mathcal F_{h,z}$ is a global smooth
diffeomorphism with positive Jacobian and uniform derivative bounds at
every fixed order for the map and its inverse.

Let $M=1+\|V_*\|_\infty$. For
$\operatorname{dist}(x,\Ga_\eta)<d_0$ and $hM<d_0$, the segment from
$x$ to $\mathcal F_{h,z}(x)$ stays in
$\{x:\operatorname{dist}(x,\Ga_\eta)<2d_0\}$. Taylor's
formula, with remainder $R_{h,z}$, gives
\begin{align*}
\rho_\eta(\mathcal F_{h,z}(x))
&=\rho_\eta(x)+h\nabla\rho_\eta(x)\cdot(z-V_*(x))+R_{h,z}(x),\\
|R_{h,z}(x)|&\le \frac{C_\rho}{1+\alpha}M^{1+\alpha}h^{1+\alpha}.
\end{align*}
The linear term is at most $-3\mu_1h/4$. Choose $h_0>0$ so that
$h_0M<d_0/2$, $h_0\|DV_*\|_\infty<1/2$, and the remainder is at
most $\mu_1h/4$ for $h\le h_0$. Then
\begin{equation*}
\rho_\eta(\mathcal F_{h,z}(x))
\le\rho_\eta(x)-\mu_1h/2.
\end{equation*}
Choose $0<\kappa\le\min\{1,\mu_1/(4L)\}$ and decrease $h_0$ so that
$\kappa h_0<d_0$. For $x\in\Om^{\kappa h}$ within distance $d_0$
of the boundary, comparison with a closest boundary point gives
$\rho_\eta(x)\le L\kappa h$; hence
$\rho_\eta(\mathcal F_{h,z}(x))\le-\mu_1h/4$.
The remaining points lie in $\Om$ at distance at least $d_0$ from the
boundary and move by less than $d_0/2$. Consequently
\begin{equation}\label{eq:sampling-inside}
x\in\Om^{\kappa h}
\quad\Longrightarrow\quad
\mathcal F_{h,z}(x)\in\Om
\qquad(0<h\le h_0,\ z\in\operatorname{supp}\varrho).
\end{equation}
In the boundary region the sampling distance is at least
$\mu_1h/(4L)$, and in the remaining region it is at least $d_0/2$.
Thus every sampled point has distance at least a fixed positive multiple
of $h$ from the boundary, uniformly on the bounded graph chart.

Take $j$ large enough that $9\cdot2^{-j}\le h_0$, and put
$\lambda_j=2^{-j}$ and $w_m=(-1)^{m+1}\binom9m$ for $1\le m\le9$.
The ninth finite difference identity gives
\begin{equation}\label{eq:moment-cancel}
\sum_{m=1}^9w_m=1,\qquad
\sum_{m=1}^9w_m m^k=0,\quad 1\le k\le8.
\end{equation}
Define
\begin{equation}
\mathcal J_jF(x)
=\sum_{m=1}^{9}w_m\int\varrho(z)\mathsf M_{m\lambda_j}(x)
F\bigl(\mathcal F_{m\lambda_j,z}(x)\bigr)\,dz.
\label{eq:fixed-smoother-definition}
\end{equation}
By \eqref{eq:sampling-inside}, this uses only values in $\Om$ even on
$\Om^{\kappa\lambda_j}$. The Piola identity
\begin{equation*}
\diver_x\bigl(\mathsf M_h(F\circ\mathcal F_{h,z})\bigr)
=\det D\mathcal F_{h,z}\,(\diver F)\circ\mathcal F_{h,z}
\end{equation*}
therefore proves that $\mathcal J_j$ preserves divergence on the entire
enlargement.

\par\smallskip
\noindent\textit{(ii) Sobolev bounds and approximation.}\par\nobreak

Let $\widetilde F=\mathscr E_\Om F$ be the uniform Stein extension from
\cite[Appendix~A.1, Proposition~A.2]{IPTT-MHD}. For an ambient vector field
$v$, write
\begin{equation*}
T_h^zv(x)=\mathsf M_h(x)v\bigl(x+h(z-V_*(x))\bigr).
\end{equation*}
The uniform bounds for the coordinate maps, change of variables, and
interpolation give, for $0\le r\le4$,
\begin{equation}\label{eq:Jj-Hr-bounded-detail}
\|\mathcal J_j\widetilde F\|_{H^r(\R^3)}
\le C\|\widetilde F\|_{H^r(\R^3)}.
\end{equation}
Differentiating the kernel also gives
$\|\mathcal J_j\widetilde F\|_{H^r}\le C2^{rj}\|\widetilde F\|_{L^2}$.
Restriction proves \eqref{eq:fixed-smoother-bounded}.

For smooth ambient $v$, Taylor expansion of $T_h^zv$ through degree
$8$ has remainder bounded in $H^q$ by $Ch^9\|v\|_{H^{q+9}}$:
each $h$ derivative introduces at most one spatial derivative of $v$.
Applying \eqref{eq:moment-cancel} at $h=m\lambda_j$ therefore gives
\begin{equation*}
\|(\mathrm{Id}-\mathcal J_j)v\|_{H^q}
\le C2^{-9j}\|v\|_{H^{q+9}}.
\end{equation*}
Uniform boundedness also yields
\begin{equation*}
\|(\mathrm{Id}-\mathcal J_j)v\|_{H^q}\le C\|v\|_{H^q}.
\end{equation*}
For $0\le q\le r\le4$, split the extension as
$\widetilde F=L_j\widetilde F+(\mathrm{Id}-L_j)\widetilde F$ using the fixed
low frequency cutoff. Bernstein's inequality and these two bounds imply
\begin{align*}
\|(\mathrm{Id}-\mathcal J_j)L_j\widetilde F\|_{H^q}
&\le C2^{-9j}2^{(q+9-r)j}\|\widetilde F\|_{H^r},\\
\|(\mathrm{Id}-\mathcal J_j)(\mathrm{Id}-L_j)\widetilde F\|_{H^q}
&\le C2^{-(r-q)j}\|\widetilde F\|_{H^r}.
\end{align*}
This proves \eqref{eq:fixed-smoother-error}. Applying the same ambient
estimate to
$\mathcal J_{j+1}-\mathcal J_j=(\mathrm{Id}-\mathcal J_j)-(\mathrm{Id}-\mathcal J_{j+1})$
and restricting to $V\subset\Om^{\kappa2^{-j-1}}$ proves
\eqref{eq:fixed-smoother-difference}: both kernels sample inside $\Om$
and hence agree with their action on the same extension.
For a fixed extension $\widetilde F\in H^r(\R^3)$, approximate
$\widetilde F$ in $H^r$ by a smooth ambient field. The Taylor estimate
above gives convergence for this fixed smooth field, while
\eqref{eq:Jj-Hr-bounded-detail} controls the approximation error uniformly
in $j$. Thus $\mathcal J_jF\to F$ in $H^r(\Om)$.

\par\smallskip
\noindent\textit{(iii) Local $L^2$ bound.}\par\nobreak

Changing variables in \eqref{eq:fixed-smoother-definition} gives a smooth
kernel $k_j(x,y)$ with
\begin{equation*}
|k_j(x,y)|\le C2^{3j},\qquad
k_j(x,y)=0\quad\hbox{if }|x-y|>C2^{-j},
\end{equation*}
and
\begin{equation*}
\sup_x\int|k_j(x,y)|\,dy
+\sup_y\int|k_j(x,y)|\,dx\le C.
\end{equation*}
If all points sampled from $V$ lie in the measurable set $D$, then
$\mathcal J_jF(x)=\int_Dk_j(x,y)F(y)\,dy$ for $x\in V$.
Schur's test proves \eqref{eq:fixed-smoother-core}. The kernel depends
only on $V_*,\varrho,j$, and is independent of the domain.
\end{proof}

The common kernel allows comparison of adjacent states on later domains.
\Cref{lem:telescoping} combines weak differences and fourth order bounds
to control the strong $H^s$ tail by the $\ell^2$ frequency envelope.

\begin{lemma}
\label{lem:telescoping}
Let $U_l=(u_l,B_l,\Ga_{\eta_l})$, $J\le l\le K$, be smooth
regularized solutions on $[0,T]$, with fluid domains $\Om_l$.
Assume that they remain in one bounded local chart, have a common positive
lower bound for the Taylor coefficient, and satisfy \eqref{eq:weak-band}
for each adjacent pair. Let $c$ be the envelope in \eqref{eq:envelope}.
Suppose
\begin{align}
\D(U_l,U_{l+1})^{1/2}&\le C_0 2^{-sl}c_l,
&&J\le l<K,                                             \label{eq:tel-low}\\
\|U_l\|_{\X^4}&\le C_1 2^{(4-s)l}c_l,
&&J\le l\le K.                                      \label{eq:tel-high}
\end{align}
For $m\in\R$, let $c_{\ge m}$ denote the sequence obtained from
$c=(c_l)_{l\ge0}$ by setting its entries with $l<m$ to zero.
For $J$ above a threshold depending only on the local chart,
$C_0+C_1$, and a bound for $\|c\|_{\ell^2}$, the following conclusions
hold uniformly in time and in $K$.

\begin{enumerate}[label=\textup{(\roman*)},leftmargin=2.5em]
\item \emph{Graph tail.} The boundary graphs satisfy
\begin{equation}
\|\eta_K-\eta_J\|_{H^s(\Ga_*)}
\le C(C_0+C_1)\|c_{\ge J}\|_{\ell^2}.             \label{eq:graph-tail}
\end{equation}
\item \emph{Geometric control and common domains.} Fix
\begin{equation*}
0<\varepsilon<\min\{1/16,(s-5/2)/4\}.
\end{equation*}
Then
\begin{align}
\sum_{l=J}^{K-1}\|\eta_{l+1}-\eta_l\|_{L^\infty(\Ga_*)}
&\le C(C_0+C_1)2^{-(s-1-\varepsilon)J}
\|c_{\ge J}\|_{\ell^2},                         \label{eq:Linfgeomtail}\\
\sum_{l=J}^{K-1}\|\eta_{l+1}-\eta_l\|_{C^{1,1/2}(\Ga_*)}
&\le C(C_0+C_1)2^{-(s-5/2-2\varepsilon)J}
\|c_{\ge J}\|_{\ell^2}.                         \label{eq:C1tail}
\end{align}
The later surfaces stay in one $C^{1,1/2}$ collar, while their
graph displacement from the $l$th surface is $o(2^{-l})$.  The fixed
operators $\mathcal J_l$ of \cref{lem:fixed-comparison-regularizer} are
therefore simultaneously defined on all later domains.

\item \emph{Bulk tail.} For $W_l^\pm=u_l\pm B_l$, the smooth comparison field
$\mathcal J_JW_J^\pm$ is defined on a neighborhood of $\Om_K$ and
\begin{equation}
\|W_K^\pm-\mathcal J_JW_J^\pm\|_{H^s(\Om_K)}
\le C(C_0+C_1)\|c_{\ge J-C}\|_{\ell^2}.          \label{eq:bulk-tail}
\end{equation}
For an infinite family satisfying the same hypotheses, the estimates
also hold on the limiting domain.
\end{enumerate}
\end{lemma}

\begin{proof}

\par\nobreak\smallskip
\noindent\textit{(i) Graph tail.}\par\nobreak

For $g_l=\eta_{l+1}-\eta_l$, \cref{prop:Dprop}\textup{(i)} and
\eqref{eq:tel-low} give
\begin{equation*}
\|g_l\|_{L^2(\Ga_*)}\le CC_0 2^{-sl}c_l.
\end{equation*}
By \eqref{eq:tel-high} and slow variation of the envelope,
\begin{equation*}
\|g_l\|_{H^4}
\le\|\eta_{l+1}\|_{H^4}+\|\eta_l\|_{H^4}
\le CC_1 2^{(4-s)l}c_l.
\end{equation*}
For a dyadic projector $P_k$ on the reference surface, these bounds yield
\begin{equation*}
2^{sk}\|P_kg_l\|_2\le C(C_0+C_1)\beta_{l-k}c_l,\qquad
\beta_m=\begin{cases}
2^{-sm},&m\ge0,\\
2^{(4-s)m},&m<0.
\end{cases}
\end{equation*}
Since $s>0$ and $4-s>0$, $\beta=(\beta_m)_{m\in\mathbb Z}\in\ell^1(\mathbb Z)$. Thus
\begin{equation*}
2^{sk}\left\|P_k\sum_{l=J}^{K-1}g_l\right\|_2
\le C(C_0+C_1)\sum_{l=J}^{K-1}\beta_{l-k}c_l.
\end{equation*}
Taking the $\ell^2$ norm in $k$ and applying Young's inequality gives
\begin{equation*}
\left\|\sum_{l=J}^{K-1}g_l\right\|_{H^s}
\le C(C_0+C_1)\|\beta\|_{\ell^1}\|c_{\ge J}\|_{\ell^2},
\end{equation*}
which proves \eqref{eq:graph-tail}.

\par\smallskip
\noindent\textit{(ii) Geometric control and common domains.}\par\nobreak

Interpolation gives
$\|g_l\|_{H^q}\le C(C_0+C_1)2^{-(s-q)l}c_l$ for $0\le q\le4$.
Use $q=1+\varepsilon$ and $q=5/2+2\varepsilon<s$, together with
$H^{1+\varepsilon}\hookrightarrow L^\infty$ and
$H^{5/2+2\varepsilon}\hookrightarrow C^{1,1/2}$ on $\Ga_*$.
Cauchy--Schwarz and the geometric series prove
\eqref{eq:Linfgeomtail}--\eqref{eq:C1tail}. The exponents
satisfy $s-1-\varepsilon>1$ and $s-5/2-2\varepsilon>0$, so the
surfaces remain in one $C^{1,1/2}$ collar and their displacement from
$\Ga_l$ is $o(2^{-l})$, uniformly over finite tails. Moreover,
\begin{equation*}
\sup_{x\in\overline\Om_K}\operatorname{dist}(x,\Om_l)
\le\|X_*\|_{L^\infty}\|\eta_K-\eta_l\|_{L^\infty}.
\end{equation*}
Increasing $J$ ensures
$\overline\Om_K\subset\Om_l^{\kappa2^{-l-3}}$ for every $J\le l\le K$.
By \eqref{eq:sampling-inside}, $\mathcal J_l$ samples
$\Om_l\cap\Om_{l+1}$ from $\Om_K$, while $\mathcal J_{l+1}$ samples
$\Om_{l+1}$. Thus $\mathcal J_l$ and $\mathcal J_{l+1}$ act on the
common terminal domain $\Om_K$.

\par\smallskip
\noindent\textit{(iii) Bulk tail.}\par\nobreak

Fix $K$. On $\Om_K$ the exact telescoping identity is
\begin{equation}\label{eq:decomp}
W_K^\pm=\mathcal J_JW_J^\pm
+\sum_{l=J}^{K-1}Z_l^\pm+(\mathrm{Id}-\mathcal J_K)W_K^\pm,
\qquad
Z_l^\pm=\mathcal J_{l+1}W_{l+1}^\pm-\mathcal J_lW_l^\pm.
\end{equation}
Uniform $H^4$ boundedness of the smoother and slow variation give
\begin{equation}\label{eq:A-high}
\|Z_l^\pm\|_{H^4(\Om_K)}\le CC_1 2^{(4-s)l}c_l.
\end{equation}
Because the kernel is independent of the domain, we can write
\begin{equation*}
Z_l^\pm=(\mathcal J_{l+1}-\mathcal J_l)W_{l+1}^\pm
+\mathcal J_l(W_{l+1}^\pm-W_l^\pm).
\end{equation*}
The first term is bounded by
\eqref{eq:fixed-smoother-difference} with $(q,r)=(0,4)$:
\begin{equation*}
\|(\mathcal J_{l+1}-\mathcal J_l)W_{l+1}^\pm\|_{L^2(\Om_K)}
\le C2^{-4l}\|W_{l+1}^\pm\|_{H^4(\Om_{l+1})}
\le CC_1 2^{-sl}c_l.
\end{equation*}
For the second, the sampling set lies in the intersection, so
\eqref{eq:fixed-smoother-core} and \eqref{eq:tel-low} yield
\begin{equation*}
\|\mathcal J_l(W_{l+1}^\pm-W_l^\pm)\|_{L^2(\Om_K)}
\le C\|W_{l+1}^\pm-W_l^\pm\|_{L^2(\Om_l\cap\Om_{l+1})}
\le CC_0 2^{-sl}c_l.
\end{equation*}
Consequently
\begin{equation}\label{eq:A-low}
\|Z_l^\pm\|_{L^2(\Om_K)}\le C(C_0+C_1)2^{-sl}c_l.
\end{equation}

Use the same Stein extension
$\mathscr E_K:=\mathscr E_{\Om_K}$ for every $Z_l^\pm$. For ambient dyadic
projectors,
\eqref{eq:A-low} and \eqref{eq:A-high} give
\begin{equation*}
2^{sk}\|P_k\mathscr E_KZ_l^\pm\|_2
\le C(C_0+C_1)\beta_{l-k}c_l,
\end{equation*}
with $\beta$ defined in part~\textup{(i)}. Young's inequality in the frequency index gives
\begin{equation}\label{eq:Asum}
\left\|\sum_{l=J}^{K-1}Z_l^\pm\right\|_{H^s(\Om_K)}
\le C(C_0+C_1)\|c_{\ge J-C}\|_{\ell^2}.
\end{equation}
The finite shift accommodates the smooth cutoffs and slow variation. The remaining term in \eqref{eq:decomp} satisfies
\begin{equation*}
\|(\mathrm{Id}-\mathcal J_K)W_K^\pm\|_{H^s(\Om_K)}
\le C2^{-(4-s)K}\|W_K^\pm\|_{H^4(\Om_K)}\le CC_1c_K.
\end{equation*}
This proves \eqref{eq:bulk-tail}.

If the surfaces converge, \eqref{eq:Linfgeomtail} ensures that the
kernels still sample strictly inside the earlier domains when evaluated
on the limiting domain. Every $Z_l^\pm$ and
$\mathcal J_JW_J^\pm$ restricts there as the same ambient kernel field.
Applying \eqref{eq:Asum} to tails of the series proves convergence in
$H^s$ on the limiting domain, uniformly in time, and yields
\eqref{eq:bulk-tail} there.
\end{proof}

\begin{remark}
The dyadic estimate retains decay in the difference of the two frequency
indices. Its $\ell^1$ kernel acts boundedly on the $\ell^2$ envelope and
yields the tail in \eqref{eq:Asum}. The common operators $\mathcal J_l$
make this summation compatible with the moving boundary.
\end{remark}

We close the $H^s$ bound for finite families of regularizations on a time
interval independent of the terminal index.

\begin{proposition}
\label{prop:E4}
After shrinking the $H^s$ data neighborhood, there are $T>0$ and an
integer $J$ such that the smooth solutions with initial data $S_jU_0$,
$j\ge J$, exist on $[0,T]$ and satisfy
\begin{equation}
\sup_{0\le t\le T}\|U_j(t)\|_{\X^4}
\le C2^{(4-s)j}c_j.
\label{eq:H4bound}
\end{equation}
Their $H^s$ norms are uniformly bounded. The surfaces remain in the fixed
collar with uniform transversality and injectivity bounds, and the Taylor
coefficients have a common positive lower bound. All constants are
uniform in $j$ and in the datum in this neighborhood.
\end{proposition}

\begin{proof}
The proof is divided into three steps.

\par\nobreak\smallskip
\noindent\textit{Step 1. Choice of bootstrap constants.}\par\nobreak

Choose a bounded initial neighborhood $\mathscr U^+$ whose surfaces
remain strictly inside a common collar, with uniform transversality and
injectivity bounds and a common positive lower bound for the Taylor
coefficient. Put
\begin{equation*}
R_0:=\sup_{U_0\in\mathscr U^+}\sup_{j\ge j_0}
\|S_jU_0\|_{\X^s}<\infty.
\end{equation*}
Fix an enlarged collar containing all initial regularized surfaces a
positive distance from its boundary. By
\cref{lem:pressure-estimates}\textup{(ii)}, shrinking $\mathscr U^+$
if necessary gives $c_*>0$ with $a_j(0)\ge2c_*$ for all initial
regularizations.

For a finite family $U_J,\ldots,U_N$, \cref{lem:telescoping} gives
\begin{equation}\label{eq:coarse-tail}
\|U_j(t)\|_{\X^s}
\le C_{\rm col}\bigl(1+\|U_J(t)\|_{\X^s}\bigr)
+C_R\|c_{\ge J-C}\|_{\ell^2},\qquad J\le j\le N,
\end{equation}
under its bootstrap hypotheses. Here $C_{\rm col}$ depends only on the
fixed chart and comparison operators, whereas $C_R$ is fixed once the
bootstrap $H^s$ bound $R$ and Taylor lower bound $c_*$ are chosen. Choose
\begin{equation*}
R>8C_{\rm col}(1+R_0).
\end{equation*}
Fix all coefficient, stability, and fourth order constants using the low
bound $1+R$ and Taylor threshold $c_*$.

Let $h=h(1+R,c_*)$ be the band width from \cref{prop:Dprop}\textup{(i)}.
Shrink the data neighborhood so its graphs lie within $h/8$ of the
reference initial graph, and increase the regularization threshold so
all regularized graphs lie within $3h/16$ of that graph. Enforce also
that all initial surfaces lie strictly inside the region where the
harmonic normal cutoff equals one. The collar and comparison operators remain fixed.

For the reference datum, choose $J$ above all regularization and sampling
thresholds so that
\begin{equation*}
C_R\|c_{\ge J-C}\|_{\ell^2}<R/8.
\end{equation*}
By continuity of the envelope, restrict further to
$\mathscr U\subset\mathscr U^+$ such that
\begin{equation}\label{eq:bootstrap-tail-small}
C_R\sup_{U_0\in\mathscr U}\|c(U_0)_{\ge J-C}\|_{\ell^2}<\frac R4.
\end{equation}
For fixed $J$, \cref{lem:compatible-regularization} and the smooth local
theory in \cite[Theorem~1.9]{IPTT-MHD} give a common time $T_J>0$ with
\begin{equation}\label{eq:coarse-small}
C_{\rm col}\left(1+
\sup_{U_0\in\mathscr U}\sup_{0\le t\le T_J}
\|U_J(t)\|_{\X^s}\right)<\frac R4.
\end{equation}
Indeed, the left side is below $R/8$ at time zero, and the fixed coarse
flow depends continuously on the datum and time.

Use the graph parametrization $G_{\eta_j}$ and the physical outward
normal $N_j$ on $\Ga_j$.
While $\|U_j\|_{\X^s}\le R$, \eqref{eq:Dt-a} and the kinematic
condition in \eqref{eq:bc} give
\begin{gather*}
\partial_t\eta_j
=\frac{(u_j\circ G_{\eta_j})\cdot(N_j\circ G_{\eta_j})}
{X_*\cdot(N_j\circ G_{\eta_j})},\\
a_j(t)\ge2c_*-C_Rt,\qquad
\|\eta_j(t)-\eta_j(0)\|_{C^0}\le C_Rt.
\end{gather*}
Interpolation with the $H^s$ bound makes the graph change small in
$C^{1,\rho}$ for $0<\rho<\min\{1,s-2\}$.
Choose $T_{\rm geo}>0$ with $C_RT_{\rm geo}\le h/32$, so each graph
stays within $7h/32<h/4$ of the reference graph. Decrease it to keep
the graphs strictly inside the enlarged collar, preserve the cutoff
support condition, and ensure $a_j\ge3c_*/2$. Set
\begin{equation*}
T=\min\{1,T_J,T_{\rm geo}\}.
\end{equation*}

\par\smallskip
\noindent\textit{Step 2. Bootstrap a finite family.}\par\nobreak

Fix $N\ge J$. Let $T_N$ be the supremum of $0<t\le T$ for which every
$U_j$, $J\le j\le N$, exists on $[0,t]$ and satisfies
\begin{equation*}
\|U_j\|_{\X^s}<R,\qquad a_j>c_*,
\end{equation*}
together with strict containment in the enlarged collar, a graph band of width $h/4$,
adjacent weak band, and cutoff support conditions. These conditions hold
strictly at the initial time, so smooth local existence gives $T_N>0$.

On $[0,T_N)$, \eqref{eq:E4local}, \cref{prop:Dprop}\textup{(ii)}, and
\cref{lem:telescoping} apply. Equations
\eqref{eq:coarse-tail}--\eqref{eq:coarse-small} imply
\begin{equation*}
\sup_{0\le t<T_N}\|U_j(t)\|_{\X^s}<R/2,\qquad J\le j\le N.
\end{equation*}
The choice of $T_{\rm geo}$ improves the remaining bootstrap bounds:
$a_j\ge3c_*/2$, all surfaces stay a positive distance from the collar and
cutoff thresholds, and each graph lies within $7h/32$ of the reference
graph. Adjacent differences are thus at most $7h/16<h/2$.
No bootstrap quantity can reach its stopping value. If $T_N<T$, one
member of this finite family must therefore lose smooth existence.

\par\smallskip
\noindent\textit{Step 3. Exclude loss of smooth existence.}\par\nobreak

Suppose such a lifespan ends at $t_*\le T$. For that fixed $j$,
\eqref{eq:E4local} gives
\begin{equation*}
\sup_{t<t_*}\|U_j(t)\|_{\X^4}
\le C_R2^{(4-s)j}c_j<\infty.
\end{equation*}
The $H^3$ algebra estimate gives
\begin{equation*}
\|D_{B_j}u_j\|_{H^3}+\|D_{B_j}B_j\|_{H^3}
\le C\|B_j\|_{H^4}(\|u_j\|_{H^4}+\|B_j\|_{H^4}),
\end{equation*}
so, including the standard $H^{7/2}$ components,
\begin{equation}\label{eq:enh-72-bound}
\|U_j(t)\|_{\mathcal Y^{7/2}}
\le C\bigl(1+\|U_j(t)\|_{\X^4}\bigr)^2.
\end{equation}
The continuation criterion \cite[Theorem~1.11]{IPTT-MHD}, written
in the enhanced spaces $\mathcal Y^r$ with $r,m$ in place of
$s,\sigma$, extends a $\mathcal Y^m$ solution while the surfaces
remain in a common collar with uniform transversality and
injectivity bounds, the Taylor coefficient stays bounded away
from zero, and the $\mathcal Y^r$ norm remains bounded for some
$5/2<r\le m$. With $r=7/2$ and $m=4$, Step~2 and
\eqref{eq:enh-72-bound} therefore give an $\mathcal Y^4$ extension
to $t_*+\epsilon_j$ for some $\epsilon_j>0$, with the same collar
conditions and a positive lower bound for the Taylor coefficient.

To retain smoothness, fix an integer $m>4$ and let $T_m$ be the maximal
$\mathcal Y^m$ lifespan of the same smooth datum. Uniqueness in
$\mathcal Y^4$ identifies this solution with the extension on their
common interval. If $T_m\le t_*+\epsilon_j/2$, the extension controls its
lower $\mathcal Y^{7/2}$ norm, preserves the collar conditions, and
keeps the Taylor coefficient uniformly positive up to $T_m$.
\cite[Theorem~1.11]{IPTT-MHD}, applied at index $m$, then extends it
past $T_m$, a contradiction. Hence all finite orders persist to
$t_*+\epsilon_j/2$ and agree by uniqueness, contradicting termination
of the smooth solution.

Thus $T_N=T$ for every finite $N$. Since $T$ and all constants are
independent of $N$, every $j\ge J$ exists on $[0,T]$, and
\eqref{eq:E4local} proves \eqref{eq:H4bound}. The finitely many coarser
regularizations may be discarded.
\end{proof}

\subsection{Frequency envelopes and the \texorpdfstring{$H^s$}{Hs} flow}

On $[0,T]$, \cref{prop:Dprop}\textup{(ii)} and \cref{prop:E4} give the low and high bounds
\eqref{eq:tel-low}--\eqref{eq:tel-high}. By
\cref{lem:telescoping}, there is a uniform $J_*$ such that
for every $K>J\ge J_*$ and every $0\le t\le T$,
\begin{equation}
\|\eta_K(t)-\eta_J(t)\|_{H^s}
+\|W_K^\pm(t)-\mathcal J_JW_J^\pm(t)\|_{H^s(\Om_K(t))}
\le C\|c_{\ge J-C}\|_{\ell^2}.                        \label{eq:uniformtail}
\end{equation}
\begin{proposition}\label{prop:persistence}
\textup{(i)} For every datum $U_0$ in a sufficiently small $H^s$
neighborhood, the smooth regularizations converge in
$C([0,T];\X^s)$ to a solution $\Phi(U_0)$ with initial state $U_0$
and positive Taylor sign.

\textup{(ii)} Two solutions in $C([0,T];\X^s)$ with the same initial
state and positive Taylor sign coincide on their common interval in the
chart.
\end{proposition}

\begin{proof}
The proof is divided into five steps.

\par\nobreak\smallskip
\noindent\textit{Step 1. Construct the strong $H^s$ limit.}\par\nobreak

By \eqref{eq:uniformtail}, the graphs converge in
$C([0,T];H^s(\Ga_*))$ to a limit $\eta$. Hence the surfaces converge
uniformly in $C^1$ within the fixed collar to $\Ga(t)=\Ga_{\eta(t)}$,
which encloses $\Om(t)$.

For a fixed scale $L$, set
\begin{equation*}
F_L^\pm(t)=\mathcal J_LW_L^\pm(t).
\end{equation*}
The sampling property in \cref{lem:telescoping} shows that this field is
defined on a neighborhood of $\Om(t)$ and of every
sufficiently late $\Om_K(t)$.  For $M>L$, applying the version on the limit domain of \eqref{eq:Asum} gives
\begin{equation*}
\sup_{0\le t\le T}
\|F_M^\pm(t)-F_L^\pm(t)\|_{H^s(\Om(t))}
\le C\|c_{\ge L-C}\|_{\ell^2}.
\end{equation*}
Denote the $H^s(\Om(t))$ limit by $W^\pm(t)$; the Cauchy estimate is
uniform in time. Set
\begin{equation*}
u=\frac{W^++W^-}{2},
\qquad B=\frac{W^+-W^-}{2}.
\end{equation*}
Passing to the limit in the finite telescoping estimate gives
\begin{equation}\label{eq:finite-comparison-tail}
\begin{aligned}
\sup_{K\ge L}\sup_t\|W_K^\pm-F_L^\pm\|_{H^s(\Om_K(t))}
&\le C\|c_{\ge L-C}\|_{\ell^2},\\
\sup_t\|W^\pm-F_L^\pm\|_{H^s(\Om(t))}
&\le C\|c_{\ge L-C}\|_{\ell^2}.
\end{aligned}
\end{equation}
For large $L$, the single smooth field $F_L^\pm$ approximates both the
limit and every sufficiently late regularization in \cref{def:stdtop}.
Together with the graph convergence, this gives strong moving domain
$H^s$ convergence. To check continuity at $t_0$, first make the tails in
\eqref{eq:finite-comparison-tail} small, then use smoothness of
$F_L^\pm(t)$ on a common neighborhood of $\overline\Om(t_0)$ for nearby
$t$. The triangle inequality gives continuity in the comparison topology.
At $t=0$, \eqref{eq:regconv} identifies the limit with $U_0$.

\par\smallskip
\noindent\textit{Step 2. Pass the constraints and the pressure to the limit.}\par\nobreak

Use the fixed maps \eqref{eq:reference-domain-map},
$\Xi_j(t)=\Xi_{\eta_j(t)}$ and $\Xi(t)=\Xi_{\eta(t)}$.
The graph convergence and the uniform
$H^s$ bounds imply
\begin{equation}\label{eq:Xij-uniform-time}
\Xi_j\to\Xi\quad\hbox{in }C_tC_x^1,
\qquad
\sup_j\|\Xi_j\|_{L_t^\infty(H^{s+1/2}\cap C^{1,\alpha})}
\le C_R
\end{equation}
where we choose $1/4<\alpha<1/2<s-2$; the $C^{1,\alpha}$ bound
follows from the uniform $H^{s+1/2}$ bound by Sobolev embedding.
Here $C_tC_x^1$ means
$C([0,T];C^1(\overline\Om_*))$, and the spatial norms in
\eqref{eq:Xij-uniform-time} are taken on $\Om_*$. We next establish
the stronger convergence
\begin{equation}\label{eq:fields-C1-pullback}
u_j\circ\Xi_j\to u\circ\Xi,
\qquad
B_j\circ\Xi_j\to B\circ\Xi
\quad\hbox{in }C_tC_x^1.
\end{equation}
Fix a large comparison scale $L$.  The tail
\eqref{eq:finite-comparison-tail}, Sobolev embedding
$H^s\hookrightarrow C^1$, and the uniform extension constants show that
both $W_j^\pm$ and $W^\pm$ are close in $C^1$ to the same smooth field
$F_L^\pm$.  For fixed $L$, composition of $F_L^\pm$ with $\Xi_j$ converges
in $C^1$ to composition with $\Xi$ by \eqref{eq:Xij-uniform-time}.
First let $j\to\infty$ and then $L\to\infty$ to obtain
\eqref{eq:fields-C1-pullback}.

The divergence constraints pass to the limit distributionally on every
compact set contained in the limiting domain.  The trace convergence in
\eqref{eq:fields-C1-pullback}, together with uniform convergence of the
normals, gives $B\cdot N=0$ on $\Ga$.

Let
$f_j=-\operatorname{tr}((\nabla u_j)^2)+\operatorname{tr}((\nabla B_j)^2)$
be the pressure source, and let $f$ be the analogous expression for
$(u,B)$. Denote by $p$ the zero Dirichlet solution of $\Delta p=f$ on
$\Om(t)$, and set
\begin{equation*}
\pi_j=p_j\circ\Xi_j,
\quad
J_j=\det D\Xi_j,
\quad
H_j=J_j(D\Xi_j)^{-1}(D\Xi_j)^{-T}.
\end{equation*}
Define $\pi=p\circ\Xi$, $J=\det D\Xi$, and
$H=J(D\Xi)^{-1}(D\Xi)^{-T}$ for the limiting map.
The sources satisfy $f_j\circ\Xi_j\to f\circ\Xi$ uniformly in
spacetime. The fixed domain energy estimate from
\cref{lem:pressure-estimates}\textup{(ii)}, applied at each time, gives
\begin{align*}
\|\pi_j-\pi\|_{H^1_0}
&\le C_R\Bigl(
\|H_j-H\|_\infty\|\nabla\pi\|_2
+\|J_j(f_j\circ\Xi_j)-J(f\circ\Xi)\|_{H^{-1}}\Bigr).
\end{align*}
All coefficient and source differences tend to zero uniformly in time;
therefore
\begin{equation}\label{eq:pi-uniform-H1}
\sup_{0\le t\le T}\|\pi_j(t)-\pi(t)\|_{H^1_0(\Om_*)}
\longrightarrow0.
\end{equation}
The uniform $C^{1,1/4}$ bounds after pullback and compactness in $C^1$
then yield
\begin{equation*}
\sup_t\|\pi_j(t)-\pi(t)\|_{C^1}\to0.
\end{equation*}
Indeed, a sequence of times violating this convergence would have
$C^1$ convergent subsequences with identical limits by
\eqref{eq:pi-uniform-H1}.
Because
$(\nabla p_j)\circ\Xi_j=(D\Xi_j)^{-T}\nabla\pi_j$, the normals and
inverse derivatives converge uniformly, and therefore the Taylor
coefficients converge uniformly.  Their common positive lower bound
passes to the limit.

\par\smallskip
\noindent\textit{Step 3. Pass the equations and the kinematic condition.}\par\nobreak

In the fixed graph chart the kinematic relation is
\begin{equation*}
\partial_t\eta_j(y)
=\frac{u_j(G_{\eta_j}(y))\cdot N_j(G_{\eta_j}(y))}
{X_*(y)\cdot N_j(G_{\eta_j}(y))}.
\end{equation*}
The denominator has a uniform positive lower bound by transversality, and
the numerator converges uniformly by \eqref{eq:fields-C1-pullback}.
Integrating from $0$ to $t$ and passing to the limit gives the same
kinematic identity for $\eta$.

Every compact spacetime cylinder in the limiting fluid interior lies
in all sufficiently late approximate domains. On such cylinders, the
$C^1$ convergence of the fields and pressure passes the nonlinear terms
to the limit, while the time derivatives converge distributionally. Thus
\begin{equation*}
D_tu=B\cdot\nabla B-\nabla p,
\qquad
D_tB=B\cdot\nabla u
\end{equation*}
in distributions on the limiting spacetime region.  The spatial
regularity $u,B\in H^s$ and the pressure estimates place the right sides
in $H^{s-1}$, so these are equations in the natural Sobolev class.  The
zero pressure trace and magnetic tangency were already obtained.  Thus
the limit is a solution in $C([0,T];\X^s)$, denoted $\Phi(U_0)$.

\par\smallskip
\noindent\textit{Step 4. Pressure material derivative for a rough solution.}\par\nobreak

Let now $U\in C([0,T];\X^s)$ be any solution in the class of the
proposition. Its state norms are bounded on the compact time interval.
Choose $R\ge1+\sup_t\|U(t)\|_{\X^s}$.
The stationary part of \cref{lem:pressure-estimates}\textup{(i)} gives
$p\in L^\infty_tH^{s+1/2}$, with a bound depending only on these state
norms and the fixed chart. Extend $u$ to a continuous family of
compactly supported $C^1$ ambient fields and let $Y(t)$ be its flow.
The kinematic condition and uniqueness of characteristics give
$Y(t)(\Ga_0)=\Ga(t)$. Since $Y(t)$ is an ambient homeomorphism, it maps
the bounded component enclosed by $\Ga_0$ onto that enclosed by
$\Ga(t)$. Thus $Y(t)(\Om_0)=\Om(t)$. The uniform $C^1$ bounds give
uniformly Lipschitz flow maps and inverses, and the divergence constraint
along interior trajectories gives $\det DY=1$.
Set
\begin{equation*}
H_L=(DY)^{-1}(DY)^{-T},
\qquad \varpi=p\circ Y,
\qquad f=\Delta p,
\qquad F=f\circ Y.
\end{equation*}
Since $\nabla u\circ Y\in C_tL^\infty_x$, the equation
$\partial_tDY=(\nabla u\circ Y)DY$ gives
$H_L\in C^1([0,T];L^\infty(\Om_0))$. In particular,
\begin{equation*}
H_L\in W^{1,\infty}(0,T;L^\infty(\Om_0)),
\qquad
\lambda|\xi|^2\le H_L(t,x)\xi\cdot\xi
\le\lambda^{-1}|\xi|^2
\end{equation*}
for a fixed $\lambda>0$, all $\xi\in\R^3$, and almost every $(t,x)$.

Equations \eqref{eq:MHD} and \cref{lem:pressure-estimates}\textup{(i)}
give $D_tu,D_tB\in L^\infty_tH^{s-1}$.  Commuting one spatial
derivative with $D_t$ and differentiating the quadratic expression
$f=\operatorname{tr}((\nabla B)^2)-\operatorname{tr}((\nabla u)^2)$
therefore yields
\begin{equation*}
D_tf\in L^\infty_tH^{s-2}.
\end{equation*}
Indeed, each differentiated term is a product of an $H^{s-1}$ factor
and an $H^{s-2}$ factor, and $s-1>3/2$.  These identities hold
distributionally by the product rule and spatial approximation.
Pullback of the transport identity by the Lipschitz flow gives
\begin{equation}\label{eq:transport-source-pullback}
F\in W^{1,\infty}(0,T;L^2(\Om_0)),
\qquad \partial_tF=(D_tf)\circ Y.
\end{equation}
To justify \eqref{eq:transport-source-pullback}, let $w$ be any component
of $\nabla u$ or $\nabla B$, and set $g=D_tw\in L^\infty_tL^2$.
Test its weak transport equation against
$\psi(t,Y(t)^{-1}(x))$, where $\psi\in C_c^\infty((0,T)\times\Om_0)$.
These $C^1$ tests are admissible by approximation. Volume preservation
gives $\partial_t(w\circ Y)=g\circ Y$ distributionally, so
$w\circ Y\in W^{1,\infty}(0,T;L^2(\Om_0))$.
The pullbacks are also continuous in time in $L^\infty$, and the product
rule for their quadratic products yields $\partial_tF=(D_tf)\circ Y$.

On the fixed domain, $\varpi\in L^\infty_tH^1_0(\Om_0)$ satisfies
\begin{equation*}
\int_{\Om_0}H_L(t)\nabla\varpi(t)\cdot\nabla\varphi
=-\int_{\Om_0}F(t)\varphi,
\qquad \varphi\in H^1_0(\Om_0).
\end{equation*}
For a time-dependent function $v$ and $0<h<T$, write
$\delta_hv(t)=(v(t+h)-v(t))/h$ for $0<t<T-h$.
Subtracting the weak equations at $t+h$ and $t$ gives
\begin{equation*}
\int_{\Om_0}H_L(t+h)\nabla\delta_h\varpi\cdot\nabla\varphi
=-\int_{\Om_0}\delta_hF\,\varphi
-\int_{\Om_0}\delta_hH_L\nabla\varpi(t)\cdot\nabla\varphi.
\end{equation*}
Choose $\varphi=\delta_h\varpi$.  Coercivity and Poincar\'e's
inequality imply, uniformly for almost every $0<t<T-h$,
\begin{equation*}
\|\delta_h\varpi(t)\|_{H^1_0}
\le C\Bigl(
\|\delta_hF(t)\|_{H^{-1}}
+\|\delta_hH_L(t)\|_{L^\infty}
\|\varpi(t)\|_{H^1_0}\Bigr)
\le C_R.
\end{equation*}
Weak star compactness on interior time intervals identifies the limit of
these difference quotients with the distributional derivative of
$\varpi$.  Hence
\begin{equation*}
\varpi\in W^{1,\infty}(0,T;H^1_0(\Om_0)).
\end{equation*}
Passing to the limit in the time integrated weak equation gives, for
almost every $t$,
\begin{equation*}
\int_{\Om_0}H_L\nabla\varpi_t\cdot\nabla\varphi
=-\int_{\Om_0}(D_tf)\circ Y\,\varphi
-\int_{\Om_0}\partial_tH_L\nabla\varpi\cdot\nabla\varphi.
\end{equation*}
The passage uses the strong convergence of $\delta_hH_L$ in
$L^q_tL^\infty_x$ and of $\delta_hF$ in $L^q_tL^2_x$ on compact time
intervals, for every finite $q$, together with the uniform $H^1_0$
bound for $\delta_h\varpi$.

Define $Q(t)=\varpi_t(t)\circ Y(t)^{-1}$.  Volume preservation and
change of variables show that $Q=D_tp$ distributionally on the moving
domain.  Moreover, $Q\in L^\infty_tH^1_0(\Om(t))$, so its boundary
trace is zero for almost every time. The commutator identity
\eqref{eq:laplace-Dt-comm} gives
\begin{equation*}
\Delta Q
=D_tf+(\Delta u)\cdot\nabla p
+2\partial_ku_i\,\partial_{ki}p
\quad\hbox{in }\Om(t),
\qquad Q|_{\Ga(t)}=0.
\end{equation*}
All terms on the right belong to $L^\infty_tH^{s-2}$.  For the last
two terms, use $\nabla u\in H^{s-1}$ and
$\nabla p\in H^{s-1/2}$ as Sobolev multipliers, together with
$\nabla^2u\in H^{s-2}$ and $\nabla^2p\in H^{s-3/2}$.
Thus \cref{lem:elliptic-estimates}\textup{(i)} yields
\begin{equation*}
Q\in L^\infty_tH^s(\Om(t))
\hookrightarrow L^\infty_tW^{1,\infty}(\Om(t)).
\end{equation*}
By \cref{lem:pressure-estimates}\textup{(i)},
$D_Bp\in L^\infty_tH^s$.
Consequently
\begin{equation*}
D_t^\pm p=Q\pm D_Bp
\in L^\infty_tW^{1,\infty}(\Om(t)).
\end{equation*}
These bounds at almost every time suffice for the time integral in the
uniqueness criterion of \cite[Theorem~1.6]{IPTT-MHD}.

\par\smallskip
\noindent\textit{Step 5. Uniqueness.}\par\nobreak

Fix a sufficiently small $0<\epsilon<1/4$. Sobolev embedding gives
$u,B\in L^\infty_tW^{1,\infty}_x$ and
$\Ga\in L^\infty_tC^{1,1/2}$, hence uniform $C^{1/2+\epsilon}$ and
$C^{1,\epsilon}$ bounds for the fields and boundary, respectively.
Step~4 gives $D_t^\pm p\in L^\infty_tW^{1,\infty}_x$.
With the common collar and positive Taylor lower bound,
\cite[Theorem~1.6]{IPTT-MHD} therefore gives uniqueness on each
common local chart interval. Continuity in $\X^s$ and an open--closed
argument extend this equality throughout the common lifespan,
proving \cref{prop:persistence}\textup{(ii)}.
\end{proof}

\subsection{Equivalence of the Eulerian metric and the state topology}

\begin{lemma}
\label{lem:extmetric-top}
For admissible states in every bounded subchart of a fixed collar,
convergence in $d$ is equivalent to $\X^s$ convergence in
\cref{def:stdtop}. The equivalence also
holds for uniform convergence of continuous trajectories on a compact
time interval.
\end{lemma}

\begin{proof}
The proof is divided into four steps.

\par\nobreak\smallskip
\noindent\emph{Step 1. From comparison convergence to metric convergence.}\par\nobreak

Assume $U_n\to U$ in the sense of \cref{def:stdtop}.  For
$\varepsilon>0$, choose the smooth divergence-free comparison fields
$\widetilde u,\widetilde B$ in item~(iii), with both comparison errors
less than $\varepsilon$.  Uniform extension bounds give
\begin{align*}
\|E_{\Om_n}u_n-E_\Om u\|_{H^s}
&\le C\|u_n-\widetilde u\|_{H^s(\Om_n)}\\
&\quad+
\|E_{\Om_n}(\widetilde u|_{\Om_n})
-E_\Om(\widetilde u|_\Om)\|_{H^s}
+C\|u-\widetilde u\|_{H^s(\Om)}.
\end{align*}
Apply the cutoff convention in Section~\ref{subsec:state-topology}
to regard $\widetilde u$ as a fixed smooth ambient field.
Then \eqref{eq:extension-continuity} makes the middle term tend to zero.
Taking the limsup gives at most $C\varepsilon$.
The same argument applies to $B$, while the graph term converges by
item~(ii).  Letting $\varepsilon\downarrow0$ proves $d(U_n,U)\to0$.

\par\smallskip
\noindent\emph{Step 2. Continuity of restrictions with respect to the domain.}\par\nobreak

Assume now $d(U_n,U)\to0$. We use the following consequence of graph
convergence: if $G\in H^s(\R^3)$ vanishes almost everywhere on
$\Om$, then
\begin{equation}\label{eq:smooth-domain-continuity}
\|G\|_{H^s(\Om_n)}\longrightarrow0.
\end{equation}
Write $\mathbf1_E$ for the indicator of a measurable set $E$.
Set $m=\lfloor s\rfloor$ and $\vartheta=s-m\in[0,1)$.
Each $D^\alpha G$, $|\alpha|\le m$, vanishes almost everywhere on $\Om$.
Since $\mathbf1_{\Om_n}\to\mathbf1_\Om$ almost everywhere, dominated
convergence gives $\|D^\alpha G\|_{L^2(\Om_n)}\to0$.
This proves the claim when $\vartheta=0$. For $\vartheta>0$ and $|\alpha|=m$,
the fractional seminorm is
\begin{equation*}
[D^\alpha G]_{H^\vartheta(\Om_n)}^2
=\iint_{\R^3\times\R^3}
\frac{|D^\alpha G(x)-D^\alpha G(y)|^2}
{|x-y|^{3+2\vartheta}}
\mathbf1_{\Om_n}(x)\mathbf1_{\Om_n}(y)\,dx\,dy.
\end{equation*}
The fraction in the integrand belongs to $L^1(\R^3\times\R^3)$, and the integrand
converges almost everywhere to zero because $D^\alpha G=0$ on $\Om$.
Dominated convergence and the uniform equivalence of intrinsic and
restriction Sobolev norms on the Lipschitz collar prove
\eqref{eq:smooth-domain-continuity}.

\par\smallskip
\noindent\emph{Step 3. From metric convergence to comparison convergence.}\par\nobreak

The strong $H^s$ convergence in
\cref{lem:fixed-comparison-regularizer}\textup{(ii)}, its preservation of
divergence, and the cutoff convention in
Section~\ref{subsec:state-topology} allow us to choose
$\widetilde u\in C_c^\infty(\R^3)$, divergence-free on a neighborhood of
$\overline\Om$ containing all sufficiently late $\Om_n$, such that
\begin{equation*}
\|u-\widetilde u\|_{H^s(\Om)}<\varepsilon.
\end{equation*}
Set $G=E_\Om(\widetilde u|_\Om)-\widetilde u$, so that
$G\in H^s(\R^3)$ and $G=0$ on $\Om$. Restricting
\begin{equation*}
E_{\Om_n}u_n-\widetilde u
=(E_{\Om_n}u_n-E_\Om u)
+E_\Om(u-\widetilde u|_\Om)+G
\end{equation*}
to $\Om_n$ gives
\begin{equation*}
\|u_n-\widetilde u\|_{H^s(\Om_n)}
\le C\|E_{\Om_n}u_n-E_\Om u\|_{H^s(\R^3)}
+C\varepsilon+\|G\|_{H^s(\Om_n)}.
\end{equation*}
Metric convergence and \eqref{eq:smooth-domain-continuity} therefore yield
\begin{equation*}
\limsup_{n\to\infty}
\|u_n-\widetilde u\|_{H^s(\Om_n)}\le C\varepsilon.
\end{equation*}
The same argument yields a smooth divergence-free comparison field for
$B$.  Letting $\varepsilon\downarrow0$ gives item~(iii) of
\cref{def:stdtop}. Item~(ii) is already one of the terms of $d$.

It remains to verify the common strict Taylor sign required in item~(i).
Metric convergence gives graph convergence and ambient field convergence
in $H^s$, hence also in every weaker $H^{s'}$ topology with
$5/2<s'<s$.  By \cref{lem:pressure-estimates}\textup{(ii)}, the pulled back normal
pressure coefficients converge uniformly. Since $U$ is admissible, its
Taylor coefficient has a positive minimum $c_*>0$, and those of $U_n$
are at least $c_*/2$ for all sufficiently large $n$. Each of the finitely
many remaining states is admissible by assumption. Taking
the minimum of their positive Taylor lower bounds and $c_*/2$ gives a
common positive lower bound for the entire sequence. The condition on
the fixed collar follows from the graph convergence.
Thus the two state topologies agree.

\par\smallskip
\noindent\emph{Step 4. Uniform convergence of trajectories.}\par\nobreak

For continuous trajectories $U_n,U:I=[0,T]\to\X^s$ in a bounded
subchart of the fixed collar, the uniform version of \cref{def:stdtop} requires a
common positive Taylor lower bound,
\begin{equation*}
\sup_{t\in I}\|\eta_n(t)-\eta(t)\|_{H^s(\Ga_*)}\longrightarrow0,
\end{equation*}
and the following comparison property: for every $\varepsilon>0$ there
are a finite relatively open cover $I_1,\ldots,I_M$ of $I$, an integer
$n_0$, and pairs $(\widetilde u_a,\widetilde B_a)\in C_c^\infty(\R^3)$,
divergence-free on a common neighborhood of all
$\overline{\Om(t)}$ and $\overline{\Om_n(t)}$ with $t\in I_a$, $n\ge n_0$,
such that
\begin{align}
&\sup_{t\in I_a}\left(
\|u(t)-\widetilde u_a\|_{H^s(\Om(t))}
+\|B(t)-\widetilde B_a\|_{H^s(\Om(t))}\right)<\varepsilon,
\label{eq:uniform-comparison-limit}\\
&\sup_{n\ge n_0}\sup_{t\in I_a}\left(
\|u_n(t)-\widetilde u_a\|_{H^s(\Om_n(t))}
+\|B_n(t)-\widetilde B_a\|_{H^s(\Om_n(t))}\right)<\varepsilon.
\label{eq:uniform-comparison-sequence}
\end{align}
Suppose these conditions hold. By \eqref{eq:extension-continuity},
compactness of $\{\eta(t):t\in I\}$ and uniform graph convergence give,
for each fixed $F\in C_c^\infty(\R^3)$,
\begin{equation}\label{eq:uniform-extension-fixed-field}
\sup_{t\in I}
\|E_{\Om_n(t)}(F|_{\Om_n(t)})
-E_{\Om(t)}(F|_{\Om(t)})\|_{H^s(\R^3)}\longrightarrow0.
\end{equation}
For any sequence $t_n\in I$, pass to a subsequence with $t_n\to t_*$.
Both graph sequences then converge to $\eta(t_*)$ in $H^s$, and
\eqref{eq:extension-continuity} gives the same limit for the two
extensions. This proves \eqref{eq:uniform-extension-fixed-field}.
Apply the triangle inequality from Step~1 and
\eqref{eq:uniform-extension-fixed-field} to the finitely many comparison
fields.  Uniform graph convergence then gives
\begin{equation*}
\limsup_{n\to\infty}\sup_{t\in I}d(U_n(t),U(t))
\le C\varepsilon.
\end{equation*}
Letting $\varepsilon\downarrow0$ proves uniform metric convergence.

Conversely, suppose $\sup_{t\in I}d(U_n(t),U(t))\to0$.
Steps~2 and~3 give, for every state $V=(v,b,\Sigma)$ and $\varepsilon>0$,
smooth comparison fields $(\widetilde v,\widetilde b)$ and $r_V>0$ such
that every state $W=(w,b',\Sigma_W)$ with $d(W,V)<r_V$ satisfies
\begin{equation}\label{eq:comparison-neighborhood}
\|w-\widetilde v\|_{H^s(\Om_W)}
+\|b'-\widetilde b\|_{H^s(\Om_W)}<\varepsilon/2.
\end{equation}
Here $\Om_V$ and $\Om_W$ are the domains enclosed by $V$ and $W$.
For example, set $G_V=E_{\Om_V}(\widetilde v|_{\Om_V})-\widetilde v$.
The decomposition in Step~3, restricted to $\Om_W$, gives
\begin{equation*}
\|w-\widetilde v\|_{H^s(\Om_W)}
\le C d(W,V)+C\|v-\widetilde v\|_{H^s(\Om_V)}
+\|G_V\|_{H^s(\Om_W)}.
\end{equation*}
Since $G_V=0$ on $\Om_V$, the last term tends to zero as $d(W,V)\to0$
by Step~2. Choose the comparison errors at $V$ sufficiently small and
then decrease $r_V$, applying the same estimate to the magnetic field.
This gives \eqref{eq:comparison-neighborhood}; graph convergence also
keeps the domain closures in a common neighborhood where the comparison
fields are divergence-free.

Apply \eqref{eq:comparison-neighborhood} at each $V=U(t_0)$.
By Steps~1--3, $U$ is continuous for $d$.  Thus each $t_0$ has a
relatively open neighborhood $I_{t_0}$ on which
$d(U(t),U(t_0))<r_{U(t_0)}/3$.  Choose a finite subcover
$I_a=I_{t_a}$.  Uniform metric convergence gives one $n_0$ such that
\begin{equation*}
\sup_{t\in I}d(U_n(t),U(t))
<\tfrac13\min_{1\le a\le M}r_{U(t_a)},
\qquad n\ge n_0.
\end{equation*}
Both states then lie in the corresponding $r_{U(t_a)}$ ball on $I_a$,
so \eqref{eq:comparison-neighborhood} gives
\eqref{eq:uniform-comparison-limit}--\eqref{eq:uniform-comparison-sequence}.
The metric already gives uniform graph convergence.

Finally, let $\widetilde a_n(t,y)$ and $\widetilde a(t,y)$ denote the Taylor
coefficients of $U_n(t)$ and $U(t)$ pulled back to $\Ga_*$ by their
respective graph parametrizations. Uniform convergence in $d$ implies
uniform convergence in its $H^{s'}$ version for $5/2<s'<s$.
By \cref{lem:pressure-estimates}\textup{(ii)}, these functions are continuous
on $I\times\Ga_*$ and satisfy
\begin{equation*}
\sup_{(t,y)\in I\times\Ga_*}
|\widetilde a_n(t,y)-\widetilde a(t,y)|\longrightarrow0.
\end{equation*}
Admissibility and compactness yield
$c_I:=\min_{I\times\Ga_*}\widetilde a>0$, so $\widetilde a_n\ge c_I/2$
for all large $n$.  Each remaining trajectory has its own positive
minimum on $I\times\Ga_*$.  The minimum of these finitely many numbers
and $c_I/2$ supplies a common positive lower bound for the Taylor coefficient.  This proves the
uniform equivalence; continuity of individual trajectories follows from
the state equivalence.
\end{proof}

\subsection{Continuous dependence}

\begin{proposition}\label{prop:cont}
\textup{(i)} Let $V_n(t)=(u_n,B_n,\Ga_n)$ and
$V(t)=(u,B,\Ga)$ be smooth states on $[0,T]$.  Write
$\eta_n:=\eta_{\Ga_n}$, $\eta:=\eta_\Ga$, and let $\Om_n,\Om$ be the
corresponding fluid domains.

Assume that the states remain in one bounded local chart and have a common
positive lower bound for the Taylor coefficient. Suppose that each pair
$(V_n,V)$ satisfies \eqref{eq:weak-band} and that
\begin{equation*}
\sup_n\sup_{0\le t\le T}
\bigl(\|V_n(t)\|_{\X^4}+\|V(t)\|_{\X^4}\bigr)\le R_4.
\end{equation*}
If
\begin{equation}
\sup_{0\le t\le T}\D(V_n(t),V(t))^{1/2}\longrightarrow0,
\label{eq:Dtwozero}
\end{equation}
then
\begin{equation*}
V_n\longrightarrow V\qquad\text{in }C([0,T];\X^s).
\end{equation*}

\textup{(ii)} After shrinking the neighborhood if necessary, the map
\begin{equation*}
U_0\longmapsto\Phi(U_0)
\end{equation*}
is continuous from the moving domain $H^s$ topology to
$C([0,T];\X^s)$.
\end{proposition}

\begin{proof}
The proof is divided into three steps.

\par\nobreak\smallskip
\noindent\textit{Step 1. Interpolation from weak distance under a common $H^4$
bound.}\par\nobreak

By \cref{prop:Dprop}\textup{(i)}, \eqref{eq:Dtwozero} implies
\begin{equation*}
\sup_t\|\eta_n(t)-\eta(t)\|_{L^2(\Ga_*)}\longrightarrow0.
\end{equation*}
Interpolation with the common $H^4$ bound gives
$\|\eta_n-\eta\|_{H^s}
\le C(R_4)\|\eta_n-\eta\|_{L^2}^{1-s/4}$, hence
\begin{equation}\label{eq:two-state-graph-Hs}
\sup_t\|\eta_n(t)-\eta(t)\|_{H^s}\longrightarrow0.
\end{equation}
In particular, the graphs converge uniformly in $C^1$, and the measure of
$\Om_n(t)\triangle\Om(t)$ tends to zero uniformly in time.

Fix a large dyadic scale $m$ and put
\begin{equation*}
\widetilde u_m=\mathcal J_mu,
\qquad
\widetilde B_m=\mathcal J_mB.
\end{equation*}
By \cref{lem:fixed-comparison-regularizer}, these fields are smooth on a
common enlargement of $\Om(t)$ and
\begin{align}
\sup_t\bigl(\|u-\widetilde u_m\|_{H^s(\Om)}
+\|B-\widetilde B_m\|_{H^s(\Om)}\bigr)
&\le CR_4 2^{-(4-s)m},\notag\\
\sup_t\bigl(\|u-\widetilde u_m\|_{L^2(\Om)}
+\|B-\widetilde B_m\|_{L^2(\Om)}\bigr)
&\le CR_4 2^{-4m}.
\label{eq:two-state-smooth-L2}
\end{align}
For each fixed $m$, uniform convergence of the graphs gives
$\Om_n(t)\subset\Om(t)^{\kappa2^{-m}}$ for all sufficiently large $n$,
uniformly in $t$.  Thus $u_n-\widetilde u_m$ is defined on $\Om_n$.
On $\Om_n\cap\Om$, the $L^2$ difference is controlled by
$\D^{1/2}$ and \eqref{eq:two-state-smooth-L2}. On the excess
$\Om_n\setminus\Om$, it tends to zero uniformly in time, since
$u_n,\widetilde u_m$ are bounded in $L^\infty$ at fixed $m$ and
$|\Om_n\triangle\Om|\to0$. Consequently
\begin{equation}\label{eq:two-state-L2-full}
\limsup_{n\to\infty}\sup_t
\|u_n-\widetilde u_m\|_{L^2(\Om_n)}
\le CR_4 2^{-4m},
\end{equation}
and the same holds for $B$.

The uniform $H^4$ boundedness of $\mathcal J_m$ gives
$\|u_n-\widetilde u_m\|_{H^4(\Om_n)}\le C(R_4)$ independently of $m$
for all sufficiently large $n$. Interpolating
\eqref{eq:two-state-L2-full} on the uniformly Lipschitz domains yields
\begin{equation*}
\limsup_{n\to\infty}\sup_t
\|u_n-\widetilde u_m\|_{H^s(\Om_n)}
\le C(R_4)2^{-(4-s)m},
\end{equation*}
and likewise for $B$. At fixed $m$, the smooth comparison fields
vary continuously in time on local common neighborhoods of the domains.
Freezing them on a finite time cover and applying the cutoff convention
in Section~\ref{subsec:state-topology} gives fixed ambient comparison
fields on each interval. Letting $m\to\infty$ gives item~(iii) of
\cref{def:stdtop} uniformly in time. Together with
\eqref{eq:two-state-graph-Hs}, this proves \cref{prop:cont}\textup{(i)}.

\par\smallskip
\noindent\textit{Step 2. Choose one frequency cutoff for a convergent sequence of
data.}\par\nobreak

Let $U_0^{(n)}\to U_0$ in the moving domain $H^s$ topology, and write
$\Phi(U_0^{(n)})=(u_n,B_n,\Ga_{\eta_n})$ and
$\Phi(U_0)=(u,B,\Ga_\eta)$, with fluid domains $\Om_n$ and $\Om$.
Let $c^{(n)}$ and $c$ be the envelopes of the respective initial states.
By \cref{lem:extmetric-top} and \eqref{eq:envelope-continuity},
\begin{equation*}
c^{(n)}\longrightarrow c\quad\hbox{in }\ell^2.
\end{equation*}
Given $\delta>0$, choose integers $n_0$ and $J$, with $J$ above every
uniform regularization threshold, so that
\begin{equation}\label{eq:common-tail-delta}
\|c_{\ge J-C}\|_{\ell^2}
+\sup_{n\ge n_0}\|c^{(n)}_{\ge J-C}\|_{\ell^2}<\delta.
\end{equation}
For this fixed $J$, \cref{lem:compatible-regularization} gives
$S_JU_0^{(n)}\to S_JU_0$ in every finite Sobolev norm.  Denote the corresponding smooth
trajectories by
$U_{n,J}=(u_{n,J},B_{n,J},\Ga_{\eta_{n,J}})$ and
$U_J=(u_J,B_J,\Ga_{\eta_J})$, with domains $\Om_{n,J}$ and $\Om_J$.

To verify the weak band condition at fixed $J$, first note that
\cref{prop:E4} gives, for large $n$, common bounds
\begin{equation*}
\|U_{n,J}\|_{\X^s}+\|U_J\|_{\X^s}\le R_s,
\quad
\|U_{n,J}\|_{\X^4}+\|U_J\|_{\X^4}\le R_4,
\quad a_{n,J},a_J\ge c_*>0.
\end{equation*}
Here $a_{n,J}$ and $a_J$ are the Taylor coefficients of the two
regularized solutions, and $c_*>0$ is independent of $n$. Let $h=h(R_s,c_*)$
be supplied by \cref{prop:Dprop}\textup{(i)}. Smooth
convergence of the regularized initial data makes the initial graph
difference smaller than $h/4$.  Suppose, for contradiction, that the
$C^0$ graph difference reaches $h$, and let $\tau_n$ be the first such
time. On $[0,\tau_n)$, \cite[Theorem~4.1]{IPTT-MHD} applies, so
\begin{equation}\label{eq:fixedJ-D-small}
\sup_{0\le t<\tau_n}
\D(U_{n,J}(t),U_J(t))
\le C_J\D(U_{n,J}(0),U_J(0))=o_n(1).
\end{equation}
By \cref{prop:Dprop}\textup{(i)}, the graph difference is $o_n(1)$
in $L^2$. Interpolating with the common $H^4$ bound gives $o_n(1)$
in $H^q$ for any $1<q<4$,
hence in $C^0$ on the two-dimensional surface. For large $n$ the
difference remains below $h/2$, contradicting the definition of $\tau_n$.
Thus the weak band and \eqref{eq:fixedJ-D-small} hold on $[0,T]$.
By \cref{prop:cont}\textup{(i)},
\begin{equation}\label{eq:fixedJ-strong-conv}
U_{n,J}\longrightarrow U_J
\quad\hbox{in }C([0,T];\X^s).
\end{equation}

\par\smallskip
\noindent\textit{Step 3. Compare the full rough solutions through the same fixed
smooth fields.}\par\nobreak

Set $W_n^\pm=u_n\pm B_n$, $W^\pm=u\pm B$,
$W_{n,J}^\pm=u_{n,J}\pm B_{n,J}$, and $W_J^\pm=u_J\pm B_J$.
Define
\begin{equation*}
F_{n,J}^\pm=\mathcal J_JW_{n,J}^\pm,
\qquad
F_J^\pm=\mathcal J_JW_J^\pm.
\end{equation*}
Equations \eqref{eq:graph-tail} and \eqref{eq:finite-comparison-tail} give
\begin{align}
\sup_t\|\eta_n-\eta_{n,J}\|_{H^s}
+\sup_t\|W_n^\pm-F_{n,J}^\pm\|_{H^s(\Om_n(t))}
&\le C\|c^{(n)}_{\ge J-C}\|_{\ell^2},\\
\sup_t\|\eta-\eta_J\|_{H^s}
+\sup_t\|W^\pm-F_J^\pm\|_{H^s(\Om(t))}
&\le C\|c_{\ge J-C}\|_{\ell^2}.
\label{eq:cont-tail-limit}
\end{align}
For fixed $J$, \eqref{eq:fixedJ-strong-conv} gives convergence of the
cutoff graphs and fields.  By the geometric tail
\eqref{eq:Linfgeomtail}, the full domains lie within $o(2^{-J})$ of their
$J$th cutoff domains, uniformly in time and uniformly for large $n$.
Consequently, with $J$ fixed above the sampling threshold and $n$
sufficiently large, every kernel point used by either fixed operator
$\mathcal J_J$ from a point of $\Om_n(t)\cup\Om(t)$ lies in
\begin{equation*}
I_{n,J}(t)=\Om_{n,J}(t)\cap\Om_J(t).
\end{equation*}
Differentiating the fixed scale kernel up to order four and applying
Schur's test gives
\begin{equation}\label{eq:fixed-comparison-convergence}
\sup_t\Bigl(
\|F_{n,J}^\pm-F_J^\pm\|_{H^4(\Om_n(t))}
+\|F_{n,J}^\pm-F_J^\pm\|_{H^4(\Om(t))}\Bigr)\le C_J\sup_t
\|W_{n,J}^\pm-W_J^\pm\|_{L^2(I_{n,J}(t))}
\longrightarrow0.
\end{equation}
Use the single smooth field $F_J^\pm$, independent of $n$, in
\cref{def:stdtop}. The triangle inequality and
\eqref{eq:common-tail-delta}, \eqref{eq:cont-tail-limit}, and
\eqref{eq:fixed-comparison-convergence} bound the comparison errors on
both families of domains by $O(\delta)+o_n(1)$.
At fixed $J$, time continuity of $F_J^\pm$ allows these fields to be
frozen on a finite time cover, as in Step~1. For each $\delta$, first
fix $J=J(\delta)$, then let $n\to\infty$, and finally let
$\delta\downarrow0$. This proves
\begin{equation*}
\Phi(U_0^{(n)})\longrightarrow\Phi(U_0)
\quad\hbox{in }C([0,T];\X^s),
\end{equation*}
and hence \cref{prop:cont}\textup{(ii)}.

\end{proof}

\begin{proof}[Proof of \cref{thm:main}]
\Cref{prop:E4} gives a common lifespan for the regularizations, and
\cref{prop:persistence} gives a unique $C([0,T];\X^s)$ solution with
a positive Taylor coefficient. Continuous dependence follows from
\cref{prop:cont}\textup{(ii)}.
\end{proof}
\section{Nonuniform dependence of the solution map}
\label{sec:nonuniform}

In this section, we prove \cref{thm:nonuniform}. We first construct a
continuous local backward flow. We then choose interior magnetic perturbations
that preserve the constraints and Taylor sign and separate under small
translations. Pulling these perturbations back to the initial time and
applying Galilean invariance gives nonuniform dependence on every
nonempty open subset of the common existence neighborhood.

\subsection{Backward reduction and terminal packets}

\begin{lemma}\label{lem:symmetries-backward}
\textup{(i)} Let $(u,B,\Om,p)$ solve \eqref{eq:mom}--\eqref{eq:bc} with $-\partial_Np>0$ on the free boundary.  Then the equations and these conditions are preserved by
\begin{equation*}
(u^{\rm rev},B^{\rm rev},p^{\rm rev})(t,x)
:=\bigl(-u(-t,x),B(-t,x),p(-t,x)\bigr),
\qquad
\Om_t^{\rm rev}:=\Om_{-t},
\end{equation*}
and, for each constant $c\in\R^3$, by
\begin{equation*}
\Om_t\longmapsto\Om_t+ct,
\qquad
(u,B,p)(t,x)\longmapsto
\bigl(u(t,x-ct)+c,B(t,x-ct),p(t,x-ct)\bigr).
\end{equation*}
With
\begin{equation*}
\mathfrak R(u,B,\Ga):=(-u,B,\Ga),
\end{equation*}
$\mathfrak R$ is an involution on states, and a Galilean
translation carries the normal pressure coefficient to its rigid translate.

\textup{(ii)} Let $U_0$ be covered by \cref{thm:main}. There are a bounded
open neighborhood $\mathscr U$ of $U_0$ and $T>0$ such that all trajectories
from $\mathscr U$ remain in one bounded graph subchart on $[0,T]$.
They stay separated from the collar edge, with uniform positive
transversality and injectivity bounds and a common positive Taylor lower bound.

For every $U\in\mathscr U$ and $0<\tau\le T$, the state $V=\Phi(\tau)U$
has a neighborhood $\mathscr V$ and a continuous backward map
\begin{equation*}
\mathcal B_\tau:\mathscr V\longrightarrow\mathscr U
\end{equation*}
such that
\begin{equation}
\mathcal B_\tau(V)=U,
\qquad
\Phi(\tau)\mathcal B_\tau(W)=W
\quad (W\in\mathscr V).
\label{eq:uniform-backward-inverse}
\end{equation}
\end{lemma}

\begin{proof}
The proof is divided into four steps.

\par\nobreak\smallskip
\noindent\emph{Step 1. Galilean invariance.}\par\nobreak

For the Galilean transform $(u^c,B^c,p^c,\Om^c)$ in part~(i),
set $y=x-ct$.
Then $N^c(t,x)=N(t,y)$. For every smooth scalar field $f$, set
$f^c(t,x)=f(t,x-ct)$. The chain rule gives
\begin{equation*}
(\partial_t+u^c\cdot\nabla_x)f^c(t,x)
=(\partial_t+u\cdot\nabla)f(t,x-ct),
\qquad
B^c(t,x)\cdot\nabla_x f^c(t,x)
=(B\cdot\nabla f)(t,x-ct).
\end{equation*}
Thus both material equations, divergence, magnetic tangency, and the
zero pressure condition are preserved. The boundary identities are
\begin{equation*}
V_{\Ga^c}=V_\Ga+c\cdot N=u^c\cdot N,
\qquad \partial_{N^c}p^c(t,x)=\partial_Np(t,y).
\end{equation*}
Thus the kinematic condition and the translated Taylor sign also hold.

\par\smallskip
\noindent\emph{Step 2. Time reversal.}\par\nobreak

With $D_t^{\rm rev}=\partial_t+u^{\rm rev}\cdot\nabla$,
\begin{equation*}
D_t^{\rm rev}u^{\rm rev}(t,x)=(D_tu)(-t,x),
\qquad D_t^{\rm rev}B^{\rm rev}(t,x)=-(D_tB)(-t,x).
\end{equation*}
Thus the momentum equation is unchanged and both sides of the induction
equation change sign.  The reversed boundary velocity is
$-u(-t)\cdot N=u^{\rm rev}\cdot N$; divergence, tangency, pressure,
and the Taylor sign are preserved.  This proves part~(i).
The map $\mathfrak R$ is an involution and, by linearity of the extensions,
\begin{equation*}
d(\mathfrak RU,\mathfrak R\widetilde U)=d(U,\widetilde U).
\end{equation*}

\par\smallskip
\noindent\emph{Step 3. A common neighborhood on which the backward map is defined.}\par\nobreak

The pressure equation and its zero boundary condition are unchanged
under $u_0\mapsto-u_0$. Thus $\mathfrak RU_0$ has the same pressure
and Taylor coefficient as $U_0$ and is admissible.
We apply \cref{thm:main} at $U_0$ and at $\mathfrak RU_0$. Let
$\mathscr W$ be a forward neighborhood of $U_0$ and $\mathscr N$ a
forward neighborhood of $\mathfrak RU_0$, with a common existence
time after decreasing both lifespans to $T>0$.  Denote the forward
solution map on $\mathscr N$ by $\Phi_-$; time reversal of this
forward evolution will give the backward map.
Write $\eta_0$ for the graph of the boundary of $U_0$.  Choose
$\varepsilon_g>0$ so that every graph satisfying
\begin{equation*}
\|\eta-\eta_0\|_{C^1(\Ga_*)}\le3\varepsilon_g
\end{equation*}
lies strictly inside the reference collar and has common positive
transversality and injectivity bounds.  Such a choice is possible because
the initial surface is compact and embedded and these properties are open
in $C^1$.  Let $C_s$ be the norm of the embedding
$H^s(\Ga_*)\hookrightarrow C^1(\Ga_*)$.  Choose $r>0$ so that the
closed $3r$ metric ball about $\mathfrak RU_0$, within the fixed collar,
lies in $\mathscr N$ and $2C_sr<\varepsilon_g$.  By continuity of the
trajectory through $U_0$, decrease $T$ so that
\begin{equation*}
d(\Phi(t)U_0,U_0)<r,
\qquad 0\le t\le T.
\end{equation*}
By continuous dependence of the forward map in
$C([0,T];\X^s)$, shrink $\mathscr W$ to a bounded
neighborhood $\mathscr U$ of $U_0$ satisfying
\begin{equation*}
\sup_{0\le t\le T}d(\Phi(t)W,\Phi(t)U_0)<r,
\qquad W\in\mathscr U.
\end{equation*}
The triangle inequality and the isometry of $\mathfrak R$ give
\begin{equation*}
d(\mathfrak R\Phi(t)W,\mathfrak RU_0)
=d(\Phi(t)W,U_0)<2r,
\qquad W\in\mathscr U,
\end{equation*}
so $\mathfrak R\Phi(t)W\in\mathscr N$.  Writing $\eta_Z$ for the
boundary graph of a state $Z$, the same estimate gives
\begin{equation*}
\sup_{0\le t\le T}
\|\eta_{\Phi(t)W}-\eta_0\|_{C^1}
\le C_s\sup_{0\le t\le T}d(\Phi(t)W,U_0)
<2C_sr<\varepsilon_g.
\end{equation*}
These estimates keep the surfaces inside the collar with uniform
positive transversality and injectivity bounds.  They also give uniform
state $H^s$ bounds, hence uniform graph $C^{1,\alpha_s}$ bounds.
The common positive Taylor bound follows from \cref{thm:main}.

Fix $U\in\mathscr U$, $0<\tau\le T$, and write
$V=\Phi(\tau)U$.  Since $\mathfrak RV\in\mathscr N$, continuity of the
reversed forward flow shows that, on a neighborhood $\mathscr V_0$ of
$V$ with $\mathfrak R\mathscr V_0\subset\mathscr N$, the map
\begin{equation*}
\mathcal B_\tau(W)
:=\mathfrak R\bigl(\Phi_-(\tau)(\mathfrak RW)\bigr)
\end{equation*}
is well defined and continuous.

\par\smallskip
\noindent\emph{Step 4. The inverse identities.}\par\nobreak

Define
\begin{equation*}
Y(t)=\mathfrak R\Phi(\tau-t)U,
\qquad 0\le t\le\tau.
\end{equation*}
By Step~2, $Y$ and $t\mapsto\Phi_-(t)(\mathfrak RV)$ are solutions in the
uniqueness class with the same initial state $\mathfrak RV$.  Hence
\begin{equation*}
Y(t)=\Phi_-(t)(\mathfrak RV),
\qquad 0\le t\le\tau.
\end{equation*}
At $t=\tau$ this gives
$\mathcal B_\tau(V)=U$.  Since $\mathscr U$ is open and
$\mathcal B_\tau$ is continuous, shrink $\mathscr V_0$ to a neighborhood
$\mathscr V$ of $V$ such that
$\mathcal B_\tau(\mathscr V)\subset\mathscr U$.

For $W\in\mathscr V$ define
\begin{equation*}
Z(t)=\mathfrak R\Phi_-(\tau-t)(\mathfrak RW),
\qquad 0\le t\le\tau.
\end{equation*}
This solves the original equations with $Z(0)=\mathcal B_\tau(W)$ and
$Z(\tau)=W$.  Since $\mathcal B_\tau(W)\in\mathscr U$, its forward
solution exists on $[0,\tau]$ and equals $Z$ by uniqueness.  This proves
\eqref{eq:uniform-backward-inverse}, with $\mathscr U$ and $T$
independent of $U\in\mathscr U$.
\end{proof}

Fix a state
\begin{equation*}
\overline V=(\overline u,\overline B,\overline\Ga)
\end{equation*}
in a bounded $\X^s$ subchart, and let $\Om$ and $\overline p$ be its
fluid domain and associated pressure.  Assume $-\partial_N\overline p>0$.
Choose $x_*\in\Om$ and $\ell>0$ such that
$Q_{4\ell}\Subset\Om$, where $Q_r:=x_*+(-r,r)^3$ for $r>0$.  Let
$\chi\in C_c^\infty(Q_{4\ell})$ satisfy $\chi=1$ on $Q_{2\ell}$.  For a parameter
$0<\delta\ll1$ and integers $k\gg1$ set
\begin{equation*}
b_k:=\delta k^{-(s+1)}
\curl\bigl(0,0,\chi(x)\sin(kx_1)\bigr).
\end{equation*}
Then $b_k$ is smooth, divergence-free and compactly supported in the interior
of $\Om$, so $b_k\cdot N=0$ on
$\overline\Ga$.

\begin{lemma}\label{lem:terminal-adm}
\textup{(i)} For $0\le r\le s$,
\begin{equation}
\|b_k\|_{H^r(\Om)}
\le C_{r,s}\delta k^{r-s}.
\label{eq:bNs}
\end{equation}
With
\begin{equation*}
h_k:=\frac{\pi}{k}e_1,
\end{equation*}
there is a constant $C\ge1$,
independent of $k$ and $\delta$, such that
\begin{equation}
\|b_k(\cdot-h_k)-b_k\|_{H^s(Q_\ell)}\ge C^{-1}\delta
\label{eq:bNsep}
\end{equation}
for all sufficiently large $k$.

\textup{(ii)} Let
\begin{equation*}
V_k:=(\overline u,\overline B+b_k,\overline\Ga),
\end{equation*}
and let $p_k$ be its associated pressure. Suppose
$-\partial_N\overline p\ge 2c_*>0$. Given an $H^s$ neighborhood of
$\overline V$, choose $\delta>0$ sufficiently small. Then, for all
sufficiently large $k$, the states $V_k$ belong to that neighborhood and satisfy
\begin{equation}
-\partial_N p_k\ge c_*>0.
\label{eq:VN-pressure}
\end{equation}
\end{lemma}

\begin{proof}
The proof is divided into three steps.

\par\nobreak\smallskip
\noindent\emph{Step 1. Explicit form and Sobolev size of the packet.}\par\nobreak

Expanding the curl gives
\begin{equation}\label{eq:bk-explicit}
b_k=-\delta k^{-s}\chi(x)\cos(kx_1)e_2
+\delta k^{-(s+1)}r_k(x),
\end{equation}
where $r_k=(\partial_2\chi,-\partial_1\chi,0)\sin(kx_1)$.
For $a\in C_c^\infty(\R^3)$ and $0\le r\le s$, Fourier translation gives
\begin{equation*}
\|a(x)e^{ikx_1}\|_{H^r(\R^3)}\le C_{a,r}k^r,
\qquad k\ge1.
\end{equation*}
Applying this to \eqref{eq:bk-explicit} proves \eqref{eq:bNs}.

\par\smallskip
\noindent\emph{Step 2. Separation under a translation by half a wavelength.}\par\nobreak

For large $k$, both $x$ and
$x-h_k$ lie in $Q_{2\ell}$ whenever $x\in Q_\ell$.  Since $\chi=1$
there and $kh_k\cdot e_1=\pi$,
\begin{equation*}
b_k(x-h_k)-b_k(x)=2\delta k^{-s}\cos(kx_1)e_2,
\qquad x\in Q_\ell.
\end{equation*}
Let $0\ne\zeta\in C_c^\infty(Q_\ell)$.  Multiplication by $\zeta$,
followed by extension by zero, is bounded
$H^s(Q_\ell)\to H^s(\R^3)$. Hence
\begin{equation}\label{eq:local-lower-via-zeta}
\|b_k(\cdot-h_k)-b_k\|_{H^s(Q_\ell)}
\ge C^{-1}\delta k^{-s}
\|\zeta\cos(kx_1)\|_{H^s(\R^3)}.
\end{equation}
The Fourier transform is
\begin{equation*}
\widehat{\zeta\cos(kx_1)}(\xi)
=\frac12\widehat\zeta(\xi-ke_1)
+\frac12\widehat\zeta(\xi+ke_1).
\end{equation*}
Dominated convergence for the diagonal terms and rapid decay for the
cross term give
\begin{equation}\label{eq:osc-limit}
\lim_{k\to\infty}k^{-2s}
\|\zeta\cos(kx_1)\|_{H^s(\R^3)}^2
=\frac12\|\zeta\|_2^2.
\end{equation}
Combining \eqref{eq:local-lower-via-zeta} and
\eqref{eq:osc-limit} proves \eqref{eq:bNsep}.

\par\smallskip
\noindent\emph{Step 3. Admissibility and persistence of the Taylor sign.}\par\nobreak

The interior curl preserves the constraints, and \eqref{eq:bNs} gives
\begin{equation*}
d(V_k,\overline V)
\le C\|b_k\|_{H^s(\Om)}
\le C\delta.
\end{equation*}
Choose $\delta>0$ so small that this $C\delta$ ball is contained in the
prescribed $H^s$ neighborhood.

For $5/2<s'<s$, \eqref{eq:bNs} also gives
\begin{equation*}
\|b_k\|_{H^{s'}(\Om)}\le C\delta k^{s'-s}\longrightarrow0.
\end{equation*}
Together with the uniform $H^s$ bound,
\cref{lem:pressure-estimates}\textup{(ii)} yields
\begin{equation*}
\|\partial_Np_k-\partial_N\overline p\|_{L^\infty(\overline\Ga)}
\longrightarrow0.
\end{equation*}
Since $-\partial_N\overline p\ge2c_*$, this proves
\eqref{eq:VN-pressure} and admissibility for all large $k$.
\end{proof}

\begin{proof}[Proof of \cref{thm:nonuniform}]
The proof is divided into three steps.

\par\nobreak\smallskip
\noindent\textit{Step 1. Pull terminal packets back into an arbitrary initial open
set.}\par\nobreak

Choose the bounded neighborhood $\mathscr U$ and common time $T$ from
\cref{lem:symmetries-backward}\textup{(ii)}. Fix $0<\tau\le T$ and a
nonempty open set $\mathscr O\subset\mathscr U$, and choose
$\overline U\in\mathscr O$. Set
\begin{equation*}
\overline V=\Phi(\tau)\overline U
=(\overline u,\overline B,\overline\Ga),
\end{equation*}
and denote its fluid domain and pressure by $\Om$ and $\overline p$.  Let $\mathcal B_\tau$ be the local backward inverse near $\overline V$.
Write $B_d(U,r)=\{W:d(W,U)<r\}$ for the open metric ball of radius
$r$ about a state $U$.  Because $d$ metrizes the state topology and
$\mathcal B_\tau$ is continuous, choose $\rho>0$ with
\begin{equation*}
B_d(\overline U,2\rho)\subset\mathscr O
\end{equation*}
and then choose a terminal neighborhood $\mathscr W\subset\mathscr V$
of $\overline V$ within the domain of $\mathcal B_\tau$ such that
\begin{equation*}
d(\mathcal B_\tau(V),\overline U)<\rho,
\qquad V\in\mathscr W.
\end{equation*}
Let
\begin{equation*}
c_*:=\frac12\inf_{\overline\Ga}(-\partial_N\overline p)>0.
\end{equation*}
By \cref{lem:terminal-adm}\textup{(ii)}, choose once and for all a sufficiently small
packet amplitude $\delta>0$ such that, for all large $k$,
\begin{equation*}
V_k=(\overline u,\overline B+b_k,\overline\Ga)
\end{equation*}
belongs to $\mathscr W$ and has Taylor coefficient at least $c_*$.  Define
\begin{equation*}
U_k=\mathcal B_\tau(V_k).
\end{equation*}
By \eqref{eq:uniform-backward-inverse} and the choice of $\mathscr W$,
\begin{equation}\label{eq:NU-pullback}
\Phi(\tau)U_k=V_k,
\qquad
d(U_k,\overline U)<\rho.
\end{equation}

\par\smallskip
\noindent\textit{Step 2. Apply a vanishing Galilean boost.}\par\nobreak

Let
\begin{equation*}
h_k=\frac\pi k e_1,
\qquad
v_k=\frac{h_k}{\tau}.
\end{equation*}
Apply the Galilean symmetry of \cref{lem:symmetries-backward}\textup{(i)} to the exact
trajectory $\Phi(t)U_k=(u_k(t),B_k(t),\Ga_k(t))$, using the constant
velocity $v_k$.  Let $\Om_k(t)$ be the domain enclosed by $\Ga_k(t)$
and denote the resulting initial state by $\widetilde U_k$.  At $t=0$ the
spatial translation is zero, so
\begin{equation*}
\widetilde U_k=(u_k(0)+v_k,B_k(0),\Ga_k(0)).
\end{equation*}
Adding a constant velocity preserves the constraints and pressure.
Linearity and uniform boundedness of the extensions give
\begin{equation}\label{eq:initialboostdist}
d(U_k,\widetilde U_k)
=\|E_{\Om_k(0)}v_k\|_{H^s(\R^3)}
\le C|v_k|
=\frac{C\pi}{\tau k}\longrightarrow0.
\end{equation}
By \eqref{eq:NU-pullback}, \eqref{eq:initialboostdist}, and the triangle
inequality, both $U_k$ and $\widetilde U_k$ lie in $\mathscr O$ for all
sufficiently large $k$.

We verify that the transformed solutions remain in the same geometric
chart throughout $[0,\tau]$.  Let $\eta_0$ and $\varepsilon_g$ be the
reference graph and radius chosen in Step~3 of the proof of
\cref{lem:symmetries-backward}. The graph functions
$\eta_k(t)$ of the unshifted surfaces $\Ga_k(t)$ satisfy
\begin{equation*}
\|\eta_k(t)-\eta_0\|_{C^1}<\varepsilon_g,
\qquad
\|\eta_k(t)\|_{H^s}+\|\eta_k(t)\|_{C^{1,\alpha_s}}\le R,
\end{equation*}
with $R$ independent of $k$ and $t$.  Every graph in the closed
$C^1$ ball of radius $3\varepsilon_g$ about $\eta_0$ remains inside the
collar with the same uniform positive transversality and injectivity bounds.

Use the collar coordinates $(\pi_*,r_*)$ from
\eqref{eq:collar-coordinates}. For a graph $\eta$ in this family and a
small constant vector $a$, set
\begin{equation*}
G_\eta(y)=y+\eta(y)X_*(y),
\qquad
q_{\eta,a}(y)=\pi_*(G_\eta(y)+a).
\end{equation*}
At $a=0$, the map $q_{\eta,a}$ is the identity on $\Ga_*$.  The chain rule and the bounded
derivatives of $\pi_*$ on the closed smaller collar give
\begin{equation*}
\|q_{\eta,a}-\mathrm{Id}\|_{C^1}\le C_R|a|.
\end{equation*}
All map norms in this step are taken in a fixed finite atlas.
Consequently, for $|a|\le a_0$ with $a_0>0$ independent of the graph,
$q_{\eta,a}$ is a diffeomorphism of $\Ga_*$ with uniform bounds on its
derivative and inverse derivative.  The translated surface has graph
\begin{equation*}
\eta^a=
\bigl(r_*\circ(G_\eta+a)\bigr)\circ q_{\eta,a}^{-1}.
\end{equation*}
The composition and inversion estimates in $H^s$ on the two-dimensional
reference surface apply because $s>2$.  Applied to these formulas, they
give $\|\eta^a\|_{H^s}\le C_R$.  To estimate the change in $C^1$,
observe that
\begin{equation*}
\|r_*\circ(G_\eta+a)-\eta\|_{C^1}\le C_R|a|,
\qquad
\|q_{\eta,a}^{-1}-\mathrm{Id}\|_{C^1}\le C_R|a|.
\end{equation*}
The H\"older bound for $D\eta$ with exponent $\alpha_s$ therefore gives
\begin{equation*}
\|\eta^a-\eta\|_{C^1}\le C_R|a|^{\alpha_s}
\qquad (|a|\le a_0\le1).
\end{equation*}
Decrease $a_0$ so that $C_Ra_0^{\alpha_s}<\varepsilon_g$.
Since $|tv_k|\le\pi/k$ on $[0,\tau]$, all sufficiently large $k$
give translated graphs within $2\varepsilon_g$ of $\eta_0$ in $C^1$,
with a uniform $H^s$ bound.  Thus they remain in a common bounded
subchart for the entire interval.

Translation preserves the bulk Sobolev norms and Taylor lower bound;
the added velocity has norm at most $C|v_k|$.  For fixed $k$, continuity
of $H^s$ composition and inversion gives time continuity of the graph.
Strong continuity of ambient translations and domain continuity of the
extensions give time continuity of the fields; a fixed cutoff treats
the constant velocity.  The transformed solution thus belongs to
$C([0,\tau];\X^s)$ and equals $\Phi(t)\widetilde U_k$ by
\cref{prop:persistence}\textup{(ii)}.  Since $\tau v_k=h_k$,
\begin{align}
\Phi(\tau)U_k
&= (\overline u,\overline B+b_k,\overline\Ga),\label{eq:terminal-unboosted}\\
\Phi(\tau)\widetilde U_k
&=\bigl(\overline u(\cdot-h_k)+v_k,
(\overline B+b_k)(\cdot-h_k),
\overline\Ga+h_k\bigr).
\label{eq:terminal-boosted}
\end{align}

\par\smallskip
\noindent\textit{Step 3. Obtain a fixed positive terminal separation.}\par\nobreak

Let $Q_\ell\Subset\Om$ be the interior box from
\cref{lem:terminal-adm}\textup{(i)}.  Since $h_k\to0$, for all large $k$,
\begin{equation*}
Q_\ell\Subset\Om\cap(\Om+h_k).
\end{equation*}
On this box the ambient extensions appearing in $d$ agree with the
physical magnetic fields.  Since restriction
$H^s(\R^3)\to H^s(Q_\ell)$ is bounded, the magnetic part of the metric and
\eqref{eq:terminal-unboosted}--\eqref{eq:terminal-boosted} imply
\begin{align}\label{eq:terminal-lower-magnetic}
d\bigl(\Phi(\tau)U_k,\Phi(\tau)\widetilde U_k\bigr)
&\ge C^{-1}
\|\overline B(\cdot-h_k)-\overline B
+b_k(\cdot-h_k)-b_k\|_{H^s(Q_\ell)}.
\end{align}
Using a fixed $H^s$ extension of $\overline B$, strong continuity of
translations gives
\begin{equation*}
\|\overline B(\cdot-h_k)-\overline B\|_{H^s(Q_\ell)}\longrightarrow0.
\end{equation*}
The reverse triangle inequality and \eqref{eq:bNsep} therefore imply
from \eqref{eq:terminal-lower-magnetic} that
\begin{equation*}
\liminf_{k\to\infty}
d\bigl(\Phi(\tau)U_k,\Phi(\tau)\widetilde U_k\bigr)
\ge c\delta>0.
\end{equation*}
At the same time \eqref{eq:initialboostdist} shows
$d(U_k,\widetilde U_k)\to0$.  The base state, nonzero observation time,
and open neighborhood were arbitrary.  Thus the map at time $\tau$ is not
uniformly continuous on any nonempty open subset of $\mathscr U$, proving
\cref{thm:nonuniform}.
\end{proof}

\subsection*{Acknowledgment}
S. Yang was supported by the National Natural Science Foundation of China under Grant No. 12601417, the Beijing Natural Science Foundation under Grant
No. 1264051, the Postdoctoral Fellowship Program (Grade C) of China
Postdoctoral Science Foundation under Grant No. GZC20261671, the Beijing
Postdoctoral Research Activity Funding Program under Grant No. 2026-ZZ-28,
and the Funding Program for Newly Recruited Young Teachers of Beijing
University of Technology. C. Hao was partially supported by the National Natural Science Foundation of China under Grant No. 12171460, the CAS Project for Young Scientists in Basic Research under Grant No. YSBR-031 and the National Key R\&D Program of China under Grant No. 2021YFA1000800.

\medskip

\noindent\textbf{Data Availability} Data sharing is not applicable to this article as no datasets were generated or analyzed during the current study.

\medskip

\noindent\textbf{Declarations}
\medskip

\noindent\textbf{Conflict of interest} The authors declare that there is no conflict of interest. 

\end{document}